\documentclass[microtype]{gtpart}

\usepackage{amsxtra}
\usepackage{mathtools}
\usepackage{extpfeil}
\usepackage{enumitem}
\usepackage{graphicx}
\usepackage{proof}
\usepackage{ifnextok}
\usepackage{xstring}

\usepackage[all]{xy}
\SelectTips{cm}{}
\SilentMatrices
\ifx\pdfoutput\undefined
  \xyoption{dvips}
\fi

\usepackage[silent]{fixme}
\usepackage{slashed}

\usepackage{xr}

\fxsetup{
  status=final
}

\title{Kan extensions and the calculus of modules for $\infty$-categories}

\author[Riehl]{Emily Riehl}
\givenname{Emily}
\surname{Riehl}
\address{
Department of Mathematics \\
Johns Hopkins University \\
Baltimore, MD 21218\\
USA}
\email{eriehl@math.jhu.edu}
\urladdr{http://www.math.jhu.edu/~eriehl}

\author[Verity]{Dominic Verity}
\givenname{Dominic}
\surname{Verity}
\address{
  Centre of Australian Category Theory \\
  Macquarie University \\
  NSW 2109 \\
  Australia
}
\email{dominic.verity@mq.edu.au}
\urladdr{}


\subject{primary}{MSC2010}{18G55}
\subject{primary}{MSC2010}{55U35}
\subject{primary}{MSC2010}{55U40}
\subject{secondary}{MSC2010}{18A05}
\subject{secondary}{MSC2010}{18D20}
\subject{secondary}{MSC2010}{18G30}
\subject{secondary}{MSC2010}{55U10}

\arxivreference{arXiv:1507.01460}
\arxivpassword{}

\volumenumber{}
\issuenumber{}
\publicationyear{}
\papernumber{}
\startpage{}
\endpage{}
\doi{}
\MR{}
\Zbl{}
\received{}
\revised{}
\accepted{}
\published{}
\publishedonline{}
\proposed{}
\seconded{}
\corresponding{}
\editor{}
\version{}

\let\co=\relax

\newcommand{\ifundef}[1]{\expandafter\ifx\csname#1\endcsname\relax}

\DeclareMathAlphabet{\mathbbe}{U}{bbold}{m}{n}

\makeatletter

\def\re@DeclareMathSymbol#1#2#3#4{%
    \let#1=\undefined
    \DeclareMathSymbol{#1}{#2}{#3}{#4}}

\ifundef{righthalfcup}
  \DeclareFontFamily{U}{MnSymbolC}{}
  \DeclareSymbolFont{mnSyC}{U}{MnSymbolC}{m}{n}
  \SetSymbolFont{mnSyC}{bold}{U}{MnSymbolC}{b}{n}
  \DeclareFontShape{U}{MnSymbolC}{m}{n}{
      <-6>  MnSymbolC5
     <6-7>  MnSymbolC6
     <7-8>  MnSymbolC7
     <8-9>  MnSymbolC8
     <9-10> MnSymbolC9
    <10-12> MnSymbolC10
    <12->   MnSymbolC12}{}
  \DeclareFontShape{U}{MnSymbolC}{b}{n}{
      <-6>  MnSymbolC-Bold5
     <6-7>  MnSymbolC-Bold6
     <7-8>  MnSymbolC-Bold7
     <8-9>  MnSymbolC-Bold8
     <9-10> MnSymbolC-Bold9
    <10-12> MnSymbolC-Bold10
    <12->   MnSymbolC-Bold12}{}
  
  \re@DeclareMathSymbol{\righthalfcup}{\mathord}{mnSyC}{184}
  \re@DeclareMathSymbol{\lefthalfcap}{\mathord}{mnSyC}{185}
\fi

\makeatother

\newcommand{\mlaux}[3]{\setbox0=\hbox{$\mathsurround=0pt #2{#3}$}%
  \dimen0=\dp0\advance\dimen0 by \ht0\lower#1\dimen0\box0}

\newcommand{\makellapm}[2]{\hbox to 0pt{\hss$\mathsurround=0pt #1{#2}$}}

\newcommand{\makerlapm}[2]{\hbox to 0pt{$\mathsurround=0pt #1{#2}$\hss}}

\newcommand{\makelapm}[2]{\hbox to 0pt{\hss$\mathsurround=0pt #1{#2}$\hss}}

\newcommand{\makeushort}[3]{%
	\setbox0=\hbox{$\mathsurround=0pt #2{#3}$}%
	\hbox to 1\wd0{\hss\underbar{\hbox to #1\wd0{\hss\box0\hss}}\hss}}

\def\makebigger#1#2#3{\scalebox{#1}{$\mathsurround=0pt #2{#3}$}}
\def\bigger#1#2{{\relax\mathpalette{\makebigger{#1}}{#2}}}

\def\scaleuphalf{1.0954}
\def\scaleupone{1.2}

\newcommand{\op}{^{\mathord{\text{\rm op}}}}
\newcommand{\co}{^{\mathord{\text{\rm co}}}}
\newcommand{\coop}{^{\mathord{\text{\rm coop}}}}

\renewcommand{\th}{^{\text{th}}}

\newcommand{\defeq}{\mathrel{:=}}

\makeatletter

\def\newmop{\@ifstar{\@newmop m}{\@newmop o}}
\def\@newmop#1{\@ifnextchar[{\@@newmop #1}{\@@@newmop #1}}
\def\@@newmop#1[#2]{\@declmathop #1#2}
\def\@@@newmop#1#2{\expandafter\@declmathop\expandafter #1\csname #2\endcsname{#2}}

\makeatother

\newdir{ |}{{}*!/-5pt/@{|}}

\newmop{im}
\newmop{coim}
\newmop{dom}
\newmop{cod}
\newmop{id}
\newmop{Map}

\newmop{obj}
\newmop{arr}
\newmop{sq}
\newmop{norm}

\newmop{el}

\newmop{ev}

\newmop{Ext}
\newmop{icon}
\newmop{pbk}
\newmop{icom}

\newmop{cell}
\newmop{cof}
\newmop{fib}

\newmop{diag}

\newmop*{colim}
\newmop*{holim}

\newmop[\lan]{lan}
\newmop[\ran]{ran}

\newcommand{\comma}{\mathbin{\downarrow}}

\newcommand{\itimes}{%
  \mathrel{\vbox{\offinterlineskip\ialign{%
    \hfil##\hfil\cr
    $\scriptscriptstyle\sim$\cr
    $\times$\cr
}}}}

\makeatletter
\newcommand{\rotatemath}[2]{\rotatebox[origin=c]{180}{$\m@th #1{#2}$}}
\makeatother

\newmop[\pr]{pr}
\newmop[\dm]{d} 

\newmop[\cpr]{in}

\newcommand{\pocorner}{\hbox to 8pt{{\vrule height8pt depth0pt width0.5pt}%
    \vbox to 8pt{{\hrule height0.5pt width7.5pt depth0pt}\vfill}}}

\newcommand{\pbcorner}{\vbox to 0pt{\kern 4pt\hbox to 0pt{\kern 4pt%
      \vbox{{\hrule height0.5pt width7.5pt depth0pt}}%
      {\vrule height8pt depth0pt width0.5pt}\hss}\vss}}
\newcommand{\pbexcursion}{\save[]+<5pt,-5pt>*{\pbcorner}\restore}
\newcommand{\pbdiamond}{\save[]+<0pt,-5pt>*{\rotatebox{-45}{$\pbcorner$}}\restore}

\newcommand{\pwr}{\pitchfork}

\newmop{cls}

\newcommand{\leib}[1]{\mathbin{\widehat{#1}}}

\newcommand{\pbtimes}[1]{\mathbin{\mathop{\times}_{#1}}}

\newcommand{\category}[1]{\underline{\smash[b]{\text{\rm{#1}}}}}

\newcommand{\tcat}{\lcat}

\newcommand{\lcat}{\mathcal}
\newcommand{\scat}{\mathbf}

\newcommand{\cattwo}{{\bigger{1.12}{\mathbbe{2}}}}
\newcommand{\catone}{{\bigger{1.16}{\mathbbe{1}}}}

\makeatletter

\def\Del@Sym{{\bigger\scaleuphalf{\mathbbe{\Delta}}}}

\def\del@fn{\futurelet\del@next}
\def\del@dn{\def\del@next}

\def\parsedel@{%
  \ifx +\del@next \del@dn+{\Del@Sym_{\mathord{+}}}%
  \else \del@dn {\del@fn\parsedel@@}%
  \fi\del@next}

\def\parsedel@@{%
  \ifx\space@\del@next \expandafter\del@dn\space{\del@fn\parsedel@@}%
  \else\ifx [\del@next \del@dn[{\del@fn\parsedel@@@}%
  \else\ifx _\del@next \del@dn{\Delta}%
  \else\ifx ^\del@next \del@dn{\Delta}%
  \else \del@dn{\Del@Sym}%
  \fi\fi\fi\fi\del@next}

\def\parsedel@@@{%
  \ifx\space@\del@next \expandafter\del@dn\space{\del@fn\parsedel@@@}%
  \else\ifx t\del@next \del@dn t{\Del@Sym_\infty\del@fn\parsedel@@@@}%
  \else\ifx b\del@next \del@dn b{\Del@Sym_{-\infty}\del@fn\parsedel@@@@}%
  \else \del@dn{\errmessage{unexpected modifier}}%
  \fi\fi\fi\del@next}

\def\parsedel@@@@{%
  \ifx\space@\del@next \expandafter\del@dn\space{\del@fn\parsedel@@@@}%
  \else\ifx ]\del@next \del@dn]{}%
  \else \del@dn{\errmessage{expecting close of option block}}%
  \fi\fi\del@next}

\def\Del{\del@fn\parsedel@}

\makeatother

\newcommand{\Horn}{\Lambda}

\newcommand{\Set}{\category{Set}}
\newcommand{\Cat}{\category{Cat}}

\newcommand{\sSet}{\category{sSet}}
\newcommand{\qCat}{\category{qCat}}
\newcommand{\Adj}{\category{Adj}}
\newcommand{\Mnd}{\category{Mnd}}

\newcommand{\isoSpan}[2]{\category{Span}_{\tcat{C}}(#1,#2)}
\newcommand{\SSpan}[1]{\category{Span}_{\lcat{#1}}}
\newcommand{\Mod}[2]{{}\category{Mod}_{\lcat{K}}(#1,#2)}
\newcommand{\MMod}[1]{\category{Mod}_{\lcat{#1}}}
\newcommand{\qMod}[2]{\mathrm{mod}(#1,#2)}
\newcommand{\dmod}[3]{\xymatrix@=1.25em{{#2} \ar[r]|\mid^{ {#1}} & {#3}}}

\newcommand{\pbshape}{{\mathord{\bigger\scaleupone\righthalfcup}}}

\newcommand{\fbv}[1]{\{{#1}\}}

\newcommand{\join}{\star}

\newmop{dec}
\newmop{fatdec}
\newmop{slc}
\newmop{fatslc}

\newcommand{\slice}{/}

\newcommand{\slicer}[2]{{#1\mkern-1mu}_{{}\slice{#2}}}

\newmop{ir} 
\newmop{incl}

\newcommand{\ho}{h}

\newcommand{\boundary}{\partial}

\def\reedyfilt#1_#2{#1_{\leq #2}}
\newmop{sk}
\newmop{cosk}
\newmop{res}

\newmop{map}
\newmop{Ho}

\newmop[\pth]{path}
\newmop{cyl}

\newmop[\const]{c}
  
\newmop{coll}
\newmop{wgt}

\newdir{ >}{{}*!/-7pt/@{>}}
\newdir{u(}{{}*!/-4pt/@^{(}}
\newdir{d(}{{}*!/-4pt/@_{(}}
\newdir{|>}{%
  !/4.5pt/@{|}*:(1,-.2)@^{>}*:(1,+.2)@_{>}*+@{}}

\makeatletter
\def\makeslashed#1#2#3#4#5{#1{\mathpalette{\sla@{#2}{#3}{#4}}{#5}}}

\def\@mathlower#1#2#3{\setbox0=\hbox{$\m@th#2#3$}\lower#1\ht0\box0}
\def\mathlower#1#2{\mathpalette{\@mathlower{#1}}{#2}}
\makeatother

\newcommand{\inc}{\hookrightarrow}

\newcommand{\tfib}{\twoheadrightarrow}
\newcommand{\fibt}{\twoheadleftarrow}

\newcommand{\xtfib}[1]{\xtwoheadrightarrow{#1}}
\newcommand{\xfibt}[1]{\xtwoheadleftarrow{#1}}

\newcommand{\prof}{\makeslashed\mathbin\shortmid{-0.1}{0.16}\to}

\newcommand{\longtwoheadrightarrow}{\mathrel{\mathord{-}\mkern-3mu\mathord\twoheadrightarrow}}

\newcommand{\we}{\xrightarrow{\mkern10mu{\smash{\mathlower{0.6}{\sim}}}\mkern10mu}}

\newcommand{\trvfib}{\stackrel{\smash{\mkern-2mu\mathlower{1.5}{\sim}}}\longtwoheadrightarrow}

\newcommand{\To}{\Rightarrow}
\newcommand{\xTo}[1]{\xRightarrow{#1}}

\makeatletter

\def\tens@fn{\futurelet\tens@next}
\def\tens@dn{\def\tens@nextcont}
\newtoks\tens@toks
\def\addtotens@toks#1{\tens@toks=\expandafter{\the\tens@toks#1}}

\def\parsetens@@{%
    \ifx\space@\tens@next \expandafter\tens@dn\space{\tens@fn\parsetens@@}%
    \else\ifx ^\tens@next \tens@dn ^##1{\parsetens@procsep^\addtotens@toks{##1}%
      \tens@fn\parsetens@@}%
    \else\ifx _\tens@next \tens@dn _##1{\parsetens@procsep_\addtotens@toks{##1}%
      \tens@fn\parsetens@@}%
    \else\tens@dn{\ifx *\tens@last \else\addtotens@toks\egroup\fi\the\tens@toks}%
    \fi\fi\fi\tens@nextcont}

\def\parsetens@procsep#1{%
  \ifx *\tens@last \addtotens@toks{#1}\addtotens@toks\bgroup%
  \else\ifx \tens@last\tens@next \addtotens@toks,%
  \else \addtotens@toks\egroup\addtotens@toks\bgroup%
    \addtotens@toks\egroup\addtotens@toks{#1}\addtotens@toks\bgroup%
  \fi\fi\let\tens@last\tens@next}

\newcommand{\tn}[1]{\let\tens@last=*\tens@toks={#1}\tens@fn\parsetens@@}

\makeatother

\def\adjdisplay#1-|#2:#3->#4.{{%
    \xymatrix@R=0em@!C=2.5em{%
      *+[l]{#3} \ar@/_0.55pc/[rr]_-{#2} & {\bot} &
      *+[r]{#4}\ar@/_0.55pc/[ll]_-{#1}}}}

\def\adjdisplaytwo#1-|#2:#3->#4.{{%
\xymatrix@=1.2em{
      {#3}\ar@/_1.5ex/[rr]_-{#2}^-{}="one"
      & & {#4}
      \ar@/_1.5ex/[ll]_-{#1}^-{}="two" 
      \ar@{}"one";"two"|{\bot}
    }}}

\def\tripleadjdisplay#1-|#2-|#3:#4->#5.{{%
\xymatrix@=2.4em{ 
{#4}\ar[r]|{#2} &
{#5} \ar@/_3ex/[l]_{#1}^{\bot} \ar@/^3ex/[l]_{\bot}^{#3}}
}}

\def\adjinline#1-|#2:#3->#4.{{#1}\dashv{#2}:#3\to #4}

\newcommand{\pent}[1]{
  \xybox{
    \POS (0,-15)*+{\a}="0", 
         (-14,-5)*+{\b}="1", 
         (-9,12)*+{\c}="2", 
         (9,12)*+{\d}="3", 
         (14,-5)*+{\e}="4"
    \POS"0" \ar "1"^{\labelstyle \ab}|{}="01"
    \POS"1" \ar "2"^{\labelstyle \bc}|{}="12"
    \POS"2" \ar "3"^{\labelstyle \cd}|{}="23"
    \POS"3" \ar "4"^{\labelstyle \de}|{}="34"
    \POS"0" \ar "4"_{\labelstyle \ae}|{}="04"
    \ifcase #1
    \POS"0" \ar "2"|{\labelstyle \ac}="02"
    \POS"0" \ar "3"|{\labelstyle \ad}="03"
    \POS"02";"1"**{}, ?(0.3) \ar@{=>} ?(0.7)^{\labelstyle \abc}
    \POS"03";"2"**{}, ?(0.25) \ar@{=>} ?(0.5)_{\labelstyle \acd}
    \POS"04";"3"**{}, ?(0.2) \ar@{=>} ?(0.4)_{\labelstyle \ade}
    \or
    \POS"1" \ar "3"|{\labelstyle \bd}="13"
    \POS"1" \ar "4"|{\labelstyle \be}="14"
    \POS"13";"2"**{}, ?(0.3) \ar@{=>} ?(0.7)_{\labelstyle \bcd}
    \POS"14";"3"**{}, ?(0.25) \ar@{=>} ?(0.5)_{\labelstyle \bde}
    \POS"04";"1"**{}, ?(0.25) \ar@{=>} ?(0.5)_{\labelstyle \abe}
    \or
    \POS"2" \ar "4"|{\labelstyle \ce}="24"
    \POS"0" \ar "2"|{\labelstyle \ac}="02"
    \POS"02";"1"**{}, ?(0.3) \ar@{=>} ?(0.7)^{\labelstyle \abc}
    \POS"04";"2"**{}, ?(0.2) \ar@{=>} ?(0.35)_{\labelstyle \ace}
    \POS"24";"3"**{}, ?(0.2) \ar@{=>} ?(0.6)^{\labelstyle \cde}
    \or
    \POS"1" \ar "3"|{\labelstyle \bd}="13"
    \POS"0" \ar "3"|{\labelstyle \ad}="03"
    \POS"04";"3"**{}, ?(0.2) \ar@{=>} ?(0.4)_{\labelstyle \ade}
    \POS"13";"2"**{}, ?(0.3) \ar@{=>} ?(0.7)_{\labelstyle \bcd}
    \POS"03";"1"**{}, ?(0.25) \ar@{=>} ?(0.5)^{\labelstyle \abd}
    \or
    \POS"2" \ar "4"|{\labelstyle \ce}="24"
    \POS"1" \ar "4"|{\labelstyle \be}="14"
    \POS"24";"3"**{}, ?(0.2) \ar@{=>} ?(0.6)^{\labelstyle \cde}
    \POS"04";"1"**{}, ?(0.25) \ar@{=>} ?(0.5)_{\labelstyle \abe}
    \POS"14";"2"**{}, ?(0.25) \ar@{=>} ?(0.5)^{\labelstyle \bce}
    \else\fi
  }
}

\newcommand{\pentofpent}[1]{
  \def\baselen{#1}
  \begin{xy}
    0;<\baselen,0mm>:
    *{\xybox{
        \POS(0,-4)*[o]{\pent 0}="zero"
        \POS(16,40)*[o]{\pent 3}="three"
        \POS(72,40)*[o]{\pent 1}="one"
        \POS(88,-4)*[o]{\pent 4}="four"
        \POS(44,-36)*[o]{\pent 2}="two"
        \ar@<1ex>"zero";"three"^-{\objectstyle\abcd}
        \ar@<1ex>"three";"one"^-{\objectstyle\abde}
        \ar@<1ex>"one";"four"^-{\objectstyle\bcde}
        \ar@<-1ex>"zero";"two"_-{\objectstyle\acde}
        \ar@<-1ex>"two";"four"_-{\objectstyle\abce}
        \ar@{=>}(44,-5);(44,+15)^{\objectstyle\abcde}
     }}
  \end{xy}
}

\setlist{}
\setenumerate{leftmargin=*,labelindent=\parindent}
\setitemize{leftmargin=*,labelindent=0.5\parindent}
\setdescription{leftmargin=1em}

\swapnumbers

\theoremstyle{plain}

\newtheorem{thm}{Theorem}[subsection]
\newtheorem{lem}[thm]{Lemma}
\newtheorem{cor}[thm]{Corollary}

\newtheorem{prop}[thm]{Proposition}

\theoremstyle{definition}
\newtheorem{defn}[thm]{Definition}
\newtheorem{ex}[thm]{Example}
\newtheorem{ntn}[thm]{Notation}

\theoremstyle{remark}
\newtheorem{obs}[thm]{Observation}
\newtheorem{rec}[thm]{Recall}
\newtheorem{rmk}[thm]{Remark}

\makeatletter
\let\c@equation\c@thm
\makeatother
\numberwithin{equation}{subsection}

\newcommand{\extRef}[3]{%
  {\protect\IfBeginWith{#3}{itm:}{}{#2.}}\ref*{#1:#3}}

\newcommand{\refI}{\extRef{found}{I}}

\newcommand{\refIV}{\extRef{yoneda}{IV}}

\date{\today}

\begin{document}

  \ifpdf
  \DeclareGraphicsExtensions{.pdf, .jpg, .tif}
  \else
  \DeclareGraphicsExtensions{.eps, .jpg}
  \fi

	\begin{abstract}
    Various models of $(\infty,1)$-categories, including quasi-categories, complete Segal spaces, Segal categories, and naturally marked simplicial sets can be considered as the objects of an $\infty$-\emph{cosmos}. In a generic $\infty$-cosmos, whose objects we call $\infty$-\emph{categories}, we introduce \emph{modules} (also called \emph{profunctors} or \emph{correspondences}) between $\infty$-categories, incarnated as as spans of suitably-defined fibrations with groupoidal fibers. As the name suggests, a module from $A$ to $B$ is an $\infty$-category equipped with a left action of $A$ and a right action of $B$, in a suitable sense. Applying the fibrational form of the Yoneda lemma, we develop a general calculus of modules, proving that they naturally assemble into a multicategory-like structure called a \emph{virtual equipment}, which is known to be a robust setting in which to develop formal category theory. Using the calculus of modules, it is straightforward to define and study pointwise Kan extensions, which we relate, in the case of cartesian closed $\infty$-cosmoi, to limits and colimits of diagrams valued in an $\infty$-category, as introduced in previous work.
	\end{abstract}

  \begin{asciiabstract}
    Various models of (infinity,1)-categories, including quasi-categories, complete Segal spaces, Segal categories, and naturally marked simplicial sets can be considered as the objects of an infinity-cosmos. In a generic infinity-cosmos, whose objects we call infinity-categories, we introduce modules (also called profunctors or correspondences) between infinity-categories, incarnated as as spans of suitably-defined fibrations with groupoidal fibers. As the name suggests, a module from A to B is an infinity-category equipped with a left action of A and a right action of B, in a suitable sense. Applying the fibrational form of the Yoneda lemma, we develop a general calculus of modules, proving that they naturally assemble into a multicategory-like structure called a virtual equipment, which is known to be a robust setting in which to develop formal category theory. Using the calculus of modules, it is straightforward to define and study pointwise Kan extensions, which we relate, in the case of cartesian closed infinity-cosmoi, to limits and colimits of diagrams valued in an infinity-category, as introduced in previous work.
	\end{asciiabstract}

  \maketitle

  

\section{Introduction}

Previous work \cite{RiehlVerity:2012tt, RiehlVerity:2012hc, RiehlVerity:2013cp, RiehlVerity:2015fy} shows that the basic theory of $(\infty,1)$-categories --- categories that are weakly enriched over $\infty$-groupoids, i.e., topological spaces --- can be developed ``model independently,'' at least if one is content to work with one of the better-behaved models: namely, quasi-categories, complete Segal spaces, Segal categories, or naturally marked simplicial sets. More specifically, we show that a large portion of the category theory of quasi-categories---one model of $(\infty,1)$-categories that has been studied extensively by Joyal, Lurie, and others---can be re-developed from the abstract perspective of the \emph{homotopy 2-category} of the $\infty$-\emph{cosmos} of quasi-categories. Each of the above-mentioned models has its own $\infty$-cosmos, a quotient of which defines the homotopy 2-category. As our development of the basic theory takes place entirely within this axiomatic framework, the basic definitions and theorems apply simultaneously and uniformly to each of the above-mentioned models.

An $\infty$-\emph{cosmos} is a universe within which to develop the basic category theory of its objects, much like a simplicial model category is a universe within which to develop the basic homotopy theory of its objects. A simplicial model category is a model category that is enriched as such over Quillen's model structure on simplicial sets, whose fibrant objects, the \emph{Kan complexes}, model $\infty$-groupoids. By analogy, an $\infty$-cosmos resembles  a model category that is enriched as such over Joyal's model structure on simplicial sets, whose fibrant objects, the \emph{quasi-categories}, model $(\infty,1)$-categories; more precisely an $\infty$-cosmos is the simplicial subcategory spanned by the fibrant objects.  The $\infty$-cosmos axioms discard the features of a quasi-categorically enriched model structure that are not necessary for our proofs. We restrict to the subcategory of fibrant objects, which traditionally model the homotopy coherent category-like structures of interest, and forget about the cofibrations, which are not needed for our constructions. We refer to fibrations between fibrant objects as \emph{isofibrations}, as these will play a role analagous to the categorical isofibrations. Finally, in contrast to the form of this axiomatization presented in \cite{RiehlVerity:2015fy}, we assume that ``all fibrant objects are cofibrant,'' which happens to be true of all of the examples that we will consider in the present paper. While everything we discuss here holds in a general $\infty$-cosmos, this cofibrancy restriction allows for useful didactic simplification of the arguments presented here.

We use the term $\infty$-\emph{categories} to refer to the objects in some $\infty$-cosmos; these are the infinite-dimensional categories within the scope of our treatment. Examples include the models of $(\infty,1)$-categories mentioned above, but also ordinary categories, $\theta_n$-spaces, general categories of ``Rezk objects'' valued in a reasonable model category, and also sliced (fibred) versions of the $\infty$-categories in any $\infty$-cosmos. In particular, theorems about $\infty$-categories, i.e., objects in some $\infty$-cosmos, are not only theorems about $(\infty,1)$-categories. This being said, for the present narrative, the interpretation of ``$\infty$-categories'' as being ``well-behaved models of $(\infty,1)$-categories'' might prove the least confusing.

Quillen's model category axioms provide a well-behaved homotopy category, spanned by the fibrant-cofibrant objects, in which the poorly behaved notion of weak equivalence is equated with a better behaved notion of homotopy equivalence. Similarly, an $\infty$-cosmos provides a well-behaved \emph{homotopy 2-category}, which should be thought of as a categorification of the usual homotopy category, in which the canonical 2-categorical notion of equivalence coincides precisely with the $\infty$-cosmos level notion of (weak) equivalence. This means that 2-categorical equivalence-invariant  definitions are appropriately ``homotopical.'' Our work is largely 2-categorical, presented in terms of the $\infty$-categories, $\infty$-functors, and $\infty$-natural transformations that assemble into the homotopy 2-category of some $\infty$-cosmos, much like ordinary categorical notions can be defined in terms of categories, functors, and natural transformations. References to \cite{RiehlVerity:2012tt, RiehlVerity:2012hc, RiehlVerity:2013cp, RiehlVerity:2015fy} will have the form I.x.x.x, II.x.x.x, III.x.x.x, and IV.x.x.x respectively. We spare the reader the pain of extensive cross referencing however, by beginning with a comprehensive survey of the necessary background in \S\ref{sec:background}.

The aim of this paper is to develop the calculus of \emph{modules} between $\infty$-categories. In classical 1-category theory, ``modules'' are our prefered name for what are also called \emph{profunctors}, \emph{correspondences}, or \emph{distributors}: a \emph{module} $E$ from $A$ to $B$ is a functor $E \colon B\op \times A \to \Set$. The bifunctoriality of $E$ is expressed via ``left'' (covariant) actions on the sets $E(b,a)$ by morphisms in $A$ and ``right'' (contravariant) actions by morphisms in $B$. The hom bifunctor associated to any category $A$ defines a module from $A$ to $A$, the \emph{arrow module} denoted by $A^\cattwo$. More generally, any functor $f \colon B \to A$ can be encoded as a \emph{covariant represented module} from $B$ to $A$ and as a \emph{contravariant represented module} from $A$ to $B$; these modules are defined by restricting one or the other variable in the arrow module  $A^\cattwo$. Given a second functor $g \colon C \to A$, there is a module from $C$ to $B$ obtained by restricting the domain variable of the arrow module $A^\cattwo$ along $f$ and restricting the codomain variable along $g$. This module can be regarded as the composite of the contravariant module representing $f$ with the covariant module representing $g$.

There are a number of equivalent 2-categorical incarnations of modules in classical 1-category theory. Our preferred mechanism is to represent a module $E$ from $A$ to $B$ as a \emph{two-sided discrete fibration} $(q,p) \colon E \tfib A \times B$. In particular, under this presentation, a module is a category fibered over $A \times B$; by analogy, a module between $\infty$-categories $A$ and $B$ will be an $\infty$-category fibred over $A \times B$. As slices of $\infty$-cosmoi are again $\infty$-cosmoi, this means that we can apply theorems from our previous work, which concern the objects in any $\infty$-cosmos, to develop the theory of modules. By contrast, Lurie \cite{Lurie:2009fk} and Barwick--Schommer-Pries  \cite{BSP:2011ot} represent modules as correspondences --- cospans rather than spans. Haugseng \cite{Haugseng:2015bn} uses an $\infty$-operadic approach to define modules for enriched $\infty$-categories.

In \S\ref{sec:modules}, we define modules between $\infty$-categories, the prototypical examples being the arrow $\infty$-categories and comma $\infty$-categories that play a central role in previous work in the series. A \emph{module} $E$ from $A$ to $B$ will be an $\infty$-category equipped with an isofibration $(q,p) \colon E \tfib A \times B$ that has ``groupoidal fibers'' and satisfies two additional properties. Firstly, $(q,p)$ defines a \emph{cartesian fibration} in the sliced $\infty$-cosmos over $A$, in the sense introduced in \S\refIV{sec:cartesian}. Loosely, this says that $B$ acts on the right of $E$, over $A$. Dually, $(q,p)$ defines a \emph{cocartesian fibration} in the sliced $\infty$-cosmos over $B$, which says that $A$ acts on the left of $E$, over $B$. Applying results about cartesian and groupoidal cartesian fibrations developed in \S\refIV{sec:stability} and \S\refIV{sec:yoneda}, we prove that modules can be pulled back along an arbitrary pair of functors, and we characterize the quasi-category of module maps out of a represented module, this result being an application of the relative case of the Yoneda lemma, in the form of \refIV{cor:groupoidal-yoneda}.

In \S\ref{sec:virtual}, we develop the calculus of modules, which resembles the calculus of (bi)modules between rings. Unital rings, ring homomorphisms, modules, and module maps assemble into a 2-dimensional categorical structure known as a \emph{double category}. Ring homomorphisms can be composed ``vertically'' while modules can be composed ``horizontally,'' by tensoring. A module map, from an $A$-$B$-module to an $A'$-$B'$-module over a pair of ring homomorphisms $A \to A'$ and $B \to B'$ is an $A$-$B$-module homomorphism from the former to the latter, where the $A$-$B$-actions on the codomain are defined by restriction of scalars. These module maps can be represented as 2-dimensional ``cells'' inside squares, which can be composed in both the ``horizontal'' and ``vertical'' directions.

Similarly, $\infty$-categories, $\infty$-functors, modules, and module maps assemble into a 2-dimensional categorical structure. At the level of our $\infty$-cosmos axiomatization, we are not able to define tensor products for all modules, which would involve homotopy colimits that are not included within this general framework. But as it turns out, this is a deficiency we can work around for our purposed here. Modules between $\infty$-categories naturally assemble into a \emph{virtual double category}, where module maps are allowed to have a ``multi-source.'' Our main theorem in this section is that the virtual double category of modules is in fact a \emph{virtual equipment}, in the sense of Cruttwell and Shulman \cite{CruttwellShulman:2010au}. The proof of this result, which appears as Theorem~\ref{thm:virtual-equipment}, follows easily from our work in \S\ref{sec:modules}, and we spend the remainder of this section exploring its consequences. In particular, we show that the homotopy 2-category of the $\infty$-cosmos embeds both covariantly and contravariantly into the virtual equipment of modules, by sending an $\infty$-category to either its covariant or contravariant represented module.

Prior categorical work suggests that Theorem~\ref{thm:virtual-equipment}, which demonstrates that modules between $\infty$-categories assemble into a virtual equipment, serves as the starting point for many further developments in the formal category theory of $\infty$-categories  \cite{Street:1974:ec,wood:proI, wood:proII, Verity:2011qy, Weber:2007ys, CruttwellShulman:2010au, Shulman:2013ei}. Here we illustrate only a small portion of the potential applications in \S\ref{sec:pointwise-kan} by introducing pointwise Kan extensions, exact squares, and final and initial functors for $\infty$-categories. There is a naive notion of Kan extension which can be defined in any 2-category, in particular in the homotopy 2-category, but the universal property so-encoded is insufficiently robust to define a good notion for $\infty$-functors between $\infty$-categories. The correct notion is of \emph{pointwise} Kan extension, which we define in two different ways that we prove equivalent in Proposition~\ref{prop:pointwise-kan}. One definition, guided by Street \cite{Street:1974:ec}, is that a pointwise Kan extension is an ordinary extension diagram in the homotopy 2-category of $\infty$-categories that is stable under pasting with comma, or more generally \emph{exact}, squares. A second definition, is that a pointwise Kan extension is a Kan extension under the covariant embeddding into the virtual equipment of modules.

In a cartesian closed $\infty$-cosmos, pointwise Kan extensions along the unique functor to the terminal $\infty$-category correspond exactly to the absolute lifting diagrams used in \S\refI{sec:limits} to define limits and colimits of diagrams valued in an $\infty$-category. Thus, pointwise Kan extensions can be used to extend this notion to non-cartesian closed $\infty$-cosmoi, such as sliced $\infty$-cosmoi or the $\infty$-cosmoi of Rezk objects. We introduce \emph{initial} and \emph{final functors} between $\infty$-categories, defined in terms of \emph{exact squares}, which are in turn characterized using the virtual equipment of modules. We prove that for any final functor $k \colon C \to D$, $D$-indexed colimits exist if and only if the restricted $C$-indexed colimits do, and when they exist they coincide. We conclude by proving the Beck-Chevalley property for functorial pointwise Kan extensions, and use it to sketch an argument that any complete and cocomplete quasi-category gives rise to a derivator in the sense of Heller \cite{Heller:1988ly} and Grothendieck.

The results contained here might appear to be specialized to the $\infty$-cosmoi whose objects model $(\infty,1)$-categories. For instance, in these $\infty$-cosmoi, the groupoidal objects,  which serve as the fibers for modules,  will be precisely the $\infty$-groupoids; for other  $\infty$-cosmoi, the groupoidal objects will be those objects whose underlying quasi-categories are Kan complexes. Nonetheless, broader applications of the present results are anticipated. For instance, Par\'e conjectured \cite{pare:dl} and Verity proved \cite{Verity:2011qy} that the \emph{flexible 2-limits}, which is the class of 2-dimensional limits that are appropriately homotopical, are captured by the double-categorical notion of \emph{persistent limits}. That is, 2-dimensional limits of diagrams defined internally to 2-categories can be studied by regarding those 2-categories as vertically-discrete double categories, and using two-sided discrete fibrations (modules) between such double categories to define the shape of the limit notion.

The flexible 2-limits mentioned here are the only 2-dimensional limits that have meaningful $(\infty,2)$-categorical analogues. This result suggests that we should be able to apply the calculus of modules---exactly as developed here in a general $\infty$-cosmos---between $\infty$-categories that model $(\infty,2)$-categories,  incarnated as Rezk objects in quasi-categories, to define weights for 2-dimensional limit and colimits of diagrams valued inside an $(\infty,2)$-category. We plan to explore this topic in a future paper.

\subsection{Acknowledgments} 

This material is based upon work supported by the National Science Foundation under Award No.~DMS-1509016 and by the Australian Research Council under Discovery grant number DP130101969. A substantial portion of this work was completed while the second-named author was visiting Harvard University, during which time he was partially supported by an NSF grant DMS-0906194 and a DARPA grant HR0011-10-1-0054-DOD35CAP held by Mike Hopkins. We are particularly grateful for his support. Some of the writing took place while the first-named author was in residence at the Hausdorff Research Institute for Mathematics, with travel support provided by the Simons Foundation through an AMS-Simons Travel Grant. Edoardo Lanari pointed out a circularity in the interpretation of the axiomatization for an $\infty$-cosmos with all objects cofibrant, now corrected.



\section{Background}\label{sec:background}

In \S\ref{ssec:cosmoi}, we introduce the axiomatic framework in which we work --- an $\infty$-\emph{cosmos} and its \emph{homotopy 2-category} --- first introduced in \cite{RiehlVerity:2015fy} but considered here in a simplified form. The underlying 1-categories of an $\infty$-cosmos (a simplicially enriched category) and its homotopy 2-category (a $\Cat$-enriched category) are identical: objects are $\infty$-\emph{categories} and morphisms are $\infty$-\emph{functors} (with the prefix ``$\infty$'' typically dropped).

In \S\ref{ssec:sliced}, we consider slices of an $\infty$-cosmos $\lcat{K}$ over a fixed object $B$. In this context, there are two closely related 2-categories: the homotopy 2-category $(\lcat{K}/B)_2$ of the sliced $\infty$-cosmos $\lcat{K}/B$ and the slice $\lcat{K}_2/B$ of the homotopy 2-category of $\lcat{K}$. Both 2-categories have the same underlying 1-categories but their 2-cells do not coincide. However, there is a canonically defined \emph{smothering 2-functor} $(\lcat{K}/B)_2 \to \lcat{K}_2/B$ which means that, for many practical purposes, the distinction between these slices is not so important.

In \S\ref{ssec:limits}, we review the construction of \emph{comma $\infty$-categories}, a particular simplicially enriched limit notion permitted by the axioms of an $\infty$-cosmos that produces an object, a pair of functors, and a natural transformation that enjoy a particular weak 2-dimensional universal property in the homotopy 2-category. Comma $\infty$-categories are well-defined up to equivalence of spans in the homotopy 2-category, but for the purpose of calculations we frequently make use of a particular model, defined up to isomorphism in the $\infty$-cosmos.

In \S\ref{ssec:cartesian}, we summarize the main definitions and results concerning cartesian fibrations and groupoidal cartesian fibrations, contained in \S\refIV{sec:cartesian}-\refIV{sec:yoneda}. These will be used in \S\ref{sec:modules} to define \emph{modules} between $\infty$-categories, which are two-sided groupoidal cartesian fibrations of a particular variety.

\subsection{\texorpdfstring{$\infty$}{infinity}-cosmoi and their homotopy 2-categories}\label{ssec:cosmoi}

The $\infty$-cosmoi of principle interest to this paper are those whose objects, the $\infty$-\emph{categories}, model $(\infty,1)$-categories. These include the $\infty$-cosmoi of quasi-categories, complete Segal spaces, Segal categories, and marked simplicial sets. In each of these, all objects are cofibrant. Adding this as an assumption to the definition of an $\infty$-cosmos, as presented in \refIV{qcat.ctxt.def}, we obtain a simplified form of the axiomatization, contained in Definition~\ref{qcat.ctxt.cof.def} below. This assumption is not required for any of the main theorems presented in this paper, but it does simplify their proofs.

 For the duration of this paper, an $\infty$-\emph{cosmos} will refer to an $\infty$-cosmos with all objects cofibrant. We refer to the objects of the underlying 1-category of an $\infty$-cosmos as $\infty$-\emph{categories} and its morphisms as $\infty$-\emph{functors}, or simply \emph{functors}.

\begin{defn}[$\infty$-cosmos]\label{qcat.ctxt.cof.def}
An $\infty$-\emph{cosmos} (with all objects cofibrant) is a simplicially enriched category $\lcat{K}$ whose mapping spaces $\map(A,B)$ are all quasi-categories that is equipped with a specified subcategory of \emph{isofibrations} satisfying the following axioms:
 \begin{enumerate}[label=(\alph*)]
    \item\label{qcat.ctxt.cof:a} (completeness) As a simplicially enriched category,  $\lcat{K}$ possesses a terminal object $1$, cotensors $U\pwr A$ of all objects $A$ by all finitely presented simplicial sets $U$, and pullbacks of isofibrations along any functor;
    \item\label{qcat.ctxt.cof:b} (isofibrations) The class of isofibrations contains the isomorphisms and all of the functors $!\colon A \to 1$ with codomain $1$; is stable under pullback along all functors; and if $p\colon E\tfib B$ is an isofibration in $\lcat{K}$ and $i\colon U\inc V$ is an inclusion of finitely presented simplicial sets then the Leibniz cotensor $i\leib\pwr p\colon V\pwr E\to U\pwr E\times_{U\pwr B} V\pwr B$ is an isofibration. Moreover, for any object $X$ and isofibration $p \colon E \tfib B$, $\map(X,p) \colon \map(X,E) \tfib \map(X,B)$ is an isofibration of quasi-categories.
    \item\label{qcat.ctxt.cof:c} (cofibrancy) All objects are \emph{cofibrant}, in the sense that they enjoy the left lifting property with respect to all \emph{trivial fibrations} in $\lcat{K}$, a class of maps that will now be defined.
  \end{enumerate}
\end{defn}

\begin{defn}[equivalences in an $\infty$-cosmos]
The underlying category of an $\infty$-cosmos $\lcat{K}$ has a canonically defined class of (representably-defined) equivalences. A functor $f \colon A \to B$ is an \emph{equivalence} just when the induced functor $\map(X,f) \colon \map(X,A) \to \map(X,B)$ is an equivalence of quasi-categories for all objects $X \in \lcat{K}$.
\end{defn}

Note that the equivalences define a subcategory and satisfy the 2-of-6 property. The  \emph{trivial fibrations} are those functors that are both equivalences and isofibrations; immediately it follows that the trivial fibrations define a subcategory containing the isomorphisms. We use the symbols ``$\tfib$'', ``$\we$'', and ``$\trvfib$'' to denote the isofibrations, equivalences, and trivial fibrations, respectively.  The trivial fibrations enjoy the following stability properties:

\begin{lem}[stability properties of trivial fibrations]\label{lem:triv.fib.stab} $\quad$ 
\begin{enumerate}[label=(\alph*)]
\item\label{itm:triv.fib.stab:a} If  $p \colon E \trvfib B$ is a trivial fibration in an $\infty$-cosmos $\tcat{K}$, then for any object $X$, $\map(X,f) \colon \map(X,E) \trvfib \map(X,B)$ is a trivial fibration of quasi-categories.
\item\label{itm:triv.fib.stab:b} The trivial fibrations are stable under pullback along any functor.
\item\label{itm:triv.fib.stab:c}  The Leibniz cotensor $i\leib\pwr p\colon V\pwr E\to U\pwr E\times_{U\pwr B} V\pwr B$ of an isofibration $p\colon E\tfib B$ in $\lcat{K}$ and a monomorphism $i\colon U\inc V$ between presented simplicial sets   is a trivial fibration when $p$ is a trivial fibration in $\lcat{K}$ or $i$ is trivial cofibration in the Joyal model structure on $\sSet$.
\end{enumerate}
\end{lem}
\begin{proof}
\ref{itm:triv.fib.stab:a} is immediate from \ref{qcat.ctxt.cof:b} and the definitions, while \ref{itm:triv.fib.stab:b} and \ref{itm:triv.fib.stab:c} follow from the analogous properties for isofibrations, the corresponding stability properties for quasi-categories established in Example~\refIV{ex:qcat-qcat-ctxt}, and the fact that the referenced simplicially limits are representably defined.
\end{proof}

\begin{rmk} An $\infty$-cosmos in the sense of Definition~\ref{qcat.ctxt.cof.def} is exactly an $\infty$-cosmos in the sense of Definition~\refIV{qcat.ctxt.def} in which the weak equivalences are taken to be the class of equivalences and in which all objects are cofibrant. 
\end{rmk}

\begin{defn}[cartesian closed $\infty$-cosmoi]\label{defn:closed-cosmoi}
An $\infty$-cosmos is \emph{cartesian closed} if the product bifunctor $-\times - \colon \lcat{K} \times \lcat{K} \to \lcat{K}$ extends to a simplicially enriched two-variable adjunction
\[ \map(A \times B,C) \cong \map(A, C^B) \cong \map(B,C^A).\]
\end{defn}

Examples~\refIV{ex:qcat-qcat-ctxt}, \refIV{ex:cat-cosmos}, \refIV{ex:CSS-cosmos}, \refIV{ex:segal-cosmos}, and \refIV{ex:marked-cosmos} establish $\infty$-cosmoi for quasi-categories, ordinary categories, complete Segal spaces, Segal categories, and marked simplicial sets, respectively. All of these examples are cartesian closed.

Example~\refIV{ex:sliced.contexts} proves that there exist \emph{sliced $\infty$-cosmoi} defined as follows:

\begin{defn}[sliced $\infty$-cosmoi]\label{defn:sliced-cosmoi}
If $\lcat{K}$ is an $\infty$-cosmos and $B$ is a fixed object, then there is an $\infty$-cosmos  $\lcat{K}\slice B$ in which the:
\begin{itemize}
\item objects are isofibrations $p \colon E \tfib B$ with codomain $B$;
\item mapping quasi-category from $p \colon E \tfib B$ to $q \colon F \tfib B$ is defined by  taking the pullback
  \begin{equation}\label{eq:sliced-map}
    \xymatrix@=1.5em{
      {\map_B(p,q)}\pbexcursion\ar[r]\ar@{->>}[d] &
      {\map(E,F)}\ar@{->>}[d]^{\map(E,q)} \\
      {\Del^0}\ar[r]_-{p} & {\map(E,B)}
    }
  \end{equation}
  in simplicial sets;
\item isofibrations, equivalences, and trivial fibrations are created by the forgetful functor $\lcat{K}/B \to \lcat{K}$;
\item the terminal object is $\id_B \colon B \tfib B$;
\item pullbacks are created by the forgetful functor $\lcat{K}/B \to \lcat{K}$;
\item the cotensor of an object $p\colon E\tfib B$ of $\lcat{K}\slice B$ by a finitely presented simplicial set $U$ is the left-hand vertical arrow in the following pullback in $\lcat{K}$:
  \begin{equation*}
    \xymatrix@=1.5em{
      {U \pwr_p E}\pbexcursion\ar[r]\ar@{->>}[d] &  
      {U \pwr E}\ar@{->>}[d]^{U \pwr p} \\
      {B}\ar[r]_-{\Delta} & {U \pwr B}
    }
  \end{equation*} where the arrow $\Delta$ appearing along the bottom is the adjoint transpose of the constant map $U \to \Delta^0 \xrightarrow{\id_B} \map(B,B)$ at the identity for $B$.  
\end{itemize}
\end{defn}

\begin{obs}\label{obs:general-functors-in-slices}
  In order to ensure that a slice of an $\infty$-cosmos is again an $\infty$-cosmos it is necessary to take only the isofibrations as the objects of $\lcat{K}/B$. For example, in order to ensure that the mapping space defined in~\eqref{eq:sliced-map} is a quasi-category we require that the vertical functor on the right of that square is an isofibration of quasi-categories and that, in turn, follows so long as the codomain functor $q\colon F\tfib B$ is an isofibration in $\lcat{K}$. Notice here, however, that this result holds without any assumption on the domain functor $p\colon E\to B$. Indeed, as a matter of general principle, we may treat an arbitrary functor $f\colon X\to B$ as if it were an object of $\lcat{K}/B$ so long as we are never called upon to place it in a codomain position in any argument.
\end{obs}

\begin{defn}\label{defn:qcat-ctxt-functor} A \emph{functor of $\infty$-cosmoi} $F \colon \lcat{K} \to \lcat{L}$ is a simplicial functor that preserves isofibrations and the limits listed  in~\ref{qcat.ctxt.cof.def}\ref{qcat.ctxt.cof:a}.
\end{defn}

Note that simplicial functoriality implies that a functor of $\infty$-cosmoi also preserves equivalences by a short exercise left to the reader (for a hint, see Proposition \refIV{prop:equiv.are.weak.equiv} recalled below) and hence also trivial fibrations. 

\begin{ex}[functors of $\infty$-cosmoi]\label{ex:functors}
The following define functors of $\infty$-cosmoi:
\begin{itemize}
\item $\map(X,-) \colon\lcat{K} \to \qCat$ for any object $X \in \lcat{K}$ (see Proposition~\refIV{prop:representable-functors}). The special case $\map(1,-) \colon \lcat{K} \to \qCat$ is the \emph{underlying quasi-category functor}.
\item $U\pwr - \colon \lcat{K} \to \lcat{K}$ for any finitely presented simplicial set $U$ (by \ref{qcat.ctxt.cof.def}\ref{qcat.ctxt.cof:c} and the fact that simplicially enriched limits commute with each other).
\item The pullback functor $f^* \colon \lcat{K}/B \to \lcat{K}/A$ for any functor $f \colon A \to B \in \lcat{K}$ (see Proposition~\refIV{prop:pullback-functors}).
\item The underlying quasi-category functor $\map(1,-)\colon \category{CSS} \to \qCat$ that takes a complete Segal space to its $0\th$ row (see Example~\refIV{ex:CSS-cosmos}).
\item The functor $t^! \colon \qCat \to \category{CSS}$ defined in Example~\refIV{ex:other-CSS-functor}.
\item The underlying quasi-category functor $\map(1,-)\colon \category{Segal} \to \qCat$ that takes a Segal category to its $0\th$ row (see Example~\refIV{ex:segal-cosmos}). 
\item The underlying quasi-category functor that carries a naturally marked simplicial set to its underlying quasi-category (see Example~\refIV{ex:marked-cosmos}).
\item The inclusion $\Cat\to \qCat$ of categories into quasi-categories that identifies a category with its nerve (see Example~\refIV{ex:cat-cosmos}).
\end{itemize}
\end{ex}

\begin{defn}[the homotopy 2-category of $\infty$-cosmos]
A quotient of an $\infty$-cosmos $\lcat{K}$ defines the \emph{homotopy 2-category}. This is a strict 2-category $\lcat{K}_2$ with the same objects and 1-morphisms and whose hom-categories are defined by  \[\hom(A,B)\defeq \ho(\map(A,B))\] to be the homotopy categories of the mapping quasi-categories in $\lcat{K}$. 

Put concisely, the homotopy 2-category is the 2-category $\lcat{K}_2 \defeq \ho_*\lcat{K}$ defined by applying the homotopy category functor $\ho \colon \qCat \to \Cat$ to the mapping quasi-categories of the $\infty$-cosmos. By the same construction, a functor $F \colon \lcat{K} \to \lcat{L}$ of $\infty$-cosmoi induces a 2-functor $F_2 \defeq \ho_*F \colon \lcat{K}_2 \to \lcat{L}_2$ between their homotopy 2-categories.
\end{defn}

Isofibrations and trivial fibrations in the $\infty$-cosmos define \emph{representable isofibrations} and \emph{representable surjective equivalences} in the homotopy 2-category:

{
\renewcommand{\thethm}{\refIV{lem:isofib.are.representably.so}}
\begin{lem} For all objects $X$ in an $\infty$-cosmos:
\begin{enumerate}[label=(\roman*)]
\item If $E \tfib B$ is an isofibration, then $\hom(X,E) \tfib \hom(X,B)$ is an isofibration.
\item If $E \trvfib B$ is a trivial fibration, then $\hom(X,E) \trvfib \hom(X,B)$ is a surjective equivalence.
\end{enumerate}
\end{lem}
\addtocounter{thm}{-1}
}

Importantly: 

{
\renewcommand{\thethm}{\refIV{prop:equiv.are.weak.equiv}}
\begin{prop} A functor $ f\colon A \to B$ is an equivalence in the $\infty$-cosmos if and only if it is an equivalence in the homotopy 2-category.
\end{prop}
\addtocounter{thm}{-1}
}

The upshot is that any categorical notion defined up to equivalence in the homotopy 2-category is characterized up to (weak) equivalence in the $\infty$-cosmos.

Axioms \ref{qcat.ctxt.cof.def}\ref{qcat.ctxt.cof:a} and \ref{qcat.ctxt.cof:b} imply that an $\infty$-cosmos has finite products satisfying a simplicially enriched universal property. Consequently:

\begin{prop}\label{prop:cartesian-closure}
The homotopy 2-category of an $\infty$-cosmos has finite products. If the $\infty$-cosmos is cartesian closed, then so is its homotopy 2-category.
\end{prop}
\begin{proof}
The homotopy category functor $\ho \colon \qCat \to \Cat$ preserves finite products. Applying this to the  defining isomorphisms $\map(X,1)\cong 1$ and $\map(X, A\times B)\cong\map(X,A)\times\map(X,B)$ for the simplicially enriched terminal object and binary products of $\lcat{K}$ yields isomorphisms $\hom(X,1)\cong \catone$ and $\hom(X, A\times B)\cong\hom(X,A)\times\hom(X,B)$. These demonstrate that $1$ and $A\times B$ are also the 2-categorical terminal object and binary products in $\lcat{K}_2$.

In this case where $\lcat{K}$ is cartesian closed, as defined in \ref{defn:closed-cosmoi}, applying the homotopy category functor to the defining isomorphisms on mapping quasi-categories yields the required natural isomorphisms
\[ \hom(A \times B, C) \cong \hom(A,C^B) \cong \hom(B,C^A)\] of hom-categories.
\end{proof}

\begin{defn}\label{defn:groupoidal-object} We say an object $E$ in an $\infty$-cosmos $\lcat{K}$ is \emph{groupoidal} if it is groupoidal in the homotopy 2-category $\lcat{K}_2$, that is, if every 2-cell with codomain $E$ is invertible. This says exactly that for each $X \in \lcat{K}$, the hom-category $\hom(X,E)$ is a groupoid. By a well-known result of Joyal \cite[1.4]{Joyal:2002:QuasiCategories}, this is equivalent to postulating that each mapping quasi-category $\map(X,E)$ is a Kan complex.
\end{defn}

\begin{rmk}
In the $\infty$-cosmoi whose objects model $(\infty,1)$-categories, note that the groupoidal objects are precisely the corresponding $\infty$-\emph{groupoids}. For instance, in the $\infty$-cosmos for quasi-categories, an object is groupoidal if and only if it is a Kan complex. In the $\infty$-cosmos for marked simplicial sets, an object is groupoidal if and only if it is a Kan complex with every edge marked. For general $\infty$-cosmoi, it is always the case that the underlying quasi-category of a groupoidal object is a Kan complex.
\end{rmk}

\subsection{Sliced homotopy 2-categories}\label{ssec:sliced}

For any $\infty$-cosmos $\lcat{K}$ and any object $B$, \ref{defn:sliced-cosmoi} recalls the definition of the \emph{sliced $\infty$-cosmos} $\lcat{K}/B$. In this section, we reprise the relationship between the homotopy 2-category $(\lcat{K}/B)_2$ of the sliced $\infty$-cosmos and the slice $\lcat{K}_2/B$ of the homotopy 2-category of $\lcat{K}$. Our convention is that both of these 2-categories have the same objects, namely the isofibrations with codomain $B$, but their hom-categories differ.

\begin{defn}\label{defn:two-slices}
The objects of $(\lcat{K}/B)_2$ and $\lcat{K}_2/B$ are the isofibrations with codomain $B$. The hom-category between $p \colon E \tfib B$ and $q \colon F \tfib B$ in $(\lcat{K}/B)_2$ is defined by applying the homotopy category functor $\qCat \to \Cat$ to the mapping quasi-category defined by the left-hand pullback of simplicial sets, while the corresponding hom-category in $\lcat{K}_2/B$ is defined by the right-hand pullback of categories
  \begin{equation*}
    \xymatrix{ \map_B(p,q)  \pbexcursion \ar@{->>}[d] \ar[r] & \map(E,F) \ar@{->>}[d]^-{\map(E,q)} &  &   \hom_B(p,q) \pbexcursion \ar[r] \ar@{->>}[d] & \hom(E,F) \ar@{->>}[d]^-{\hom(E,q)} \\ 
       \Del^0 \ar[r]_-p & \map(E,A) & &     \catone \ar[r]_-p & \hom(E,A)}
   \end{equation*}
The vertices of $\map_B(p,q)$ and the objects of $\hom_B(p,q)$ are exactly the functors from $p$ to $q$ in $\lcat{K}/B$, i.e., commutative triangles over $B$. In particular, $(\lcat{K}/B)_2$ and $\lcat{K}_2/B$ have the same underlying 1-category. However, their 2-cells differ. 

Given a parallel pair of 1-cells \[ \xymatrix@=1.5em{ E \ar@{->>}[dr]_p \ar@/^1ex/[rr]^f \ar@/_1ex/[rr]_g & & F \ar@{->>}[dl]^q \\ & B}\] a 2-cell from $f$ to $g$ in

\begin{itemize}
\item $\lcat{K}_2\slice B$ is a homotopy class of 1-simplices $f \to g$ in $\map(E,F)$ that whisker with $q$ to the homotopy class of the degenerate 1-simplex on $p$ in $\map(E,B)$.
    \item $(\lcat{K}\slice B)_2$ is a homotopy class of 1-simplices $f \to g$ in the fibre of of the isofibration $\map(E,q) \colon \map(E,F) \tfib \map(E,B)$ over the vertex $p\in \map(E,B)$ under homotopies which are also constrained to that fibre.
  \end{itemize}
The distinction is that the notion of homotopy involved in the description of 2-cells in $(\lcat{K}\slice B)_2$ is more refined (identifies fewer simplices) than that given for 2-cells in $\lcat{K}_2\slice B$. Each homotopy class representing a 2-cell in $\lcat{K}_2\slice B$ may actually split into a number of distinct homotopy classes representing 2-cells in $(\lcat{K}\slice B)_2$.
\end{defn}

  For any object $B$ in an $\infty$-cosmos $\lcat{K}$ there exists a canonical comparison 2-functor $(\lcat{K}\slice B)_2\to\lcat{K}_2\slice B$. This acts identically on objects $p\colon E\tfib B$ and $q\colon F\tfib B$, while its action $\ho(\map_B(p,q))\to \hom_B(p,q)$ on hom-categories is induced by applying the universal property of the defining pullback square for the hom-category $\hom_B(p,q)$ of $\lcat{K}_2\slice B$ to the square obtained by applying the homotopy category functor $\ho$ to the defining pullback square for the mapping quasi-category $\map_B(p,q)$ of $\lcat{K}\slice B$.  The arguments leading to Proposition~\refI{prop:slice-smothering-2-functor} generalise immediately to the $\infty$-cosmos $\lcat{K}$ to demonstrate that $(\lcat{K}\slice B)_2\to\lcat{K}_2\slice B$ is a \emph{smothering 2-functor}. A functor is  \emph{smothering} if it is surjective on objects, locally surjective on arrows, and conservative (see Definition \refI{defn:smothering}). A 2-functor is \emph{smothering} if it is surjective on objects and locally smothering (see Definition~\refI{defn:smothering-2-functor}).

\begin{prop}\label{prop:smothering-slice-comparison} The canonical 2-functor $(\lcat{K}/B)_2 \to \lcat{K}_2/B$ that acts identically on underlying 1-categories and acts via the quotient map $\ho\map_B(p,q) \to \hom_B(p,q)$ on hom-categories is a smothering 2-functor.\qed
\end{prop}

The ramifications of Proposition~\ref{prop:smothering-slice-comparison} are that for many purposes it makes no difference whether we work in $(\lcat{K}/B)_2$ or in $\lcat{K}_2/B$. The following corollary summaries a few particular instantiations of this principle.

\begin{cor}\label{cor:slice-2-cat-comparison} Fix an $\infty$-cosmos $\lcat{K}$ and an object $B$. 
\begin{enumerate}[label=(\roman*)]
\item A pair of isofibrations over $B$ are equivalent as objects in $(\lcat{K}/B)_2$ if and only if they are equivalent in $\lcat{K}_2\slice B$.
\item\label{itm:slice-2-cat-comparison-equiv} A functor over $B$ is an equivalence in $(\lcat{K}/B)_2$ if and only if it is an equivalence in $\lcat{K}_2\slice B$ if and only if it is an equivalence in $\lcat{K}$.
\item\label{itm:slice-2-cat-comparison-iso-class} A parallel pair of functors over $B$ are isomorphic in $(\lcat{K}/B)_2$ if and only if they are isomorphic in $\lcat{K}_2\slice B$.
\item\label{itm:slice-2-cat-comparison-groupoidal} An object $p \colon E \tfib B$ is groupoidal, in the sense that any 2-cell with codomain $p$ is invertible, in $(\lcat{K}\slice B)_2$ if and only if it is groupoidal in $\lcat{K}_2\slice B$.
\item\label{itm:slice-2-cat-comparison-adj} A functor over $B$ admits a right or left adjoint in $(\lcat{K}\slice B)_2$ if and only if it admits the corresponding adjoint in $\lcat{K}_2/B$.
\end{enumerate}
\end{cor}

Note that all of these results, with the exception of the final clause of \ref{itm:slice-2-cat-comparison-equiv}, are formal consequences of the fact that $(\lcat{K}/B)_2 \to \lcat{K}_2/B$ is a smothering 2-functor.

\begin{proof}
The canonical identity-on-underlying-1-categories 2-functor $(\lcat{K}/B)_2 \to \lcat{K}_2/B$ preserves equivalences, isomorphic 2-cells, and adjunctions. The ``reflection'' part of these assertions follows in each case from the fact that  $(\lcat{K}/B)_2 \to \lcat{K}_2/B$ is a smothering 2-functor: smothering 2-functors reflect equivalence and equivalences and reflect and create  2-cell isomorphisms.
A proof of the final assertion in \ref{itm:slice-2-cat-comparison-equiv}, that a functor over $B$ defines an equivalence in the slice 2-categories if and only if it defines an equivalence in $\lcat{K}_2$, can be found in Lemma~\refI{lem:proj-is-1-conservative}, and a proof of \ref{itm:slice-2-cat-comparison-adj} can be found in Lemma~\refI{lem:missed-lemma}.
\end{proof}

\begin{defn}[fibred equivalence] In an $\infty$-cosmos $\tcat{K}$, we say that two isofibrations with codomain $B$ are \emph{equivalent over $B$} if the equivalent conditions of Corollary~\ref{cor:slice-2-cat-comparison}\ref{itm:slice-2-cat-comparison-equiv} are satisfied. By Proposition~\refIV{prop:equiv.are.weak.equiv}, recalled in the previous section, this is equivalent to asking that there is an equivalence between them as objects in the sliced $\infty$-cosmos $\lcat{K}/B$.
\end{defn}

\begin{defn}[fibred adjunction] In \cite{RiehlVerity:2012tt}, we reserved the term \emph{fibred adjunction} for an adjunction in the homotopy 2-category $(\tcat{K}/B)_2$ of a sliced $\infty$-cosmos. However, on account of Corollary~\ref{cor:slice-2-cat-comparison}\ref{itm:slice-2-cat-comparison-adj} we also apply this appellation to adjunctions in $\tcat{K}_2/B$, as the unit and counit 2-cells in here can always be lifted to unit and counit 2-cells in $(\tcat{K}/B)_2$.
\end{defn}

\begin{rmk}\label{rmk:adjunctions-pullback}
In particular, a functor $f \colon A \to B$ induces a functor of sliced $\infty$-cosmoi $f^* \colon \lcat{K}/B \to \lcat{K}/A$ that carries an isofibration over $B$ to its (simplicial) pullback, an isofibration over $A$. The induced 2-functor $f^* \colon (\lcat{K}/B)_2 \to (\lcat{K}/A)_2$, like any 2-functor, preserves adjunctions. By contrast, there is no ``pullback 2-functor'' $f^* \colon \lcat{K}_2/B \to \lcat{K}_2/A$. However, on account of Corollary~\ref{cor:slice-2-cat-comparison}\ref{itm:slice-2-cat-comparison-adj} we can say nonetheless assert that fibred adjunctions may be pulled back along any functor.  For a discussion of this point at the level of homotopy 2-categories, without reference to the simplicially enriched universal property of pullbacks, see \S\refIV{ssec:smothering}.
\end{rmk}

\subsection{Simplicial limits modeling comma \texorpdfstring{$\infty$}{infinity}-categories}\label{ssec:limits}

The homotopy 2-categories of $\infty$-cosmoi, including those of the slices of other $\infty$-cosmoi, are \emph{abstract homotopy 2-categories}: that is, strict 2-categories admitting comma objects and iso-comma objects of a particular weak variety discussed abstractly in \S\refIV{ssec:htpy-2-cat-is-abstract}.  At the level of the $\infty$-cosmos, these weak 2-limits are constructed as particular weighted limits, an up-to-isomorphism limit notion. The constructions resemble familiar homotopy limits but the term ``weighted limit'' is more precise: for instance, the difference between the construction of the comma object and of the iso-comma object is the choice of a non-invertible or invertible interval. 

Our development of the theory of cartesian fibrations in \cite{RiehlVerity:2015fy} is entirely 2-categorical, taking place in an abstract homotopy 2-category. Here, for simplicity, we frequently take advantage of extra strictness provided by $\infty$-cosmos-level models of weak 2-limit constructions, which commute up to isomorphism (rather than simply isomorphic 2-cell) and preserve the chosen class of isofibrations. In the present paper, we can do without iso-commas entirely. As noted in \S\refIV{ssec:pullback}, iso-commas formed from a cospan in which at least one leg is an isofibration are equivalent to the pullbacks of \ref{qcat.ctxt.cof.def}\ref{qcat.ctxt.cof:a}. As we won't make use of the weak 2-universal properties of these pullbacks (which are somewhat less well behaved than the closely related iso-commas), we will typically refer to them as \emph{simplicial pullbacks} here, to avoid confusion with the terminology used in previous papers in this series. Our aim here is to achieve an expository simplification: simplicial pullbacks, i.e., ordinary strict pullbacks satisfying a simplicially enriched version of the usual universal property, are quite familiar.

\begin{rec}[comma $\infty$-categories]\label{rec:comma}
Given a pair of functors  $f\colon B\to A$ and $g\colon C\to A$ in an $\infty$-cosmos $\lcat{K}$, their comma object, which we call their \emph{comma $\infty$-category}, may be constructed by the following simplicial pullback, formed in $\lcat{K}$:
  \begin{equation}\label{eq:comma-as-simp-pullback}
    \xymatrix@=2.5em{
      {f\comma g}\pbexcursion \ar[r]\ar@{->>}[d]_{(p_1,p_0)} &
      {\Del^1\pwr A} \ar@{->>}[d]^{(p_1,p_0)} \\
      {C\times B} \ar[r]_-{g\times f} & {A\times A}
    }
  \end{equation}
The data of the simplicial pullback defines a canonical square 
\begin{equation}\label{eq:comma-square}
    \xymatrix@=10pt{
      & f \downarrow g \ar[dl]_{p_1} \ar[dr]^{p_0} \ar@{}[dd]|(.4){\phi}|{\Leftarrow}  \\ 
      C \ar[dr]_g & & B \ar[dl]^f \\ 
      & A}
  \end{equation}
in the homotopy 2-category $\lcat{K}_2$ with the property that for any object $X$,  the induced comparison functor of hom-categories
  \begin{equation*}
    \hom(X,f\comma g)\longrightarrow\hom(X,f)\comma\hom(X,g)
  \end{equation*}
is \emph{smothering}: surjective on objects, locally surjective on arrows, and conservative.

Explicitly, this weak universal property supplies us with three operations in the homotopy 2-category.
  \begin{enumerate}[label=(\roman*)]
    \item (1-cell induction) Given a \emph{comma cone} $\alpha \colon fb \To gc$
\[    \vcenter{ \xymatrix@=10pt{
      & X \ar[dl]_{c} \ar[dr]^{b} \ar@{}[dd]|(.4){\alpha}|{\Leftarrow}  \\ 
      C \ar[dr]_g & & B \ar[dl]^f \\ 
      & A}} \quad =\quad \vcenter{     \xymatrix@=10pt{ & X \ar@{-->}[d]^{a} \ar@/^2ex/[ddr]^b \ar@/_2ex/[ddl]_c \\ 
      & f \downarrow g \ar[dl]_{p_1} \ar[dr]^{p_0} \ar@{}[dd]|(.4){\phi}|{\Leftarrow}  \\ 
      C \ar[dr]_g & & B \ar[dl]^f \\ 
      & A}}
\] 
over the pair of functors $f$ and $g$, there exists a 1-cell $a\colon X\to f\comma g$ so that $p_0a = b$, $p_1a = c$, and $\alpha = \phi a$.
    \item (2-cell induction) Given a pair of functors $a,a'\colon X\to f\comma g$ and a pair of 2-cells $\tau_0\colon p_0a \Rightarrow p_0a'$ and $\tau_1\colon p_1a\Rightarrow p_1a'$ which are compatible in the sense that $\phi a'\cdot f\tau_0 = g\tau_1\cdot \phi a$, then there exists a 2-cell $\tau\colon a\Rightarrow a'$ with $p_0\tau=\tau_0$ and $p_1\tau=\tau_1$.
    \item (conservativity) Any 2-cell $\tau\colon a\Rightarrow a'\colon X\to f\comma g$ with the property that the whiskered 2-cells $p_0\tau$ and $p_1\tau$ are both isomorphisms is also an isomorphism.
  \end{enumerate}
\end{rec}

We refer to \eqref{eq:comma-square} as a \emph{comma square} and $C \xleftarrow{p_1} f \comma g \xrightarrow{p_0} B$ as a \emph{comma span}. Note that, by construction, the map $(p_1,p_0) \colon f \comma g \tfib C \times B$ is an isofibration.

\begin{rec}\label{rec:ess.unique.1-cell.ind}
As discussed in \S 3.5 of \cite{RiehlVerity:2012tt}, a parallel pair of functors $a,a' \colon X \to f \comma g$ are isomorphic over $C \times B$ if and only if $a$ and $a'$  both enjoy the same defining properties as 1-cells induced by the weak 2-universal property of $f\comma g$, i.e., they satisfy $p_0a= p_0a'$, $p_1a = p_1a'$, and $\phi a = \phi a'$. That is, 2-cells of the form displayed on the left
\[ \vcenter{\xymatrix@=10pt{
      & X \ar[dl]_{c} \ar[dr]^{b} \ar@{}[dd]|(.4){\alpha}|{\Leftarrow}  \\ 
      C \ar[dr]_g & & B \ar[dl]^f \\ 
      & A}} \qquad \leftrightsquigarrow \qquad \vcenter{     \xymatrix@=10pt{ & X \ar@{->}[dd]^{a} \ar[dr]^b \ar[dl]_c \\ C & & B \\ & f \comma g \ar@{->>}[ul]^{p_1} \ar@{->>}[ur]_{p_0}}}\]
stand in bijection with isomorphism classes of maps of spans, as displayed on the right. Note however that the isomorphic 2-cells between a parallel pair of isomorphic spans are not  typically unique.
\end{rec}

\begin{rec}[special cases of commas]
When $f$ or $g$ is an identity, we write $A \comma g$ or $f \comma A$, respectively, for the comma object. In the case where both $f$ and $g$ are identities, we write $A^\cattwo$ for $A \comma A$ because this object is a \emph{weak} $\cattwo$-\emph{cotensor}, in the sense introduced in \S\refI{subsec:weak-2-limits}.

Their weak universal properties in the homotopy 2-category only characterize these objects up to equivalence, but we frequently make use of the preferred construction of $A^\cattwo$ as the cotensor $\Del^1 \pwr A$. This allows us to make use of the fact that $\Del^1 \pwr - \colon \lcat{K} \to \lcat{K}$ is a functor of $\infty$-cosmoi, preserving isofibrations and simplicial limits.
\end{rec}

\begin{lem}\label{lem:equiv-invariance-of-commas} Given a pair of cospans connected by equivalences
\[ \xymatrix{ C' \ar[r]^{g'} \ar[d]_{c}^{\rotatebox{90}{$\labelstyle\sim$}} & A' \ar[d]_a^{\rotatebox{90}{$\labelstyle\sim$}} & B' \ar[l]_{f'} \ar[d]^b_{\rotatebox{90}{$\labelstyle\sim$}} \\ C \ar[r]_g & A & B \ar[l]^f}\] the induced functor $f' \comma g' \we f \comma g$ between the comma constructions is an equivalence, commuting, via the legs of the comma spans, with the equivalence $c \times b \colon C' \times B' \we C \times B$. 
\end{lem}
\begin{proof}
Lemma~\refI{lem:comma-obj-maps} shows that if the maps $a$, $b$, and $c$ are trivial fibrations, then so is the functor $f' \comma g' \trvfib f \comma g$ induced between the pullbacks \eqref{eq:comma-as-simp-pullback}. The general result follows from Ken Brown's lemma. Lemma~\refIV{lem:Brown.fact} shows that any map can be factored as a fibration preceded by an equivalence that defines a section of a trivial fibration; of course, if the original map is an equivalence, then the right factor is a trivial fibration. This construction, making use of various simplicial limits, is functorial, and so induces a corresponding factorization of the induced functor $f' \comma g' \to f \comma g$. Lemma~\refI{lem:comma-obj-maps} implies that the right factor in this factorization is a trivial fibration, and so, by the 2-of-3 property, the composite functor is an equivalence.
\end{proof}

\begin{lem}\label{lem:commas-pullback}
Consider functors $f\colon B\to A$, $g\colon C\to A$, $h\colon B'\to B$, and $k\colon C'\to C$ in an $\infty$-cosmos. Then the preferred simplicial models of comma $\infty$-categories are related by the following simplicial pullbacks.
\[ \xymatrix{ f \comma gk \ar@{->>}[d]_{p_1}\pbexcursion \ar[r] & f \comma g \ar@{->>}[d]^{p_1} & fh \comma g \ar@{->>}[d]_{p_0} \pbexcursion \ar[r] & f \comma h \ar@{->>}[d]^{p_0} \\ C' \ar[r]_k & C & B' \ar[r]_h & B}\]
\end{lem}
\begin{proof}
This follows easily from the standard composition and cancellation results for simplicial pullback squares and rectangles. Compare with Lemma~\refIV{lem:restricted-commas-pullback}.
\end{proof}

It is easy to show that any isofibration $(q,p) \colon E \tfib C\times B$ equipped with a 2-cell satisfying the weak universal property of the comma $\infty$-category for the cospan $C \xrightarrow{g} A \xleftarrow{f} B$ must be equivalent over $C \times B$ to the object $f \comma g$ constructed in \eqref{eq:comma-as-simp-pullback}; see   Lemma~\refI{lem:unique-weak-2-limits}. The following lemma proves the converse: that any $\infty$-category that is equivalent to a comma $\infty$-category via an equivalence that commutes with the legs of the comma span must enjoy the same weak universal property in the homotopy 2-category. 

\begin{lem}\label{lem:equiv-to-comma} Suppose $(q,p) \colon E \tfib C \times B$ is an isofibration and $C \xrightarrow{g} A \xleftarrow{f} B$ is a pair of functors so that $E$ is equivalent to $f \comma g$ over $C \times B$. Then the composite of the equivalence $E \to f \comma g$ with the canonical comma 2-cell displays $E$ as a weak comma object for the functors $f$ and $g$. 
\end{lem}

\begin{proof}
For any $X$, the canonical functor 
\[ \hom(X,E) \to \hom(X,f \comma g) \to \hom(X,f) \comma \hom(X,g)\] is the composite of an equivalence with a smothering functor, and as such is immediately full and conservative. It remains only to show that the composite, which is clearly essentially surjective on objects, is in fact surjective on objects. 

To this end, observe that any object in $\hom(X,f) \comma \hom(X,g)$ has a preimage in $\hom(X,f \comma g)$, which is isomorphic, via some isomorphism projecting to an identity in $\hom(X,C \times B)$ to an object in the image of $\hom(X,E) \to \hom(X,f \comma g)$. This follows from Corollary~\ref{cor:slice-2-cat-comparison}\ref{itm:slice-2-cat-comparison-equiv} which tells us that any equivalence between the domains of isofibrations over a common base can be promoted to an equivalence in the slice 2-category over that base, provided at least one of the maps is fibred.  But any pair of objects in $\hom(X,f \comma g)$, which are isomorphic over an identity in $\hom(X,C \times B)$, have the same image in $\hom(X,f) \comma \hom(X,g)$. Thus $\hom(X,E) \to \hom(X,f) \comma \hom(X,g)$ is surjective on objects, as desired.
\end{proof}

\begin{ex}\label{ex:cone-comma-squares}
For any pair of finitely presented simplicial sets $X$ and $Y$, Proposition~\refI{prop:join-fatjoin-equiv} supplies a map $X \diamond Y \to X \join Y$, under $X \coprod Y$, that is a weak equivalence in the Joyal model structure. It follows that for any object $A$ in an $\infty$-cosmos $\lcat{K}$, the induced map $(X \join Y) \pwr A \to (X \diamond Y) \pwr A$ on cotensors is an equivalence of $\infty$-categories over $(X \pwr A) \times (Y \pwr A)$. As observed in the proof of Lemma~\refI{lem:cone-equiv-fatcone}, $(X \diamond \Delta^0)\pwr A$ is isomorphic to $(X \pwr A) \comma \Delta$, where $\Delta \colon A \to (X \pwr A)$ is the constant diagram functor, as both of these objects are defined by the same pullback in $\lcat{K}$. Using the common notation $X^{\triangleright} \defeq X \join \Delta^0$ and $X^{\triangleleft} \defeq \Delta^0 \join X$, Lemma~\ref{lem:equiv-to-comma} supplies comma squares
\[ \    \xymatrix@=10pt{ & \save[]{ {X^{\triangleright}}\pwr A}  \ar[dl]_-{\pi} \ar[dr]^-{\pi} \ar@{}[dd]|{\Leftarrow} \restore & & & & \save[]{ {X^{\triangleleft}}\pwr A}  \ar[dl]_-{\pi} \ar[dr]^-{\pi} \ar@{}[dd]|{\Leftarrow} \restore\\ A \ar[dr]_-\Delta & & \save[]{ X \pwr A}  \ar@{=}[dl] \restore&& \save[]{ X \pwr A} \ar@{=}[dr] \restore & & A \ar[dl]^-\Delta \\ & \save[]{ X \pwr A} \restore & & & & \save[]{ X \pwr A} \restore}\] 
under the spans defined by restricting to the diagrams on $X$ or on the cone point.

In the case where the indexing simplicial set $X$ is (the nerve of) a small category, these comma squares arise from the cocomma squares
\[    \xymatrix{ X \ar[d]_{!} \ar@{=}[r] \ar@{}[dr]|{\Downarrow} & X \ar[d] & X \ar[r]^{!} \ar@{=}[d] \ar@{}[dr]|{\Downarrow} & \catone \ar[d] \\ \catone \ar[r] & X^\triangleright & X \ar[r] & X^\triangleleft}\]
upon application of the 2-functor $(-) \pwr A \colon \Cat_2\op \to \lcat{K}_2$.
\end{ex}

\begin{prop}\label{prop:induced-2-functor} A functor $F \colon \lcat{K} \to \lcat{L}$ of $\infty$-cosmoi induces a 2-functor $F_2 \colon \lcat{K}_2 \to \lcat{L}_2$ between their homotopy 2-categories that preserves  adjunctions, equivalences, isofibrations, trivial fibrations, products, and comma objects.
\end{prop}
\begin{proof}
Any 2-functor preserves  adjunctions and equivalences. Preservation of isofibrations and products  are direct consequences of the hypotheses in Definition~\ref{defn:qcat-ctxt-functor}; recall that the class of trivial fibrations in this intersection of the classes of isofibrations and equivalences. Preservation of commas follows from the construction of \eqref{eq:comma-as-simp-pullback}, which is preserved by a functor of $\infty$-cosmoi, and the observation made before Lemma~\ref{lem:equiv-to-comma} that all comma objects over the same cospan are equivalent.
\end{proof}

\subsection{Cartesian fibrations and groupoidal cartesian fibrations}\label{ssec:cartesian}

Cartesian fibrations and groupoidal cartesian fibrations are defined in \S\refIV{sec:cartesian} to be certain isofibrations in an abstract homotopy 2-category. Here we consider only cartesian fibrations and groupoidal cartesian fibrations in the homotopy 2-category $\lcat{K}_2$ of an $\infty$-cosmos, which we may as well refer to as cartesian fibrations and groupoidal cartesian fibrations in $\lcat{K}$.

\begin{defn}[cartesian 2-cells]
A 2-cell $\chi \colon e' \Rightarrow e \colon A \to E$ in $\lcat{K}_2$ is \emph{cartesian} for an isofibration $p \colon E \tfib B$ if and only if
\begin{enumerate}[label=(\roman*)]
  \item\label{itm:weak.cart.i} (induction) for any pair of 2-cells $\tau \colon e'' \Rightarrow e$ and $\gamma\colon pe''\Rightarrow pe'$ with $p \tau = p\chi \cdot \gamma$ there is some $\overline{\gamma} \colon e'' \Rightarrow e'$ with $p\overline\gamma = \gamma$ ($\bar\gamma$ {\em lies over\/} $\gamma$) and the property that $\tau = \chi \cdot \bar\gamma$.
  \item\label{itm:weak.cart.ii} (conservativity) for any 2-cell $\gamma \colon e' \Rightarrow e'$ if $\chi \cdot \gamma = \chi$ and $p\gamma$ is an identity then $\gamma$ is an isomorphism.
\end{enumerate}
\end{defn}

All isomorphisms with codomain $E$ are $p$-cartesian. The class of $p$-cartesian 2-cells is stable under composition and left cancellation (Lemmas~\refIV{lem:cart-arrows-compose} and~\refIV{lem:cart-arrows-cancel}).

\begin{defn}[cartesian fibration]
An isofibration  $p\colon E\tfib B$ is a {\em cartesian fibration\/} if and only if:
\begin{enumerate}[label=(\roman*)]
  \item Every 2-cell $\alpha \colon b \To pe$ has a $p$-cartesian lift $\chi_\alpha \colon e' \To e$:
    \begin{equation}\label{eq:2-cell.to.lift}
      \vcenter{
        \xymatrix@=3em{
          {} \ar@{}[dr]|(0.68){\Uparrow\alpha} & {E}\ar@{->>}[d]^p \\
          {A}\ar[ur]^{e} \ar[r]_b & {B}
      }}
      \mkern40mu=\mkern40mu
      \vcenter{
        \xymatrix@=3em{
          {} & {E}\ar@{->>}[d]^p \\
          {A}\ar@/^1.6ex/[ur]^-{e}
          \ar@/_1.6ex/[ur]_-*!/^0.3pc/{\scriptstyle e'} 
          \ar@{}[ur]|(0.55){\Uparrow\chi_\alpha}\ar[r]_b & {B}
      }}
        \end{equation}
  \item The class of $p$-cartesian 2-cells for $p$ is closed under pre-composition by all 1-cells.
  \end{enumerate}
\end{defn}

Any functor $p \colon E \to B$ induces functors between comma $\infty$-categories 
  \begin{equation*}
 \vcenter{\xymatrix@C=0.8em@R=1.2em{
    & {E}\ar[d]^-{i} & \\
    & {B\comma p}\ar@{->>}[dl]_{p_1}\ar@{->>}[dr]^{p_0} & \\
    {E}\ar[rr]_{p} && {B}
    \ar@{} "2,2";"3,2" |(0.6){\Leftarrow\phi}
  }} = 
  \vcenter{\xymatrix@C=0.8em@R=1.2em{
    & {E}\ar@{=}[dl]\ar@{->>}[dr]^{p} & \\
    {E}\ar[rr]_{p} && {B}
    \ar@{} "1,2";"2,2" |(0.6){=}
  }}   
  \mkern50mu
  \vcenter{\xymatrix@C=0.8em@R=1.2em{
    & {E^\cattwo}\ar[d]^-{k} & \\
    & {B\comma p}\ar@{->>}[dl]_{p_1}\ar@{->>}[dr]^{p_0} & \\
    {E}\ar[rr]_{p} && {B}
    \ar@{} "2,2";"3,2" |(0.6){\Leftarrow\phi}
  }} = 
  \vcenter{\xymatrix@C=0.8em@R=1.2em{
    & {E^\cattwo}\ar@{->>}_{q_1}[dl]\ar@{->>}[dr]^{pq_0} & \\
    {E}\ar[rr]_{p} && {B}
    \ar@{} "1,2";"2,2" |(0.6){\Leftarrow p\psi}
  }}
  \end{equation*}
 that are well-defined up to isomorphism over $E \times B$. These functors are used to provide an alternate characterization of cartesian fibrations:

{
\renewcommand{\thethm}{\refIV{thm:cart.fib.chars}}
\begin{thm} For an isofibration $p \colon E \tfib B$, the following are equivalent:
  \begin{enumerate}[label=(\roman*)]
    \item\label{itm:cart.fib.chars.i} $p$ is a cartesian fibration.
    \item\label{itm:cart.fib.chars.ii} The functor $i\colon E\to B\comma p$ admits a right adjoint which is fibred over $B$.
\[
      \xymatrix@R=2em@C=3em{
        {B\comma p}\ar@{->>}[dr]_{p_0} \ar@/_0.6pc/[]!R(0.5);[rr]_{r}^{}="a" & &
        {E}\ar@{->>}[dl]^{p} \ar@/_0.6pc/[ll]!R(0.5)_{i}^{}="b" 
        \ar@{}"a";"b"|{\bot} \\
        & B &
      }
    \]
    \item\label{itm:cart.fib.chars.iii} The functor $k\colon E^\cattwo\to B\comma p$ admits a right adjoint right inverse, i.e., with isomorphic counit.
\[
    \xymatrix@C=6em{
      {B\comma p}\ar@/_0.8pc/[]!R(0.6);[r]!L(0.45)_{\bar{r}}^{}="u" &
      {E^\cattwo}\ar@/_0.8pc/[]!L(0.45);[l]!R(0.6)_{k}^{}="t"
      \ar@{}"u";"t"|(0.6){\bot}
    }
  \]
  \end{enumerate}
\end{thm}
\addtocounter{thm}{-1}
}

\begin{defn}[groupoidal cartesian fibrations]
  An isofibration $p\colon E\tfib B$  is a {\em group\-oid\-al cartesian fibration\/} if and only it is a cartesian fibration and it is groupoidal as an object of the slice $\lcat{K}_2\slice B$.
\end{defn}

A groupoidal cartesian fibration is a cartesian fibration whose fibers are groupoidal $\infty$-categories. As a consequence of Theorem~\ref{thm:gpd.cart.fib.chars}\ref{itm:gpd.cart.fib.chars.ii} below, if $p \colon E \tfib B$ is a groupoidal cartesian fibration, then all 2-cells with codomain $E$ are $p$-cartesian.

Propositions~\refIV{prop:groupoidal-ext-char} and \refIV{prop:iso-groupoidal-char} combine to give the following alternate characterizations of groupoidal cartesian fibrations.

\begin{thm}\label{thm:gpd.cart.fib.chars}For an isofibration $p \colon E \tfib B$, the following are equivalent:
  \begin{enumerate}[label=(\roman*)]
    \item\label{itm:gpd.cart.fib.chars.i} $p$ is a groupoidal cartesian fibration.
    \item\label{itm:gpd.cart.fib.chars.ii} Every 2-cell $\alpha\colon b\Rightarrow pe\colon X\to B$ has an essentially unique lift $\chi\colon e'\Rightarrow e\colon X\to E$, where the uniqueness is up to a (non-unique) isomorphic 2-cell over an identity.
     \item\label{itm:gpd.cart.fib.chars.iii} The functor $k\colon E^\cattwo\to B\comma p$ is an equivalence.
\end{enumerate}
\end{thm}

Theorems~\refIV{thm:cart.fib.chars} and \ref{thm:gpd.cart.fib.chars} have an important corollary:

\begin{cor} $\quad$
\begin{enumerate}[label=(\roman*)]
\item Any isofibration that is equivalent to a (groupoidal) cartesian fibration is a (group\-oid\-al) cartesian fibration.
\item Cartesian fibrations and groupoidal cartesian fibrations are preserved by functors of $\infty$-cosmoi.
\end{enumerate}
\end{cor}
\begin{proof}
For (i), it follows easily from Lemma~\ref{lem:equiv-invariance-of-commas} that the functors $k \colon E^\cattwo \to B \comma p$ induced from equivalent isofibrations $p \colon E \tfib B$ are equivalent. By the 2-of-3 property, the notion of equivalence is equivalence invariant, and so Theorem~\ref{thm:gpd.cart.fib.chars}\ref{itm:gpd.cart.fib.chars.iii} proves this result in the groupoidal case. For general cartesian fibrations, the existence of adjoints in a 2-category is invariant under equivalence, and so  Theorem~\refIV{thm:cart.fun.chars}.\ref{itm:cart.fib.chars.iii} implies that an isofibration that is equivalent to a  cartesian fibration is a cartesian fibration. 

The claim in (ii) follows easily from a combination of Proposition~\ref{prop:induced-2-functor} and Theorem~\refIV{thm:cart.fib.chars} or Theorem~\ref{thm:gpd.cart.fib.chars}, as appropriate.
\end{proof}

Cartesian fibrations are stable under composition (Proposition~\refIV{prop:cart-fib-comp}) and groupoidal cartesian fibrations are additionally stable under left cancellation (Lemma~\refIV{lem:gpd-cart-cancel}).

\begin{defn}\label{defn:cartesian-functor}  A   commutative square 
\[ \xymatrix{ F \ar[r]^g \ar@{->>}[d]_q & E \ar@{->>}[d]^p \\ A \ar[r]_f & B}\]between a pair of cartesian fibrations  $q \colon F \tfib A$  and $p \colon E \tfib B$
 defines  a \emph{cartesian functor} if and only if $g$ preserves cartesian 2-cells: i.e., if whiskering with $g$ carries $q$-cartesian 2-cells to  $p$-cartesian 2-cells.
\end{defn}

\begin{prop}[{\refIV{prop:cart-fib-pullback} and \refIV{cor:groupoidal-pullback}}]\label{prop:cart-fib-pullback}
Consider a simplicial pullback
\[ \xymatrix{ F \pbexcursion \ar@{->>}[d]_q \ar[r]^g & E \ar@{->>}[d]^p \\ A \ar[r]_f & B}\] in $\lcat{K}$. If $p \colon E \tfib B$ is a (groupoidal) cartesian fibration, then $q \colon F \tfib A$ is a (groupoidal) cartesian fibration and the pullback square defines a cartesian functor. 
\end{prop}
\begin{proof}
The simplicial pullbacks in $\lcat{K}$ are examples of the pullbacks in the homotopy 2-category $\tcat{K}_2$ of the sort  considered in Proposition~\refIV{prop:cart-fib-pullback} and Corollary~\refIV{cor:groupoidal-pullback}.
\end{proof}


\begin{ex} Example~\refIV{ex:domain-fibration} shows that the domain-projection functor $p_0 \colon A^\cattwo \tfib A$ from an arrow $\infty$-category is a cartesian fibration. It follows from Proposition~\ref{prop:cart-fib-pullback} that the domain-projection $p_0 \colon f \comma A \tfib B$ defines a cartesian fibration. Interpreting the theory just developed in the dual 2-category $\lcat{K}_2\co$, reversing the 2-cells but not the 1-cells, we see also that the codomain projection functors $p_1 \colon E^\cattwo \tfib E$ and $p_1 \colon A \comma g \tfib C$ are cocartesian fibrations. Indeed, in the next section we shall show that the projection $p_0\colon f\comma g\tfib B$ (resp.\ $p_1\colon f\comma g\tfib C$) from any comma is a cartesian (resp.\ cocartesian) fibration.

For any point $b \colon 1 \to B$ of $B$, Example~\refIV{ex:groupoidal-representable} shows that $p_0 \colon B \comma b \tfib B$ is a groupoidal cartesian fibration. Dually, $p_1 \colon b \comma B \tfib B$ is a groupoidal cocartesian fibration.
\end{ex}

The Yoneda lemma, Theorem~\refIV{thm:yoneda}, supplies an equivalence between  the underlying quasi-category $\map_B (b \colon 1 \to B, p \colon E \tfib B)$  of the fiber of a cartesian fibration $p$ over a point $b$, and the quasi-category of functors between the cartesian fibration represented by $b$ and $p$. In this paper, we'll require only the special case where $p$ is a groupoidal cartesian fibration.

{
\renewcommand{\thethm}{\refIV{cor:groupoidal-yoneda}}
\begin{cor}[Yoneda lemma] Given any groupoidal cartesian fibration $p \colon E \tfib B$ and any point $b \colon 1 \to B$, restriction along the terminal object $t \colon 1 \to B\comma b$ induces an equivalence of quasi-categories
\[ \map_B(p_0 \colon B \comma b \tfib B, p \colon E \tfib B) \simeq \map_B(b \colon 1 \to B, p \colon E \tfib B).\]
\end{cor}
  \addtocounter{thm}{-1}
}




\section{Modules between \texorpdfstring{$\infty$}{infinity}-categories}\label{sec:modules}

As the name suggests, a \emph{module} from an $\infty$-category $A$ to an $\infty$-category $B$ is an $\infty$-category $E$ equipped with an isofibration $E \tfib A \times B$ with groupoidal fibers that satisfies conditions that can be informally summarized by saying that $A$ acts on the left and $B$ acts on the right. The paradigmatic example is given by the arrow $\infty$-category construction $(p_1,p_0) \colon A^\cattwo \tfib A \times A$, which defines a module from $A$ to itself. We have shown that the domain projection functor $p_0 \colon A^\cattwo \tfib A$ is a cartesian fibration and that the codomain projection functor $p_1 \colon A^\cattwo \tfib A$ is a cocartesian; this is the sense in which $A$ acts on the left and on the right of $A^\cattwo$. But really more is true: as observed in \refIV{obs:domain-cart-lifts}, $p_0$-cartesian lifts can be chosen to lie in the fibers of $p_1$ and similarly that $p_1$-cartesian lifts can be chosen to live in the fibers of $p_0$. The fact that  $(p_1,p_0) \colon A^\cattwo \tfib A \times A$ has groupoidal fibers, or more precisely, is a groupoidal object in the slice $\infty$-cosmos over $A \times A$, is a consequence of conservativity of 2-cell induction.

\subsection{Modules between \texorpdfstring{$\infty$}{infinity}-categories}

Fix an ambient $\infty$-cosmos $\lcat{K}$.

\begin{defn}  A \emph{module} $E$ \emph{from $A$ to $B$} is given by an isofibration $(q,p) \colon E \tfib A \times B$ to the product of $A$ and $B$ so that:
  \begin{enumerate}[label=(\roman*)]
\item\label{itm:cartesian-on-the-right} $(q,p) \colon E \tfib A \times B$ is a cartesian fibration in  $(\lcat{K}/A)_2$.
\item\label{itm:cocartesian-on-the-left}  $(q,p) \colon E \tfib A \times B$ is a cocartesian fibration in $(\lcat{K}/B)_2$.
\item\label{itm:groupoidal-fibers} $(q,p) \colon E \tfib A \times B$ is groupoidal as an object in $\lcat{K}/A \times B$.
\end{enumerate}
\end{defn}

\begin{rmk}
By Definition~\ref{defn:groupoidal-object}, condition \ref{itm:groupoidal-fibers} asks that $(q,p) \colon E \to A \times B$
is groupoidal as an object of $(\lcat{K}/A \times B)_2$ or equivalently, by Corollary~\ref{cor:slice-2-cat-comparison}\ref{itm:slice-2-cat-comparison-groupoidal}, is groupoidal in $\lcat{K}_2/A\times B$. Note that this does not imply that the isofibrations $q \colon E \tfib A$ and $p \colon E \tfib B$ are themselves groupoidal in $\lcat{K}/A$ and $\lcat{K}/B$.

Condition \ref{itm:cartesian-on-the-right} asserts that the isofibration $(q,p) \colon E \tfib A \times B$ is a \emph{cartesian fibration on the right}, while condition \ref{itm:cocartesian-on-the-left} asserts that it is a \emph{cocartesian fibration on the left}. Condition \ref{itm:groupoidal-fibers} implies that these sliced map define, respectively, a groupoidal cartesian fibration in $(\lcat{K}/A)_2$ and groupoidal cocartesian fibration in $(\lcat{K}/B)_2$. However, the condition of being groupoidal in both slices is weaker than being groupoidal in the slice  over the product.
\end{rmk}

Our first task is to demonstrate that the motivating example, the arrow $\infty$-category $(p_1,p_0)\colon A^\cattwo \tfib A \times A$, defines a module from $A$ to $A$.

\begin{lem}\label{lem:disc-cart-on-right} \[\xymatrix{ A^\cattwo \ar[rr]^-{(p_1,p_0)} \ar@{->>}[dr]_{p_1} & & A \times A \ar@{->>}[dl]^{\pi_1} \\ & A}\] is a groupoidal cartesian fibration in $(\lcat{K}/A)_2$.
\end{lem}
\begin{proof}
 Applying Theorem~\ref{thm:gpd.cart.fib.chars}\ref{itm:gpd.cart.fib.chars.iii}, this is the case if and only if the induced functor in $(\lcat{K}/A)_2$ from the $\cattwo$-cotensor of $p_1 \colon A^\cattwo \tfib A$ to the comma object $\pi_1 \comma (p_1,p_0)$ is an equivalence. Applying the forgetful 2-functor $(\lcat{K}/A)_2 \to \lcat{K}_2$, by Corollary~\ref{cor:slice-2-cat-comparison}\ref{itm:slice-2-cat-comparison-equiv}, it suffices to show that the map between the domains defines an equivalence in $\lcat{K}_2$. Of course, the notion of equivalence is equivalence invariant, so we are free to use our preferred models of the $\cattwo$-cotensor and comma constructions, defined using simplicial pullbacks \eqref{eq:comma-as-simp-pullback} in $\lcat{K}$. 

As recalled in \ref{defn:sliced-cosmoi}, the $\cattwo$-cotensor of $p_1 \colon \Del^1 \pwr A \tfib A$ in $(\lcat{K}/A)_2$ is defined to be  the left-hand map formed by the simplicial pullback
\[ \xymatrix{ A' \ar@{->>}[d] \pbexcursion \ar[r] & \Del^1 \pwr (\Del^1 \pwr A) \ar@{->>}[d]^{\Del^1 \pwr p_1} \\ A \ar[r]_-\Delta & \Del^1 \pwr A}\] Up to equivalence over $A$, $A' \tfib A$  is $\ev_{\fbv{2}}\colon \Del^2 \pwr A \tfib A$. 

Similarly, the $\cattwo$-cotensor of $\pi_1 \colon A \times A \tfib A$ in $(\lcat{K}/A)_2$ is defined to be the left-hand map defined by the simplicial pullback:
\[ \xymatrix{ A \times (\Del^1 \pwr A) \ar@{->>}[d] \pbexcursion \ar[r] & \Del^1 \pwr (A \times A) \ar@{->>}[d]^{\Del^1 \pwr \pi_1} \\ A \ar[r]_\Delta & \Del^1 \pwr A}\] Using this, the domain of the comma construction $\pi_1 \comma (p_1,p_0)$ in $(\lcat{K}/A)_2$ is defined by the simplicial pullback
\[ \xymatrix{ \Horn^{2,1} \pwr A \pbexcursion \ar@{->>}[d] \ar[r] & A \times (\Del^1 \pwr A)  \ar@{->>}[d]^{1 \times p_1} \\ \Del^1 \pwr A \ar[r]_-{(p_1,p_0)} & A \times A}\] and the projection $\ev_{\fbv{2}} \colon \Horn^{2,1}\pwr A \to A$ from the comma construction to $A$  is again evaluation at the vertex $\fbv{2}$ in $\Horn^{2,1}$.

In this way, we see that, up to equivalence,  the map considered by Theorem~\ref{thm:gpd.cart.fib.chars}\ref{itm:gpd.cart.fib.chars.iii} is the map
\[ \xymatrix{ \Delta^2\pwr A \ar@{->>}[rr]^\sim \ar@{->>}[dr]_{\ev_{\fbv{2}}} & & \Horn^{2,1}\pwr A \ar@{->>}[dl]^{\ev_{\fbv{2}}} \\ & A}\] Lemma~\ref{lem:triv.fib.stab}\ref{itm:triv.fib.stab:c} implies that this is a trivial fibration, which indeed is an equivalence.
\end{proof}

\begin{prop}\label{prop:hom-is-a-module} The arrow $\infty$-category $(p_1,p_0) \colon A^\cattwo \tfib A \times A$ defines a module from $A$ to $A$.
\end{prop}
\begin{proof}
Lemma~\ref{lem:disc-cart-on-right} and its dual imply in particular that $(p_1,p_0) \colon A^\cattwo \tfib A \times A$ is  cartesian on the left and cocartesian on the right. Conservativity of 2-cell induction implies that it is groupoidal.
\end{proof}

More generally, given any functors $f \colon B \to A$ and $g \colon C \to A$, the comma $\infty$-category $(p_1,p_0) \colon f \comma g  \tfib C \times B$ encodes a module from $C$ to $B$. The proof makes use of a few intermediate results, which are of interest in their own right.

\begin{lem}\label{lem:sliced-adjunction-on-the-right} An isofibration $(q,p) \colon E \tfib A \times B$
is cartesian on the right if and only if any of the following equivalent conditions are satisfied:
\begin{enumerate}[label=(\roman*)]
\item\label{itm:sliced-right-defn}  $(q,p) \colon E \tfib A \times B$ is a cartesian fibration in $(\lcat{K}/A)_2$.
\item\label{itm:sliced-right-fibred-adj} The functor $i \colon E \to B \comma p$ induced by $\id_p$ admits a right adjoint in $(\lcat{K}/A \times B)_2$.
\item\label{itm:sliced-right-adj} The  functor $i \colon E \to B \comma p$ induced by $\id_p$ admits a right adjoint in $\lcat{K}_2/A \times B$.
\end{enumerate}
\end{lem}
\begin{proof}
Corollary~\ref{cor:slice-2-cat-comparison}\ref{itm:slice-2-cat-comparison-adj} implies that  conditions \ref{itm:sliced-right-fibred-adj} and \ref{itm:sliced-right-adj} are equivalent.

The equivalence with \ref{itm:sliced-right-defn} is an application of Theorem~\refIV{thm:cart.fib.chars}. (i)$\Leftrightarrow$(ii), whose statement is recalled in \S\ref{ssec:cartesian}, to the cartesian fibration $(q,p) \colon E \tfib A \times B$ in the slice 2-category $(\lcat{K}/A)_2$. This tells us that $(q,p) \colon E \tfib A \times B$ is a cartesian fibration if and only if a certain functor admits a right adjoint in $(\lcat{K}/A)_2/(\pi_1 \colon A \times B \tfib A)$. There is a commutative square of forgetful 2-functors, all of which are isomorphisms on underlying 1-categories:
\[ \xymatrix{(\lcat{K}/A\times B)_2 \ar[r] \ar[d] & (\lcat{K}/A)_2/(\pi_1 \colon A \times B \tfib A) \ar[d] \\ \lcat{K}_2/A \times B \ar[r]^-{\cong} & (\lcat{K}_2/A)/(\pi_1 \colon A \times B \tfib A)}\]  The left-hand map is a smothering 2-functor, the bottom functor is an isomorphism, and the top functor is surjective on objects and 1-cells. These properties imply that the right-hand functor is surjective on objects, surjective on 1-cells, and 2-full. The right-hand functor is also 2-conservative, as it commutes with the 2-conservative forgetful 2-functors to $\lcat{K}_2/A$. So the right-hand functor is a smothering 2-functor, and it suffices by Lemma~\refI{lem:missed-lemma} to demonstrate the adjunction in $\lcat{K}_2/A\times B$.

So we have argued, using \refIV{thm:cart.fib.chars}\refIV{itm:cart.fib.chars.i}$\Leftrightarrow$\refIV{itm:cart.fib.chars.ii},  that $(q,p) \colon E \tfib A \times B$ is cartesian on the right if and only if a certain functor admits a right adjoint in $\lcat{K}_2/A\times B$. This proves the equivalence of the stated conditions \ref{itm:sliced-right-defn}  and \ref{itm:sliced-right-adj} because the certain functor turns out to be $i\colon E \to B \comma p$. The following computation, included for the sake of completeness, justifies this claim. 

The $\cattwo$-cotensor of $\pi_1 \colon A \times B \tfib A$ in $(\lcat{K}/A)_2$ is computed by the simplicial pullback in $\lcat{K}$:
\[ \xymatrix{ A \times B^\cattwo \ar@{->>}[d]_{\pi_1} \pbexcursion \ar[r] & (A \times B)^\cattwo \ar@{->>}[d]^{\pi_1^\cattwo} \\ A \ar[r]_\Delta & A^\cattwo}\]
Then the comma object $\pi_1 \comma (q,p)$ in $(\lcat{K}/A)_2$ is defined by the left-hand simplicial pullback square in $\lcat{K}$, which we recognize as the pullback of the composite rectangle:
\[ \xymatrix{ B \comma p \ar@{->>}[d] \pbexcursion \ar[r] & A \times B^\cattwo \ar@{->>}[d]_{A \times p_1} \pbexcursion \ar[r]^-{\pi_0} & B^\cattwo \ar@{->>}[d]^{p_1} \\ E \ar[r]_-{(q,p)}  & A \times B \ar[r]_-{\pi_0} & B}\] 

We leave it to the reader to verify that the induced map from $q \colon E \tfib A$ to $qp_1 \colon B \comma p \to A$ is 
$i \colon E \to B \comma p$.
\end{proof}

\begin{prop}\label{prop:two-sided-pullback} Suppose $(q,p) \colon E \tfib A \times B$ is cartesian on the right, and consider the simplicial pullback 
$(q',p') \colon E' \tfib A' \times B'$ of $(q,p)$ along a pair of maps $a \times b \colon A' \times B' \to A \times B$. Then $(q',p')$ is again cartesian on the right. In particular, the pullback of a module is a module.
\end{prop}
\begin{proof}
We factor the simplicial pullback rectangle:
\[ \xymatrix{ E' \ar@{->>}[d]_{(q',p')} \ar[r] \pbexcursion & \bar{E} \ar@{->>}[d]_{(q,p')} \pbexcursion \ar[r] & E \ar@{->>}[d]^{(q,p)} \\ A' \times B' \ar[r]_{ a \times 1} & A \times B' \ar[r]_{1 \times b} & A \times B}\] 
The right-hand square is also a simplicial pullback in $\lcat{K}/A$. Applying Proposition~\ref{prop:cart-fib-pullback} in $\lcat{K}/A$, $(q,p') \colon \bar{E} \tfib A \times B'$ is cartesian on the right. The general result now follows from the special case where $b=\id_B$:
\[ \xymatrix{ E' \ar@{->>}[d]_{(q',pe)} \ar[r]^e \pbexcursion & E \ar@{->>}[d]^{(q,p)} \\ A' \times B \ar[r]_{ a \times 1} & A \times B }\] 
 and accordingly, we simplify our notation by dropping the now-superfluous primes.

The composite rectangle on the left below
\[ \vcenter{\xymatrix{ E' \ar@{->>}[d]_{(q',pe)} \ar[r]^e \pbexcursion & E \ar@{->>}[d]^{(q,p)} \\ A' \times B \ar[r]_{ a \times 1} \ar@{->>}[d]_{\pi_1} \pbexcursion & A \times B \ar@{->>}[d]^{\pi_1} \\ A' \ar[r]_{a} & A }} \quad \quad \vcenter{\xymatrix{ E' \times B \pbexcursion \ar@{->>}[d]_{q' \times 1} \ar[r]^{e \times 1} & E \times B \ar@{->>}[d]^{q \times 1} \\ A' \times B \ar[r]_{a \times 1} & A \times B}} \] 
defines a pullback in $\lcat{K}$, and thus, so does the right-hand square. Composing this with the pullback square 
\[ \xymatrix{ B \comma pe \ar@{->>}[d]_{(p_1,p_0)} \ar[r] \pbexcursion & B \comma p \ar@{->>}[d]^{(p_1,p_0)} \\ E' \times B \ar[r]_{e \times 1} & E \times B}\] we see that $(qp_1,p_0) \colon B \comma p \tfib A \times B$ pulls back along $a \times 1$ to $(q'p_1, p_0) \colon B \comma pe \tfib A' \times B$. Thus, the map $i \colon E \to B \comma p$ in $\lcat{K}/A \times B$ pulls back to the corresponding map $i \colon E' \to B \comma pe$ in $\lcat{K}/A' \times B$. Applying Lemma~\ref{lem:sliced-adjunction-on-the-right} and Remark~\ref{rmk:adjunctions-pullback}, the adjunction that demonstrates that $(q,p)$ is cartesian on the right also pulls back, proving that $(q',pe)$ is also cartesian on the right, as required.

To conclude that the pullback of a module is a module, it remains only to observe that the pullback of a groupoidal object is a groupoidal object. This follows directly from the fact that simplicial pullbacks in $\lcat{K}$ define weak pullbacks in the homotopy 2-category $\lcat{K}_2$ satisfying the universal property described in Definition~\refIV{defn:pullback}, which includes the usual 2-cell conservativity.
\end{proof}

Combining Propositions~\ref{prop:hom-is-a-module} and \ref{prop:two-sided-pullback}, we have:

\begin{cor}\label{cor:comma-is-a-module} For any pair of functors $f \colon B \to A$ and $g \colon C \to A$, the comma construction $f \comma g \tfib C \times B$ defines a module from $C$ to $B$.\qed
\end{cor}

\begin{defn} Given a functor $f \colon A \to B$, Corollary~\ref{cor:comma-is-a-module} implies that $B\comma f$ defines a module from $A$ to $B$ and $f \comma B$ defines a module from $B$ to $A$, which we refer to, respectively, as the \emph{covariant} and \emph{contravariant representable modules} associated to the functor $f \colon A \to B$. 
\end{defn}

\begin{lem}\label{lem:two-sided-groupoidal-cells} If $(q,p) \colon E \tfib A \times B$ is cartesian on the right, then $p$ is a cartesian fibration. Moreover, a $p$-cartesian 2-cell $\lambda \colon e' \To e \colon X \to E$ must have $q\lambda$ an isomorphism, and if $(q,p)$ is groupoidal cartesian on the right, the converse holds: if $q\lambda$ is an isomorphism, then $\lambda$ is $p$-cartesian.
\end{lem}
\begin{proof}
Lemma~\ref{lem:sliced-adjunction-on-the-right} tells us that $i \colon E \to B \comma p$ admits a right adjoint in the slice 2-category $\lcat{K}_2/A\times B$. Composition with the projection $\pi_B \colon A \times B \tfib B$ induces a forgetful 2-functor $\lcat{K}_2/A \times B \to \lcat{K}_2/B$. The image of this sliced adjunction tells us, via Theorem~\refIV{thm:cart.fib.chars}, that $p$ is a cartesian fibration in $\lcat{K}_2$. Via Observation~\refIV{obs:constructing-weak-lifts}, any $p$-cartesian 2-cell is isomorphic to one defined to be a whiskered composite of the counit $\epsilon$ of the adjunction $i\dashv r$. As this adjunction lifts to  $\lcat{K}_2/A \times B$, these 2-cells project along $q$ to an identity. Thus, for any $p$-cartesian 2-cell $\lambda$, we must have $q\lambda$ an isomorphism.

Finally,  suppose that $(q,p) \colon E \tfib A \times B$ is a groupoidal cartesian fibration in $(\lcat{K}/A)_2$.  Consider a 2-cell $\lambda \colon e' \To e \colon X \to E$ has $q\lambda$ an isomorphism. Lifting $(q\lambda)^{-1}$ along the isofibration $q \colon E \tfib A$, we see that $\lambda$ is isomorphic in $\lcat{K}_2$ to a 2-cell $\lambda'$ with $q\lambda$ an identity. Now $\lambda'$ is a 2-cell in $\lcat{K}_2/A$ with codomain $q \colon E \tfib A$. By local fullness of the smothering 2-functor $(\lcat{K}/A)_2 \to \lcat{K}_2/A$, it can be lifted to a 2-cell of the same kind in $(\lcat{K}/A)_2$. Since $(q,p)$ is groupoidal, any 2-cell of this form is $(q,p)$-cartesian. So $\lambda'$ is also $p$-cartesian and $\lambda$, which is isomorphic to a $p$-cartesian 2-cell, is itself $p$-cartesian.
\end{proof}

As the motivating example $(p_1,p_0) \colon A^\cattwo \tfib A \times A$ shows,  the legs of a module need not be groupoidal fibrations when considered separately in $\lcat{K}_2$.

\begin{defn}[horizontal composition of isofibrations over products]\label{defn:horiz-comp} Consider a pair of isofibrations $(q,p) \colon E \tfib A \times B$ and $(s,r) \colon F \tfib B \times C$  in an $\infty$-cosmos $\lcat{K}$. This data defines a composable pair of spans of isofibrations. Their \emph{horizontal composite} will define a span of isofibrations from $A$ to $C$ whose summit is formed by the simplicial pullback
\[ \xymatrix@!=5pt{ & & E \times_B F \pbdiamond \ar@{->>}[dl]_-{\pi_1} \ar@{->>}[dr]^-{\pi_0} \\  & E \ar@{->>}[dl]_q \ar@{->>}[dr]^p &  & F \ar@{->>}[dl]_s \ar@{->>}[dr]^r \\ A  & & B & & C}\]
Up to isomorphism, this span is constructed as the composite of the left-hand vertical in the simplicial pullback
\[ \xymatrix{ E \times_B F \pbexcursion \ar@{->>}[d]_{(q\pi_1,\pi_0)} \ar@{->>}[r]^-{\pi_1} & E \ar@{->>}[d]^{(q,p)} \\ A \times F \ar@{->>}[r]_{A \times s} & A \times B}\] with $A \times r \colon A \times F \tfib A \times C$. In particular, the projection map $E \times_B F \tfib A \times C$ is again an isofibration.
\end{defn}

\begin{rmk} In an abstract homotopy 2-category $\tcat{C}$ with finite 2-products, isofibrations $(q,p) \colon E \tfib A \times B$ over a product correspond bijectively to \emph{two-sided isofibrations} $A \xfibt{q} E \xtfib{p} B$ introduced in Definition~\refIV{defn:span-isofib}. Indeed, the 2-category $\isoSpan{A}{B}$ of two-sided isofibrations from $A$ to $B$ is isomorphic to the slice 2-category $\tcat{C}/A\times B$ of isofibrations over the product. 

At that level of generality, the horizontal composition of two-sided isofibrations is constructed via an iso-comma: 
\[ \xymatrix@!=5pt{ & & E \itimes_B F \ar@{}[dd]|\cong \ar@{->>}[dl]_-{\pi_1} \ar@{->>}[dr]^-{\pi_0} \\  & E \ar@{->>}[dl]_q \ar@{->>}[dr]^p &  & F \ar@{->>}[dl]_s \ar@{->>}[dr]^r \\ A  & & B & & C}\] Using Lemma~\refIV{lem:comma-span-isofib} it is easy to see that the composite span is again a two-sided isofibration, and hence defines an isofibration $E \itimes_B F \tfib A\times C$. This construction, via weak 2-limits, is well defined up to equivalence in $\tcat{C}/A\times C$. In the case where $\tcat{C}$ is the homotopy 2-category of an $\infty$-cosmos, this construction is equivalent to the horizontal composition operation defined via simplicial pullback in Definition~\ref{defn:horiz-comp}.
\end{rmk}

\begin{lem}\label{lem:module-pullback-cartesian} 
If $(q,p) \colon E \tfib A \times B$ and $(s,r) \colon F \tfib B \times C$ are each cartesian on the right then the horizontal composite $E \times_B F \tfib A \times C$ is again cartesian on the right.
\end{lem}
\begin{proof}
We have the following simplicial pullback in $\lcat{K}/A$, created from the simplicial pullback in $\lcat{K}$: 
\[ \xymatrix{ E \times_B F \pbexcursion \ar@{->>}[d]_{(q\pi_1,\pi_0)} \ar@{->>}[r]^-{\pi_1} & E \ar@{->>}[d]^{(q,p)} \\ A \times F \ar@{->>}[r]_{A \times s} & A \times B}\] 
Proposition~\ref{prop:cart-fib-pullback} demonstrates that $(q\pi_1,\pi_0) \colon E \times_B F \tfib A \times F$ is cartesian on the right. 

By Lemma~\ref{lem:two-sided-groupoidal-cells}, $r \colon F \tfib C$ is a cartesian fibration, and thus $(!,r) \colon F \tfib 1 \times C$ is cartesian on the right. Pulling back along $! \times C \colon A \times C \to 1 \times C$, Proposition~\ref{prop:two-sided-pullback} provides a cartesian fibration $A \times r \colon A \times F \tfib A \times C$ in $(\lcat{K}/A)_2$. Proposition~\refIV{prop:cart-fib-comp} now implies that the composite $(q\pi_1,r\pi_0) \colon E \times_B F \tfib A \times C$ is a cartesian fibration in $(\lcat{K}/A)_2$, as claimed.
\end{proof}

\begin{ex}\label{ex:modules-do-not-compose} It is not, however, generally the case that the pullback of a pair of modules is again a module. Consider
\[ \xymatrix@!=5pt{ & & {\Horn^{2,1}} \pwr A \pbdiamond \ar@{->>}[dl]_(.6){\pi_1} \ar@{->>}[dr]^(.6){\pi_0} \\ & A^\cattwo \ar@{->>}[dl]_{p_1} \ar@{->>}[dr]^{p_0} &  & A^\cattwo \ar@{->>}[dl]_{p_1} \ar@{->>}[dr]^{p_0} \\ A & & A & & A}\] 
The composite projections $A \fibt {\Horn^{2,1}}\pwr A \tfib A$ are induced by the inclusions of the endpoints $\fbv{0}$ and $\fbv{2}$ into the horn $\Horn^{2,1}$. By Lemma~\refI{lem:pointwise-equiv}, a 2-cell into $ {\Horn^{2,1}}\pwr A $ is an isomorphism if and only if it projects to an isomorphism when evaluated at all three vertices of $\Horn^{2,1}$, and thus this span is not a groupoidal object of $\lcat{K}/A \times A$. 
\end{ex}

\subsection{Module maps}

In this section we study maps between modules. For a pair of modules $E$ and $F$ from $A$ to $B$, a module map from $E$ to $F$ will be an isomorphism class of functors over $A \times B$; Corollary~\ref{cor:slice-2-cat-comparison}\ref{itm:slice-2-cat-comparison-iso-class} implies that it will not matter whether this notion is defined in the 2-category $(\lcat{K}\slice A \times B)_2$ or in $\lcat{K}_2\slice A \times B$. In \S\ref{sec:virtual}, we will see that the module maps define 2-cells in a 2-dimensional categorical structure to be introduced there.

\begin{lem}\label{lem:span-maps-cartesian} A commutative square
\begin{equation}\label{eq:span-map} \xymatrix{ E \ar@{->>}[d]_{(q,p)} \ar[r]^e & \bar{E} \ar@{->>}[d]^{(\bar{q},\bar{p})} \\ A \times B \ar[r]_{ a \times b} & \bar{A} \times \bar{B}}\end{equation} in which the vertical isofibrations define modules, induces a pair of cartesian functors
\[ \xymatrix{ E \ar[r]^e \ar@{->>}[d]_q & \bar{E} \ar@{->>}[d]^{\bar{q}} & E \ar@{->>}[d]_p \ar[r]^e & \bar{E} \ar@{->>}[d]^{\bar{p}} \\ A \ar[r]_a & \bar{A} & B \ar[r]_b & \bar{B}}\]
\end{lem}
\begin{proof}
By Lemma~\ref{lem:two-sided-groupoidal-cells}, any $p$-cartesian 2-cell is isomorphic to one that projects to an identity upon applying $q$. By commutativity of the left-hand square, the image of such a 2-cell under $e$ likewise projects to an identity under $\bar{q}$, whence it defines a $\bar{p}$-cartesian 2-cell. 
\end{proof}

\begin{defn}
For a fixed pair of  objects $A, B$ in an $\infty$-cosmos $\lcat{K}$, we write $\Mod{A}{B}$ for the full quasi-categorically enriched subcategory of $\lcat{K}/A \times B$ whose objects are modules $(q,p) \colon E \tfib A \times B$ from $A$ to $B$. The quasi-category of maps from $(q,p)$ to a module $(s,r) \colon F \tfib A \times B$ is defined by the simplicial pullback
 \[ \xymatrix{ \map_{A \times B}((q,p),(r,s)) \pbexcursion \ar[d] \ar[r] & \map(E,F) \ar[d]^{\map(E,(r,s))} \\ \Delta^0 \ar[r]_-{(p,q)} & \map(E,A \times B)}\]
which we abbreviate to $\map_{A \times B}(E,F)$ whenever possible.
\end{defn}

As in previous similar situations, when considering mapping quasi-categories between spans we frequently allow the domain object to be an arbitrary span from $A$ to $B$ that is not necessarily a module and whose legs might not be isofibrations. In such situations we continue to insist that codomain spans are modules. The Yoneda lemma provides the following characterization of the quasi-category of maps from a representable module to a generic module.

\begin{prop}\label{prop:two-sided-yoneda} Consider any functor $f \colon A \to B$ and the induced map
\[  \vcenter{\xymatrix@C=0.8em@R=1.2em{
    & {A}\ar[d]^-{t} & \\
    & {B \comma f}\ar@{->>}[dl]_{p_1}\ar@{->>}[dr]^{p_0} & \\
    {A}\ar[rr]_{f} && {B}
    \ar@{} "2,2";"3,2" |(0.6){\Leftarrow\phi}
  }} = 
  \vcenter{\xymatrix@C=0.8em@R=1.2em{
    & {A}\ar@{=}[dl]\ar[dr]^{f} & \\
    {A}\ar[rr]_{f} && {B}
    \ar@{} "1,2";"2,2" |(0.6){=}
  }} \]
over $A \times B$. Then for any module $E$ from $A$ to $B$, precomposition with $t \colon A \to B \comma f$  induces an equivalence of quasi-categories 
\[ \map_{A \times B} ( B \comma f , E) \simeq \map_{A \times B} ( A , E ).\] 
\end{prop}
\begin{proof}
We apply Corollary~\refIV{cor:groupoidal-yoneda} to the groupoidal cartesian fibration $(q,p) \colon E \tfib A \times B$ and the point $(1,f) \colon A \to A \times B$  in $\lcat{K}/A$. The module represented by $(1,f)$ in $\lcat{K}/A$ is $(p_1,p_0) \colon B \comma f \tfib A \times B$ and the map $t$ is induced, as usual, by the identity 2-cell in $\lcat{K}$.
\end{proof}

We know, from Recollection~\ref{rec:ess.unique.1-cell.ind} for example, that 2-cells of the following form \[\xymatrix@=10pt{ & X \ar[dl]_c \ar[dr]^b \ar@{}[dd]|(.4){\chi}|(.6){\Leftarrow} \\ C \ar[dr]_g & & B \ar[dl]^f \\ & A}\] in the homotopy 2-category $\lcat{K}_2$ correspond bijectively to isomorphism classes of functors $X \to f \comma g$ over $C \times B$. With this correspondence in mind, we will define a \emph{module map}, between a pair of modules from $A$ to $B$, to be an isomorphism class of functors over $A \times B$. The module maps form the morphisms in a 1-category $\qMod{A}{B}$, defined as a quotient of $\Mod{A}{B}$. Its definition makes use of  the product-preserving functor $\tau_0\colon\qCat \to \Set$ that sends a quasi-category to the set of isomorphism classes of its objects that carries an equivalence of quasi-categories to a bijection between sets of isomorphism classes of objects.

\begin{defn}\label{defn:1-cat-of-modules}
In an $\infty$-cosmos $\lcat{K}$, define a 1-category $\qMod{A}{B}$ whose:
\begin{itemize}
\item objects are modules from $A$ to $B$, and 
\item whose morphisms are \emph{module maps}. 
\end{itemize}
 The hom-set between a pair of modules $E$ and $F$ from $A$ to $B$ is $\tau_0\map_{A \times B}(E,F)$. On account of the factorization \[ \tau_0 \colon  \qCat \xrightarrow{\ho} \Cat \xrightarrow{\tau_0} \Set,\] the category $\qMod{A}{B}$ could also be regarded as a quotient of the full sub 2-category of  $(\lcat{K}/A\times B)_2$ spanned by the modules. Isomorphism classes of vertices in $\map_{A\times B}(E,F)$ coincide exactly with isomorphism classes of functors over $A\times B$, in either $(\lcat{K}/A\times B)_2$ or $\lcat{K}_2\slice A\times B$ by Corollary~\ref{cor:slice-2-cat-comparison}\ref{itm:slice-2-cat-comparison-iso-class}.

A module map from $E$ to $F$ will be denoted by $E \To F$ because these will be the 2-morphisms in a 2-dimensional categorical structure to be introduced in section \ref{sec:virtual}.
\end{defn}

Note that two modules $E$ and $F$ from $A$ to $B$ are equivalent as objects in $\lcat{K}/A\times B$ if and only if they are isomorphic in $\qMod{A}{B}$. A special case of Proposition~\ref{prop:two-sided-yoneda} allows us to define fully-faithful embeddings $\hom(A,B) \to \qMod{A}{B}$ and $\hom(A,B)\op \to \qMod{B}{A}$  whose images are the full subcategories spanned by the covariant and contravariant representables, respectively.

\begin{lem}\label{lem:full-yoneda} There is a fully-faithful embedding  $\hom(A,B) \to \qMod{A}{B}$ defined on objects by mapping a functor $f \colon A \to B$ to the covariant representable module $B \comma f$. On morphisms, this functor carries a 2-cell $\alpha \colon f \To g \colon A \to B$ to the module map representing the unique isomorphism class of functors $B\comma f \to B\comma g$ over $A \times B$  defined by 1-cell induction from the left-hand pasting diagram:
\[  \vcenter{\xymatrix@C=0.8em@R=1.2em{
    & {B\comma f}\ar@{->>}[dl]_{p_1}\ar@{->>}[dr]^{p_0} & \\
    {A}\ar[rr]|{f} \ar@/_3ex/[rr]_g^{\Downarrow \alpha} && {B}
    \ar@{} "1,2";"2,2" |(0.6){\Leftarrow\phi} 
  }} = 
  \vcenter{\xymatrix@C=0.8em@R=1.2em{    & {B\comma f}\ar@{->>}[ddl]_{p_1}\ar@{->>}[ddr]^{p_0} \ar[d] & \\
    & {B\comma g}\ar@{->>}[dl]|{p_1}\ar@{->>}[dr]|{p_0} & \\ 
    {A}\ar[rr]_{g} && {B}  \ar@{} "2,2";"3,2" |(0.6){\Leftarrow\phi} 
  }} \qquad   \vcenter{\xymatrix@C=0.8em@R=1.2em{
    & {g\comma B}\ar@{->>}[dl]_{p_1}\ar@{->>}[dr]^{p_0} & \\
    {B} && {A} \ar[ll]|{g} \ar@/^3ex/[ll]^f_{\Uparrow \alpha}
    \ar@{} "1,2";"2,2" |(0.6){\Leftarrow\phi} 
  }} = 
  \vcenter{\xymatrix@C=0.8em@R=1.2em{    & {g\comma B}\ar@{->>}[ddl]_{p_1}\ar@{->>}[ddr]^{p_0} \ar[d] & \\
    & {f\comma B}\ar@{->>}[dl]|{p_1}\ar@{->>}[dr]|{p_0} & \\ 
    {B}&& {A} \ar[ll]^f  \ar@{} "2,2";"3,2" |(0.6){\Leftarrow\phi} 
  }} \]
A dual construction defines a fully-faithful embedding  $\hom(A,B)\op \to \qMod{B}{A}$ that carries $f$ to the contravariant represented module $f \comma B$ and carries the 2-cell $\alpha$ to the unique isomorphism class of functors $g \comma B \to f\comma B$ over $B \times A$.
\end{lem}
\begin{proof}
For any pair of functors $f,g \colon A \to B$, Proposition~\ref{prop:two-sided-yoneda} provides an equivalence of quasi-categories \[\map_{A \times B}(B \comma f, B \comma g) \simeq \map_{A \times B} ((1,f), B \comma g).\] Passing to isomorphism classes of objects, the left-hand side is the set of module maps $B \comma f \To B \comma g$, i.e., the set of isomorphism classes of functors $B\comma f \to B\comma g$ over $A \times B$. The right-hand side is the set of 1-cells \[ \xymatrix{ & A \ar@{=}[dl] \ar[dr]^f \ar@{-->}[d] \\ A & B \comma g \ar@{->>}[l]^{p_1} \ar@{->>}[r]_{p_0} & B}\] up to a 2-cell isomorphism over $A \times B$. By Recollection~\ref{rec:ess.unique.1-cell.ind}, this is isomorphic to the set of 2-cells $f \To g \colon A \to B$.
\end{proof}

We can extend our definition of module map to include maps between modules between different pairs of objects, such as displayed in \eqref{eq:span-map}.

\begin{defn}\label{defn:module-map} Given modules $(q,p) \colon E \tfib A \times B$ and $(\bar{q},\bar{p})\colon \bar{E} \tfib \bar{A} \times \bar{B}$ and a pair of functors $a \colon A \to \bar{A}$ and $b \colon B \to \bar{B}$ a \emph{module map} from $E$ to $\bar{E}$ over $a \times b$ is an isomorphism class of objects in the quasi-category defined by the simplicial pullback
 \[ \xymatrix{ \map_{\bar{A} \times \bar{B}}((aq,bp),(\bar{q},\bar{p})) \pbexcursion \ar[d] \ar[r] & \map(E,\bar{E}) \ar[d]^{\map(E,(\bar{q},\bar{p}))} \\ \Delta^0 \ar[r]_-{(ap,bq)} & \map(E,\bar{A} \times \bar{B})}\] which we abbreviate to $\map_{a,b}(E,\bar{E})$. 
\end{defn}

The following lemma shows that this new definition amounts to no substantial generalization.

\begin{lem}\label{lem:module-map-pullback-equivalence} Given modules $(q,p) \colon E \tfib A \times B$ and $(\bar{q},\bar{p})\colon \bar{E} \tfib \bar{A} \times \bar{B}$ and a pair of functors $a \colon A \to \bar{A}$ and $b \colon B \to \bar{B}$, there is an equivalence of quasi-categories
\[ \xymatrix{\map_{A \times B}(E,\bar{E}(b,a)) \ar[r]^-\simeq & \map_{a,b}(E, \bar{E}) ,}\]
where $\bar{E}(b,a)$ is the module defined by the simplicial pullback
\begin{equation}\label{eq:pullback-for-UP} \xymatrix{ \bar{E}(b,a) \pbexcursion \ar@{->>}[d] \ar[r] & \bar{E} \ar@{->>}[d]^{(\bar{q},\bar{p})} \\ A \times B \ar[r]_{a \times b} & \bar{A} \times \bar{B}}\end{equation}
In particular, there is a bijection between module maps $E \To \bar{E}(b,a)$ and module maps from $E$ to $\bar{E}$ over $a \times b$.
\end{lem}
\begin{proof}
The simplicial pullback defining $\map_{a,b}(E, \bar{E})$ factors as follows
 \[ \xymatrix@C=6em{ \map_{a,b}(E,\bar{E}) \pbexcursion \ar[d] \ar[r] & \map(E,\bar{E}(b,a)) \pbexcursion \ar[d] \ar[r] & \map(E,\bar{E}) \ar[d]^{\map(E,(\bar{q},\bar{p}))} \\ \Delta^0 \ar[r]_-{(p,q)} & \map(E, A \times B) \ar[r]_{\map(E,a \times b)} & \map(E,\bar{A} \times \bar{B})}\] where the right-hand pullback is the image of  \eqref{eq:pullback-for-UP} under the functor $\map(E,-) \colon \tcat{K} \to \qCat$. The left-hand pullback, which defines $\map_{A \times B}(E,\bar{E}(b,a))$, demonstrates that this hom-quasi-category is isomorphic to $\map_{a,b}(E,\bar{E})$.
\end{proof}

\subsection{Equivalence of modules}

The following lemma defines a suitable notion of  \emph{equivalence} between modules.

\begin{lem}\label{lem:equiv-of-modules} Given a pair of modules $(q,p) \colon E \tfib A\times B$ and $(s,r) \colon F \tfib A \times B$ the following are equivalent:
\begin{enumerate}[label=(\roman*)]
\item There exists a functor $f \colon E \to F$ over $A\times B$ that is  an equivalence in $\lcat{K}$.
\item The isofibrations $E$ and $F$ are equivalent as objects of $(\lcat{K}/A\times B)_2$ or of $\lcat{K}_2/A\times B$.
\item The modules $E$ and $F$ are isomorphic as objects in $\qMod{A}{B}$.
\end{enumerate}
\end{lem}
\begin{proof}
Corollary~\ref{cor:slice-2-cat-comparison}\ref{itm:slice-2-cat-comparison-equiv} establishes the equivalence (i)$\Leftrightarrow$(ii).  The implication (ii)$\Rightarrow$(iii) follows by applying $\tau_0 \colon \Cat\to\Set$ to the hom-categories, which carries an equivalence in the sub 2-category of $(\lcat{K}/A\times B)_2$ spanned by the modules to an isomorphism in the 1-category $\qMod{A}{B}$. Conversely, the data of an isomorphism and its inverse in $\qMod{A}{B}$ provides an equivalence in $(\lcat{K}/A \times B)_2$, proving that (iii)$\Rightarrow$(ii).
\end{proof}

The following result provides a criterion for recognizing when a module $E\tfib A\times B$ is covariantly represented, that is to say when it is equivalent to the representable module $B\comma f$ associated with some functor $f\colon A\to B$. By Lemma~\refI{lem:technicalsliceadjunction}, the codomain-projection functor $p_1 \colon B \comma f \to A$ associated to a covariant representable admits a right adjoint right inverse $t$ induced by the identity 2-cell associated to $f$.
  \begin{equation*}
    \vcenter{\xymatrix@=1em{
      & {A}\ar@{=}[dl]\ar[dr]^{f} & \\
      {A}\ar[rr]_{f} && {B}
      \ar@{} "1,2";"2,2" |(0.6){\textstyle =}
    }}
    \mkern20mu = \mkern20mu
    \vcenter{\xymatrix@=1em{
      & {A}\ar[d]^-t \ar@/^/[ddr]^f \ar@/_/@{=}[ddl] & \\
      & {B\comma f}\ar[dl]|{p_1}\ar[dr]|{p_0} & \\
      {A}\ar[rr]_{f} && {B}
      \ar@{} "2,2";"3,2" |(0.6){\Leftarrow\beta}
    }}
  \end{equation*}
so that the composite $p_1t$ equals $f$.  We now show that this property characterizes the representable modules.

\begin{lem}\label{lem:rep-mod-char} Suppose $(q,p) \colon E \tfib A \times B$ defines a module from $A$ to $B$ and that $q$ admits a right adjoint right inverse $t \colon A \to E$. Then $E$ is equivalent to $B \comma pt$ over $A \times B$.
\end{lem}
\begin{proof}
The unit $\eta$ of the adjunction $q \dashv t$ induces a 1-cell $r \colon E \to B\comma pt$
  \begin{equation}\label{eq:r-equiv-defn}
    \vcenter{\xymatrix@=12pt{
 & {E}\ar@{=}[dd] \ar@/^/[ddr]^p \ar@/_/[ddl]_q &  & & 
      & {E}\ar[d]^-r \ar@/^/[ddr]^p \ar@/_/[ddl]_q & \\     \ar@{}[dr]|{\Leftarrow\eta}  & &   & = &
      & {B\comma pt}\ar[dl]|{p_1}\ar[dr]|{p_0} & \\    {A}\ar[r]_{t} & E \ar[r]_p & {B} & & 
      {A}\ar[rr]_{pt} && {B}
      \ar@{} "2,6";"3,6" |(0.6){\Leftarrow\beta}
    }}
  \end{equation}
Define $e \colon B\comma pt \to E$ to be the domain component of the cartesian lift of the morphism
\begin{equation}\label{eq:e-equiv-defn} \vcenter{\xymatrix{ B \comma pt \ar[r]^-{p_1} \ar@/_/[drr]_{p_0} \ar@{}[drr]^(.6){\Uparrow\beta} & A \ar[r]^t & E \ar[d]^p \\ & & B}}
    \mkern20mu = \mkern20mu
    \vcenter{\xymatrix{ B \comma pt \ar[r]^-{p_1}  \ar@/_3.5ex/[rr]_(.6)e^{\Uparrow\chi} & A \ar[r]^t & E \ar[d]^p \\ & & B}}
\end{equation} chosen so that $q\chi = \id_{p_1}$; Lemma~\ref{lem:two-sided-groupoidal-cells} tells us this is possible. In particular, $e$ and $r$ are both maps over $A \times B$.

Restricting \eqref{eq:e-equiv-defn} along $r$, we see that $\chi r \colon er \To tq$ is a $p$-cartesian lift of $\beta r = p \eta$. Since $q \dashv t$ is a right-adjoint-right-inverse, $q\eta= \id_q$, and Lemma~\ref{lem:two-sided-groupoidal-cells} implies that $\eta$ is also a $p$-cartesian lift of $p\eta$. Observation~\refIV{obs:weak.cart.conserv} then implies that $er \cong \id_E$ over $A \times B$.

To show that $re \cong \id_{B \comma pt}$, Lemma~\ref{lem:two-sided-groupoidal-cells} implies that $r\chi$ is a $p_0$-cartesian lift of $\beta \colon p_0 \To p_0 rt p_1$. From the defining equation \eqref{eq:r-equiv-defn} and \refI{lem:technicalsliceadjunction}, we see that $rt$ defines a right adjoint right inverse to $p_1 \colon B \comma pt \to A$. The unit of $p_1 \dashv rt$, as constructed in the proof of \refI{lem:technicalsliceadjunction}, defines a lift of $\beta$ along $p_0$, projecting along $p_1$ to an identity. Lemma~\ref{lem:two-sided-groupoidal-cells} tells us this unit is $p_0$-cartesian, and, as before, Observation~\refIV{obs:weak.cart.conserv} provides the desired isomorphism $re \cong \id_{B \comma pt}$ over $A \times B$. This demonstrates that $E$ is equivalent to $B \comma pt$ over $A \times B$.
\end{proof}

 On combining this with the following result for quasi-categories, which is a corollary of Lemma~\refI{lem:RARI-lifting}, we obtain a familiar ``pointwise'' recognition principle for representable modules between quasi-categories:

 \begin{lem}\label{lem:qcat-pointwise-rep}
   A cocartesian fibration $q\colon E\tfib A$ of quasi-categories admits a right adjoint right inverse $t\colon A\to E$ if and only if for each object $a\in A$ the fibre $E_a$ over that object has a terminal object.
 \end{lem}

 \begin{proof}
    To prove necessity, consider the following commutative diagram:
    \begin{equation*}
      \xymatrix@=2em{
        \Delta^0 \ar[r]_-{\fbv{n}} \ar@/^2ex/[rr]^-{ta} & 
        \boundary\Delta^n \ar@{u(->}[d] \ar[r] & 
        {E_a}\pbexcursion\ar@{->>}[d]^q\ar@{u(->}[r] &
        E \ar@{->>}[d]^q \\ 
         & \Delta^n \ar[r]\ar@{.>}[rru]\ar@{-->}[ru] & 1 \ar[r]_a & A
      }
    \end{equation*}
    Under our assumption that $q$ has a right adjoint right inverse $t$ we may apply the lifting condition depicted in equation~\refI{eq:RARI-lifting} of Lemma~\refI{lem:RARI-lifting} to show that the outer composite square has a lifting (the dotted arrow) and then apply the pullback property of the right hand square to obtain a lifting for the left hand square (the dashed arrow). This lifting property of the left hand square shows that $ta$ is a terminal object in $E_a$.

    To establish sufficiency, we start by taking the object $ta\in E$ to be the terminal object in the fibre $E_a$ for each object $a\in A$. Now by Lemma~\refI{lem:RARI-lifting} our desired result follows if we can show that each lifiting problem
    \begin{equation}\label{eq:lift-term-in-fibre}
      \xymatrix{ 
        \Delta^0 \ar[r]_-{\fbv{n}} \ar@/^2ex/[rr]^{ta} & 
        \boundary\Delta^n \ar@{u(->}[d] \ar[r]_y & 
        E \ar@{->>}[d]^q \\ 
        & \Delta^n  \ar[r]_x & A}
    \end{equation}
    has a solution. Consider the order preserving function $k\colon [n]\times[1]\to [n]$ defined by $k(i,0) \defeq i$ and $k(i,1) \defeq n$. Taking nerves, this gives rise to a simplicial map $k\colon\Del^n\times\Del^1\to\Del^n$ and it is easy to check that this restricts to a simplicial map $k\colon\boundary\Del^n\times\Del^1\to\Del^n$ which we may compose with $x\colon\Del^n\to A$  to give (a representative of) a 2-cell 
    \begin{equation*}
      \xymatrix@=2em{
        \boundary\Delta^n \ar[d]_{!} \ar[r]^y & 
        E \ar@{->>}[d]^q \\ 
        1  \ar[r]_a & A
        \ar@{} "1,1";"2,2"|{{\textstyle\Downarrow}\kappa}
      }
    \end{equation*}
    with the property that $\kappa\fbv{n}$ is the identity 2-cell on $a$. Now we may take a cocartesian lift $\chi\colon y\Rightarrow u$ of $\kappa$ and by construction the image of $u\colon\boundary\Del^n\to E$ is contained entirely in the fibre $E_a\subseteq E$. What is more, we know that $\chi\fbv{n}$ is an isomorphism since we have, by pre-composition stability of cocartesian 2-cells, that it is a cocartesian lift of the identity 2-cell on $a$. Consequently, we see that $u\fbv{n}$ is a terminal object of $E_a$, since it is isomorphic to $ta$, and it follows that we may apply its universal property to extend $u\colon\boundary\Del^n\to E_a$ to a simplex $v\colon\Del^n\to E_a$. 

    We may combine (a representative of) the 2-cell $\chi$ with $v$ to assemble the upper horizontal map in the following commutative square:
    \begin{equation*}
      \xymatrix@=2em{
        {\boundary\Del^n\times\Del^1 \cup \Del^n\times\Del^{\fbv{1}}}
        \ar[r]\ar@{u(->}[d] & E \ar@{->>}[d]^q \\
        {\Del^n\times\Del^1}\ar[r] \ar@{-->}[ur]_\ell & {A}
      }
    \end{equation*}
    Now we can construct a solution for this lifting problem by successively picking fillers for each of the non-degenerate $(n+1)$-simplices in $\Del^n\times\Del^1$. These are guaranteed to exist the for the first $n-1$ of those because $q$ is an isofibrations and they entail the filling of an inner horn. To obtain a filler for the last one we need to fill an outer horn, but observe that its final edge maps to an isomorphism of $E$, since it is the image of $\chi\fbv{n}$, and so it too has a filler. Finally on restricting the resulting map $\ell \colon \Del^n\times\Del^1\to E$ to the initial end of the cylinder that is its domain, we obtain an $n$-simplex $\Del^n\to E$ which is easily seen to be a solution to the original lifting problem in~\eqref{eq:lift-term-in-fibre} as required.
\end{proof}

 \begin{cor}\label{cor:fibrewise-rep}
   Suppose that $(q,p)\colon E\tfib A\times B$ defines a module of quasi-categories from $A$ to $B$. Then $E$ is covariantly represented if and only if for all objects $a\in A$ the module $E(\id_B,a)$ from $1$ to $B$ is covariantly represented by some object $b\in B$.
 \end{cor}

 \begin{proof}
    Applying Lemma~\ref{lem:rep-mod-char} and Lemma~\ref{lem:qcat-pointwise-rep}, we find that the module $(q,p)\colon E\tfib A\times B$ is covariantly represented if and only if each fibre of $q\colon E\tfib A$ has a terminal object. For each object $a\in A$ the module $E(\id_B,a)$ is given by pullback along $a\times\id_B\colon 1\times B\to A\times B$, and hence it is isomorphic to the module
    \begin{equation*}
      \xymatrix@R=1em@C=1.5em{
        {} & {E_a}\ar@{->>}[dl]_{!}\ar@{->>}[dr]^{p} & {} \\
        {1} & {} & {B}
      }
    \end{equation*}
    where $E_a$ is the fibre of $q\colon E\tfib A$ at $a$. Applying Lemma~\ref{lem:rep-mod-char}, it follows that $E(\id_B,a)$ is covariantly represented if and only if the map $!\colon E_a\tfib 1$ has a right adjoint right inverse which is equivalent to asking that $E_a$ has a terminal object.
 \end{proof}

We have long been acquainted with a particular instance of equivalence between modules. As the following example recalls, a pair of functors in the homotopy 2-category are adjoints if and only if the contravariant module represented by the left adjoint is equivalent to the covariant module represented by the right adjoint.

\begin{ex}\label{ex:adjoint-equivalence-mods}
  The arguments of \S\refI{sec:limits}  generalise, word for word, to any $\infty$-cosmos $\lcat{K}$ to  demonstrate that a pair of functors $u\colon A\to B$ and $f\colon B\to A$ comprise an adjoint pair $f\dashv u$ if and only if the comma objects $f\comma A$ and $B\comma u$ are equivalent as objects over $A\times B$. This latter condition means that there exists some equivalence $w\colon f\comma A\to B\comma u$ which makes the following triangle 
\begin{equation*}
  \xymatrix@=1em{
    {f\comma A}\ar@{->>}[dr]_(0.3){(p_1,p_0)}\ar[rr]^{w}_{\simeq}
    && {B\comma u}\ar@{->>}[dl]^(0.3){(q_1,q_0)} \\
    & {A\times B}&
  }
\end{equation*}
commute. More precisely, isomorphism classes of such equivalences $w$ in the 2-categorical slice $\lcat{K}_2\slice A\times B$ stand in bijective correspondence with choices of unit and counit for an adjunction $f \dashv u$.

  Using the language established above, we might equivalently observe that a functor $f\colon B\to A$ admits a right adjoint if and only if the contravariant representable module $f\comma A$ is also covariantly represented by some functor $u\colon A\to B$.

  Restricting to the case of quasi-categories, we may apply Lemma~\ref{lem:qcat-pointwise-rep} to show that a functor $f\colon B\to A$ of quasi-categories admits a right adjoint if and only if for all objects $a\in A$ the comma $f\comma a$ has a terminal object. On exploiting the equivalence between the comma $f\comma a$ and the slice $\slicer{f}{a}$, we recover the pointwise criterion for the existence of a right adjoint that is the converse to Proposition~\refI{prop:pointwise-univ-adj} implicit in Theorem~\refI{thm:pointwise}.
\end{ex}



\section{The virtual equipment of modules}\label{sec:virtual}

A \emph{double category} is a sort of 2-dimensional category with objects; two varieties of 1-morphisms, the ``horizontal'' and the ``vertical''; and 2-dimensional cells fitting into ``squares'' whose boundaries consist of horizontal and vertical 1-morphisms with compatible domains and codomains. A motivating example from abstract algebra is the double category of modules: objects are rings, vertical morphisms are ring homomorphisms, horizontal morphisms are bimodules, and whose squares are bimodule homomorphisms. In the literature, this sort of structure is sometimes called a \emph{pseudo double category} --- morphisms and squares compose strictly in the ``vertical'' direction but only up to isomorphism in the ``horizontal'' direction --- but we'll refer to this simply as a ``double category'' here as it is the only variety that we will consider.

Our aim in this section is to describe a similar structure whose objects and vertical morphisms are the $\infty$-categories and functors in an $\infty$-cosmos, whose horizontal morphisms are modules, and whose squares are module maps, as defined in \ref{defn:module-map}. The challenge is that composition of modules is a complicated operation, making use of certain colimits that are not within the purview of the axioms of an $\infty$-cosmos. 

Rather than leave the comfort of our axiomatic framework in pursuit of a double category of modules, we instead describe the structure that naturally arises within the axiomatization: it turns out to be familiar to category theorists and robust enough for our desired applications, which will be the subject of the next section. We first demonstrate that $\infty$-categories, functors, modules, and module maps assemble into a \emph{virtual double category}, a weaker structure than a  double category in which cells are permitted to have a multi horizontal source, as a replacement for horizontal composition of modules. We then observe that certain cells in this virtual double category satisfy strict universal properties, defining what Cruttwell and Shulman call a \emph{virtual equipment}  \cite{CruttwellShulman:2010au}. This universal property encodes numerous bijections between module maps, which we exploit in the next section to develop the theory of pointwise Kan extensions for $\infty$-categories.

\subsection{The virtual double category of modules}

\begin{defn}[the double category of isofibrations]\label{defn:double-iso-spans}
The homotopy 2-category $\tcat{K}_2$ of an $\infty$-cosmos supports a \emph{double category of spans} $\SSpan{K}$  whose:
\begin{itemize}
\item objects are $\infty$-categories
\item vertical arrows are functors
\item horizontal arrows $E \colon A \prof B$ are isofibrations  $(q,p) \colon E \tfib A \times B$ together with the identity span from $A$ to $A$
\item 2-cells, with boundary as displayed below
\begin{equation}\label{eq:2-cell.bdary}
  \xymatrix{ A \ar[d]_f \ar[r]^E|{\mid} & B \ar[d]^g \ar@{}[dl]|{\Downarrow} \\ C \ar[r]_F|{\mid} & D}
\end{equation}
are isomorphism classes of maps of spans, i.e., a 2-cell from $A \xfibt{q} E \xtfib{p} B$ to $C \xfibt{s} F \xtfib{r} D$ over $f$ and $g$ is an isomorphism class of objects in the category defined by the pullback diagram
\begin{equation}
  \label{eq:2-cell.hom.defn}
  \xymatrix{ \hom_{f,g}(E,F) \pbexcursion \ar[r] \ar[d] & \hom(E,F) \ar[d]^{\hom(E,(s,r))} \\ \catone \ar[r]_-{(fq,gp)} & \hom(E,C\times D)}
\end{equation}
\end{itemize}
Horizontal composition of two-sided isofibrations are given by forming the simplicial pullback
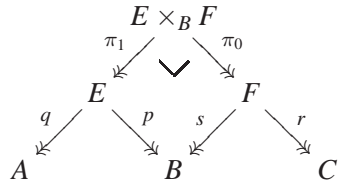
\[ \xymatrix@!=5pt{ & & E \times_B F \pbdiamond \ar@{->>}[dl]_-{\pi_1} \ar@{->>}[dr]^-{\pi_0} \\  & E \ar@{->>}[dl]_q \ar@{->>}[dr]^p &  & F \ar@{->>}[dl]_s \ar@{->>}[dr]^r \\ A  & & B & & C}\]
as described in Definition~\ref{defn:horiz-comp}. As explained there, this construction indeed defines an isofibration $E \times_B F \tfib A \times C$. Simplicial functoriality of the pullbacks in $\lcat{K}$ implies that horizontal composition of morphisms and 2-cells is associative and unital up to isomorphism.
\end{defn}

\begin{obs}\label{obs:double-iso-spans}
  It is instructive to relate the notion of cell given in the last definition with that of module map given in Definition~\ref{defn:module-map}. Were we to follow that latter definition, we might define a 2-cell of the form displayed in~\eqref{eq:2-cell.bdary} as an isomorphism class of objects in the mapping quasi-category defined in the following pullback:
  \begin{equation*}
    \xymatrix{ \map_{f,g}(E,F) \pbexcursion \ar[r] \ar@{->>}[d] & \map(E,F) \ar@{->>}[d]^{\map(E,(s,r))} \\ \Del^0 \ar[r]_-{(fq,gp)} & \map(E,C\times D)}
  \end{equation*}
  Applying the homotopy category functor to this pullback we obtain a cone over the diagram in ~\eqref{eq:2-cell.hom.defn}, thus inducing a comparison functor $\ho(\map_{f,g}(E,F))\to \hom_{f,g}(E,F)$ which, by Proposition~\refI{prop:weak-homotopy-pullbacks}, is a smothering functor which acts identically on objects. Now we know that isomorphism classes of objects of a quasi-category and its homotopy category correspond, as do those of a pair of categories related by a smothering functor. So it follows that isomorphism classes of objects in $\map_{f,g}(E,F)$ and $\hom_{f,g}(E,F)$ coincide and thus that the 2-cells of $\SSpan{K}$ may be defined equally in terms of isomorphism classes in either of these hom-spaces. Consequently we see that module maps are simply 2-cells in $\SSpan{K}$ whose vertical domain and codomain spans happen to be modules.
\end{obs}


\begin{rmk}
Lemma~\ref{lem:module-pullback-cartesian} reveals that the substructure of $\SSpan{K}$ obtained by restricting our attention only to those isofibrations that are both cocartesian on the left and cartesian on the right is {\em almost\/} a sub double category of $\SSpan{K}$. It fails to be such only in as much as the identity span $A\leftarrow A\to A$ on a general object $A$ may fail to be in that substructure. While this lack of identities might present only a minor inconvenience, our real interest is in the substructure defined by restricting further to those spans that are \emph{modules}, i.e., groupoidal in addition to being cartesian on the right and cocartesian on the left. Example~\ref{ex:modules-do-not-compose} illustrates that modules do not form a sub double category of the double category of spans in $\lcat{K}_2$. However, if we are instead willing to consider $\SSpan{K}$ as a \emph{virtual double category}, a concept introduced by Leinster \cite{Leinster:1999fc,Leinster:1999ge,Leinster:2002ge} under the name \emph{fc-multicategory} and renamed by Cruttwell and Shulman\cite[2.1]{CruttwellShulman:2010au}, then the substructure $\MMod{K}$ determined the modules is indeed a sub virtual double category of $\SSpan{K}$.
\end{rmk}

\begin{defn}[{virtual double category}] A \emph{virtual double category} consists of
\begin{itemize}
\item a category of \emph{objects} and \emph{vertical arrows}, which we call \emph{functors}
\item for any pair of objects $A, B$, a class of \emph{horizontal arrows} $A \prof B$, which we call \emph{modules}
\item \emph{cells}, with boundary depicted as follows
\begin{equation}\label{eq:generic-cell} \xymatrix{ A_0 \ar[d]_{f} \ar[r]|{\mid}^{E_1} & A_1 \ar[r]|{\mid}^{E_2}  \ar@{}[dr]|{\Downarrow} & \cdots \ar[r]|{\mid}^{E_n} & A_n \ar[d]^g \\ B_0 \ar[rrr]|{\mid}_{F} & & & B_n}\end{equation} including those whose horizontal source has length zero, in the case $A_0 = A_n$.
\item a \emph{composite cell}, for any configuration
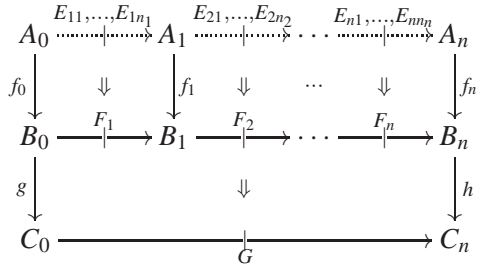
\[ \xymatrix@C=35pt{ A_0 \ar[d]_{f_0} \ar@{..>}[r]|{\mid}^{E_{11},\dots, E_{1n_1}} \ar@{}[dr]|{\Downarrow}  & A_1 \ar@{..>}[r]|{\mid}^{E_{21},\ldots, E_{2n_2}} \ar[d]^{f_1}  \ar@{}[dr]|{\Downarrow} & \cdots \ar@{..>}[r]|{\mid}^{E_{n1},\ldots, E_{nn_n}}  \ar@{}[d]|\cdots \ar@{}[dr]|{\Downarrow} & A_n \ar[d]^{f_n} \\ B_0 \ar[d]_{g} \ar[r]|{\mid}^{F_1} & B_1 \ar[r]|{\mid}^{F_2}  \ar@{}[dr]|{\Downarrow} & \cdots \ar[r]|{\mid}^{F_n} & B_n \ar[d]^h \\ C_0 \ar[rrr]|{\mid}_{G} & & & C_n}\] 
\item an \emph{identity cell} for every horizontal arrow 
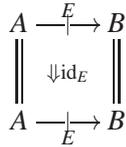
\[ \xymatrix{ A \ar@{=}[d] \ar[r]|{\mid}^E & B \ar@{=}[d] \ar@{}[dl]|{\Downarrow \id_E} \\ A \ar[r]|{\mid}_E & B}\] 
\end{itemize}
so that composition of cells is associative and unital in the usual multi-categorical sense.
\end{defn}

\begin{obs}[double categories are virtually such]
  Any double category is, in particular, a virtual double category. Specifically $\SSpan{K}$ becomes a virtual double category with the same classes of objects, vertical arrows, and horizontal arrows and with cells as depicted in~\eqref{eq:generic-cell} given as 2-cells 
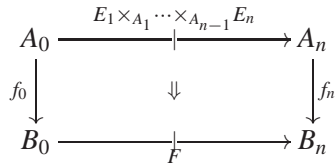
\[ \xymatrix@C=8em{ A_0 \ar[d]_{f_0} \ar[r]^{E_1\pbtimes{A_1}\cdots\pbtimes{A_{n-1}} E_n} |{\mid} & A_n \ar[d]^{f_n} \ar@{}[dl]|{\Downarrow} \\ B_0 \ar[r]_F|{\mid} & B_n}\] whose single vertical source is the (n-1)-fold pullback of the sequence of spans comprising the vertical source in~\eqref{eq:generic-cell}. In other words, such a cell is an isomorphism class of objects in the category $\hom_{f_0,f_n}(E_1\pbtimes{A_1}\cdots\pbtimes{A_{n-1}} E_n,F)$ of Definition~\ref{defn:double-iso-spans} or, equivalently, in the quasi-category $\map_{f_0,f_n}(E_1\pbtimes{A_1}\cdots\pbtimes{A_{n-1}} E_n,F)$ of Observation~\ref{obs:double-iso-spans}.

 The 0-fold pullback of an empty sequence of spans is simply an identity span $A\leftarrow A\to A$. So a cell with such an empty sequence as its vertical domain on the left of the following diagram
 \begin{equation*}
  \vcenter
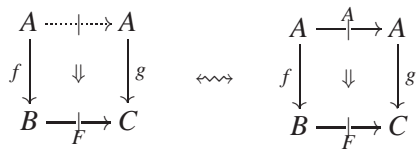
{\xymatrix{
    {A} \ar@{.>}[r]|{\mid}\ar[d]_f\ar@{}[dr]|{\Downarrow} & {A} \ar[d]^g \\
    {B} \ar[r]|{\mid}_F & {C}
  }} 
  \mkern20mu\leftrightsquigarrow\mkern20mu
  \vcenter{\xymatrix{
    {A} \ar[r]|{\mid}^A\ar[d]_f\ar@{}[dr]|{\Downarrow} & {A} \ar[d]^g \\
    {B} \ar[r]|{\mid}_F & {C}
  }}
 \end{equation*}
 is simply a 2-cell with vertical domain the identity span as on the right. This, in turn, is an isomorphism class of objects in the category $\hom_{f,g}(A,F)$ (or equivalently in the quasi-category $\map_{f,g}(A,F)$). On comparing the defining pullbacks in Definitions~\ref{defn:two-slices} and~\ref{defn:double-iso-spans} it becomes clear that $\hom_{f,g}(A,F)$ is isomorphic to the hom-category $\hom_{C\times B}(A,F)$ between objects $(g,f)\colon A\to C\times B$ and $(p_1,p_0)\colon F\tfib C\times B$ in the slice 2-category $\lcat{K}_2\slice{C\times B}$. In other words, such cells with empty vertical domains simply correspond to isomorphism classes of functors
 \begin{equation*}
   \xymatrix@C=1em@R=1.5em{
     {A}\ar[dr]_(0.35){(g,f)}\ar[rr]^k && {F}\ar@{->>}[dl]^(0.35){(p_1,p_0)} \\
     & {C\times B} &
   }
 \end{equation*}
in the slice 2-category $\lcat{K}_2\slice{C\times B}$.
\end{obs}

\begin{obs}[full sub virtual double categories]
  Suppose we are given classes of objects and of horizontal arrows between those objects in a virtual double category. We can then form a substructure comprising these chosen objects and horizontal arrows along with all vertical arrows between chosen objects and all cells for which the horizontal arrows in its domain list and its codomain are all in the chosen class. Now the only operations given in the structure of a virtual double category are vertical sources and targets, vertical identities, and vertical composition; so it is clear that this substructure is closed under all of these operations, and it follows easily that it inherits the structure of a virtual double category. We call this the {\em full\/} sub virtual double category \emph{determined\/} by the chosen classes of objects and horizontal arrows.
\end{obs}

\begin{defn}\label{defn:modules-virtual} The virtual double category $\MMod{K}$ of modules is defined to be the full sub virtual double category of $\SSpan{K}$ determined by the classes of all $\infty$-categories and modules between them. It has objects all $\infty$-categories, vertical arrows all functors, horizontal arrows modules, and cells the module maps of Definition~\ref{defn:module-map} (see Observation~\ref{obs:double-iso-spans}).
\end{defn}

\begin{defn}[composable modules]\label{defn:composable-modules}
We refer to a finite sequence of modules
\[ E_1 \colon A_0 \prof A_1, E_2 \colon A_1 \prof A_2, \ldots, \cdots E_n \colon A_{n-1}\prof A_n,\] in $\MMod{K}$ as a \emph{composable sequence of modules}; this just means that their horizontal sources and targets are compatible in the evident way. The horizontal composition operation described in Definition~\ref{defn:horiz-comp} yields an isofibration
\[E_1 \pbtimes{A_1} \cdots \pbtimes{A_{n-1}} E_n \tfib A_0 \times A_n,\] defined uniquely up to equivalence over $A_0 \times A_n$,  that is cartesian on the left and cartesian on the right. This isofibration is unlikely to define a groupoidal object of $\lcat{K}_2/A_0 \times A_n$ and hence does not define a module. When referring to the horizontal domains of cells in $\MMod{K}$, we frequently drop the subscripts and write simply $E_1 \times \cdots \times E_n$ for the composite isofibration. A cell with this domain is an $n$-\emph{ary cell}. Note that the cells in $\MMod{K}$ with unary source are precisely the module maps over a pair of functors introduced in Definition~\ref{defn:module-map}.
\end{defn}

\begin{obs}\label{obs:cells-above-a-comma} Recollection~\ref{rec:ess.unique.1-cell.ind}, which expresses 1-cell induction as a bijection between isomorphism classes of maps of spans whose codomain is a comma span and certain 2-cells in the homotopy 2-category,  provides an alternate characterization of cells in the virtual double category of modules whose codomain is a comma module. Explicitly, for any cospan  $B_0 \xrightarrow{k} C \xleftarrow{h} B_n$, there is a bijection 
\[ \vcenter{ \xymatrix{ A_0 \ar[d]_{f} \ar[r]|{\mid}^{E_1} & A_1 \ar[r]|{\mid}^{E_2}  \ar@{}[dr]|{\Downarrow} & \cdots \ar[r]|{\mid}^{E_n} & A_n \ar[d]^g \\ B_0 \ar[rrr]|{\mid}_{h \comma k} & & & B_n}}\qquad \leftrightsquigarrow\qquad \vcenter{ \xymatrix@!0@C=40pt@R=30pt{ & E_1 \pbtimes{A_1} \cdots \pbtimes{A_{n-1}} E_n \ar[dl] \ar[dr] \ar@{}[ddd]|{\displaystyle\Leftarrow} \\ A_0  \ar[d]_f & & A_n \ar[d]^g \\ B_0 \ar[dr]_k & & B_n \ar[dl]^h \\ & C}}\] between cells in $\MMod{K}$ whose codomain is the comma module $h \comma k \colon B_0 \prof B_n$ and 2-cells in the homotopy 2-category $\lcat{K}_2$ under the pullback of the spans encoding the domain modules and  over the cospan defining the comma module $h \comma k$. 
\end{obs}

\subsection{The virtual equipment of modules}

Proposition~\ref{prop:two-sided-pullback} tells us that modules in an $\infty$-cosmos can be pulled back. Given $E \colon A \prof B$ and functors $a \colon A' \to A$ and $b \colon B' \to B$, we write $E(b,a) \colon A' \prof B'$ for the pullback module 
\begin{equation}\label{eq:simplicial-pullback-of-module} \xymatrix{ E(b,a) \pbexcursion \ar@{->>}[d]_{(q',p')} \ar[r]^-\rho & E \ar@{->>}[d]^{(q,p)} \\ A' \times B' \ar[r]_{a \times b} & A \times B}\end{equation} The horizontal functor $\rho$ defines a cell in the virtual double category of modules with a universal property that we now describe.

\begin{prop}\label{prop:cartesian-pullback-cells} In $\MMod{K}$, the cell 
\[ \xymatrix{ A' \ar[d]_a \ar[r]|{\mid}^{E(b,a)} \ar@{}[dr]|{\Downarrow\rho} & B' \ar[d]^b \\ A \ar[r]|{\mid}_E & B}\]
defined by pulling back a module $E \colon A \prof B$ along functors $a \colon A' \to A$ and $b \colon B' \to B$ has the property that any cell as displayed on the left
\begin{equation}\label{eq:cartesian-cell-UP} \vcenter{ \xymatrix{  X_0 \ar[d]_{af} \ar[r]|{\mid}^{E_1} & X_1 \ar[r]|{\mid}^{E_2}  \ar@{}[dr]|{\Downarrow} & \cdots \ar[r]|{\mid}^{E_n} & X_n \ar[d]^{bg} \\ A \ar[rrr]_E|{\mid} & & & B}} \quad = \quad \vcenter{ \xymatrix{  X_0  \ar[d]_{f} \ar[r]|{\mid}^{E_1} & X_1 \ar[r]|{\mid}^{E_2}  \ar@{}[dr]|{\Downarrow \exists !} & \cdots \ar[r]|{\mid}^{E_n} & X_n \ar[d]^{g} \\ A' \ar[d]_a\ar[rrr]^{E(b,a)}|{\mid} & \ar@{}[dr]|{\Downarrow\rho} & & B' \ar[d]^b \\  A \ar[rrr]_E|{\mid} & & & B}}  \end{equation}  
factors uniquely as displayed on the right.
\end{prop}

Proposition~\ref{prop:cartesian-pullback-cells} asserts that $\rho$ is a \emph{cartesian cell} in $\MMod{K}$.

\begin{proof}
As in Lemma~\ref{lem:module-map-pullback-equivalence}, the simplicial pullback \eqref{eq:simplicial-pullback-of-module}, induces an equivalence of hom quasi-categories
\[ \map_{af,bg}(E_1 \times \cdots \times E_n, E) \simeq \map_{f,g}(E_1 \times \cdots \times E_n, E(b,a)). \qedhere \] 
\end{proof}

Each module $A^\cattwo \colon A \prof A$ defined by the arrow construction comes with a canonical cell with nullary source. Under the identification of Observation~\ref{obs:cells-above-a-comma}, this cell corresponds via 1-cell induction to the isomorphism class of maps of spans representing the identity 2-cell at the identity 1-cell of the object $A$.
\[ \vcenter{\xymatrix {  A \ar@{=}[d] \ar@{=}[r] \ar@{}[dr]|{\Downarrow\iota} & A \ar@{=}[d]  \\ A \ar[r]_{A^\cattwo}|{\mid}  & A}} \quad \leftrightsquigarrow \quad    \vcenter{
  \xymatrix@R=2.5em@C=0.8em{
    {A}\ar@/^2ex/[]!D(0.4);[d]!U(0.5)^{\id_A}
    \ar@/_2ex/!D(0.4);[d]!U(0.5)_{\id_A}
    \ar@{}[d]|{=} \\ {A}
  }} \quad = \quad   \vcenter{
    \xymatrix@R=2.5em@C=0.8em{
      {A}\ar[d]^-{j} \\
      {A^\cattwo}\ar@{->>}@/^2ex/[]!D(0.4);[d]!U(0.5)^{q_0}
      \ar@{->>}@/_2ex/!D(0.4);[d]!U(0.5)_{q_1}
      \ar@{}[d]|{\Leftarrow\psi} \\ {A}
    }} \] 
This cell also has a universal property in the virtual double category of modules.

\begin{prop}\label{prop:arrows-are-units} Any cell in the virtual double category of modules whose horizontal source includes the object $A$, as displayed on the left
\[  \vcenter{\xymatrix@C=15pt{ X \ar[d]_{f} \ar[r]|{\mid}^{E_1} & \cdots \ar[r]|-{\mid}^-{E_n} &  A  \ar@{}[d]|{\Downarrow}    \ar[r]|-{\mid}^-{F_1} & \cdots \ar[r]|{\mid}^{F_m} & Y \ar[d]^g \\ B \ar[rrrr]|{\mid}_{G} & &  &  & C}} \quad=\quad  \vcenter{ \xymatrix@C=20pt{ X \ar@{=}[d] \ar[r]|{\mid}^{E_1} \ar@{}[dr]|{\Downarrow\id_{E_1}{~}} & \cdots \ar@{=}@<-1.5ex>[d]^{\cdots} \ar@{=}@<1.5ex>[d] \ar[r]|-{\mid}^-{E_n} \ar@{}[dr]|{{~}{~}\Downarrow\id_{E_n}}&  A \ar@{=}[r] \ar@{=}[d] \ar@{}[dr]|{\Downarrow\iota}  & \ar@{=}[d]A \ar@{}[dr]|{\Downarrow\id_{F_1}{~}} \ar[r]|-{\mid}^-{F_1} & \cdots \ar@{=}@<-1.5ex>[d]^{\cdots} \ar@{=}@<1.5ex>[d] \ar[r]|{\mid}^{F_m} \ar@{}[dr]|{{~}{~}\Downarrow\id_{F_m}}& Y \ar@{=}[d]  \\ X \ar[d]_{f} \ar[r]|{\mid}^{E_1} & \cdots \ar[r]|-{\mid}^-{E_n} &  A \ar[r]|{\mid}^{A^\cattwo} \ar@{}[dr]|{\Downarrow\exists !}  & A  \ar[r]|-{\mid}^-{F_1} & \cdots \ar[r]|{\mid}^{F_m} & Y \ar[d]^g \\ B \ar[rrrrr]|{\mid}_{G} & & & &  & C}}\]
factors uniquely through $\iota$ as displayed on the right.
\end{prop}

Proposition~\ref{prop:arrows-are-units} asserts that $\iota$ is a \emph{cocartesian cell} in $\MMod{K}$.

\begin{proof}
In the case where both of the sequences $E_i$ and $F_j$ are empty, the Yoneda lemma, in the form of Proposition~\ref{prop:two-sided-yoneda}, and  Lemma~\ref{lem:module-map-pullback-equivalence} supply  an equivalence of quasi-categories
\[ \xymatrix{ \map_{f,g}(A^\cattwo, G) \simeq \map_{A \times A}(A^\cattwo, G(g,f)) \ar[r]^-{j^*}_{\simeq} & \map_{A \times A} (A, G(g,f)) \simeq \map_{f,g}(A,G) .}\] This equivalence descends to a bijection between isomorphism classes of objects, i.e., to a bijection between cells
\[ \vcenter{\xymatrix{A \ar[d]_f \ar[r]|{\mid}^{A^\cattwo} \ar@{}[dr]|{\Downarrow} & A \ar[d]^g \\ B \ar[r]|\mid_G & C}} \qquad \stackrel{\cong}{\mapsto}\qquad  \vcenter{\xymatrix{A \ar[d]_f \ar@{=}[r] \ar@{}[dr]|{\Downarrow} & A \ar[d]^g \\ B \ar[r]|\mid_G & C}} \] 
implemented by restricting along the cocartesian cell $\iota$. 

In general, write $(q,p) \colon E \tfib X \times A$ and $(s,r) \colon F \tfib A \times Y$ for the composite spans $E_1 \times \cdots \times E_n$ and $F_1 \times \cdots \times F_m$, which we take to be the identity span $A \leftarrow A \rightarrow A$ if the sequence of modules is empty. In the remaining cases, at least one of the sequences $E_i$ and $F_j$ is non-empty, so we may assume without loss of generality, by Lemma~\ref{lem:module-pullback-cartesian}, that $(q,p) \colon E \tfib X \times A$ is cartesian on the left and on the right. By Lemma~\ref{lem:sliced-adjunction-on-the-right},  the functor $i \colon E \to A \comma p$, which is isomorphic to the pullback $E \pbtimes{A} j$, admits a right adjoint $t$ over $X \times A$. This adjunction may be pulled back along $X \times s$ and pushed forward along $X \times r$ to define an adjunction 
\[ \adjdisplay E \pbtimes{A} j \pbtimes{A} F -| t\pbtimes{A} F : A \comma p \pbtimes{A} F \cong E \pbtimes{A} A^\cattwo \pbtimes{A} F -> E \pbtimes{A} F.\] over $X \times Y$. The module $G(g,f) \colon X \prof Y$ is a groupoidal object in the slice 2-category $\lcat{K}_2/X \times Y$. Therefore, the functor \[\map_{X \times Y}(-,G(g,f)) \colon (\lcat{K}_2/X\times Y)\op \to\qCat_2\] carries the fibered unit and counit 2-cells to isomorphisms. In particular,  the induced map \[ (E \pbtimes{A} j \pbtimes{A} F)^* \colon \map_{X \times Y}(E \pbtimes{A} A^\cattwo \pbtimes{A} F, G(g,f)) \to \map_{X \times Y}(E \pbtimes{A} F, G(g,f))\] defines an (adjoint) equivalence of quasi-categories. Passing to isomorphism classes of objects, we obtain the claimed bijection between cells in $\MMod{K}$.
\end{proof}

Propositions~\ref{prop:cartesian-pullback-cells} and \ref{prop:arrows-are-units} imply that the virtual double category of modules is a \emph{virtual equipment} in the sense introduced by Cruttwell and Shulman.

\begin{defn}[{\cite[\S 7]{CruttwellShulman:2010au}}]\label{defn:virtual-equipment} A \emph{virtual equipment} is a virtual double category such that 
\begin{enumerate}
\item For any module $E \colon A \prof B$ and pair of functors $a \colon A' \to A$ and $b \colon B' \to B$, there exists a module $E(b,a) \colon A' \prof B'$ together with a cartesian cell $\rho$ satisfying the universal property of Proposition~\ref{prop:cartesian-pullback-cells}.
\item Every object $A$ admits a \emph{unit} module $A^\cattwo \colon A \prof A$ equipped with a nullary cocartesian  cell $\iota$ satisfying the universal property of Proposition~\ref{prop:arrows-are-units}.
\end{enumerate}
\end{defn}

\begin{thm}\label{thm:virtual-equipment} The virtual double category $\MMod{K}$ of modules in an $\infty$-cosmos $\lcat{K}$ is a virtual equipment.\qed
\end{thm}

The \emph{virtual equipment of modules} in $\lcat{K}$ has a lot of pleasant properties, which follow formally from the axiomatization of Definition~\ref{defn:virtual-equipment}  \cite[\S 7]{CruttwellShulman:2010au}. These include Lemma~\ref{lem:comp-with-unit}, Lemma~\ref{lem:comp-with-unit-cells}, Theorem~\ref{thm:companion-conjoint}, Corollary~\ref{cor:companion-conjoint}, Lemma~\ref{lem:one-sided-rep-comp}, Corollary~\ref{cor:two-sided-rep-comp}, and the portion of Lemma \ref{lem:yoneda-embedding} describing bijections between cells in the virtual equipment.

However, rather than take these facts (whose proofs are hard to find in the literature) for granted and given the fact that the virtual equipment of modules in an $\infty$-cosmos is the only example that concerns us here, we find it more illuminating to give direct proofs. Many of our arguments are the formal ones but others make use of the particular structure of $\MMod{K}$, such as Observation~\ref{obs:cells-above-a-comma} and the fibered adjunctions of Lemma \ref{lem:sliced-adjunction-on-the-right}. Our efforts to this end in the remainder of this section aim to better acquaint the reader with the calculus of models between $\infty$-categories, as encapsulated by the virtual equipment of Theorem~\ref{thm:virtual-equipment}.

\subsection{Composition and units}

\begin{ntn}
To unclutter displayed diagrams, we adopt the convention that an unlabeled unary cell in a virtual equipment whose vertical arrows are identities and whose horizontal source and target agree is an identity cell. 

Cells whose vertical boundary functors are identities, and hence whose source and target spans lie between the same pair of $\infty$-categories, may be displayed inline using the notation $\mu \colon E_1 \times \cdots \times E_n \To E$. In the unary case, i.e., for ordinary module maps, this notation was already introduced in Definition~\ref{defn:1-cat-of-modules}. Whenever we write a cell in this form, our use of this notation implicitly asserts that:
\begin{itemize}
\item the modules $E_1,\ldots, E_n$ define a composable sequence, in the sense of Definition~\ref{defn:composable-modules},
\item the source spans $E_1 \times \cdots \times E_n$ and target module $E$ lie between the same pair of objects, $A_0$ and $A_n$, 
\item $\mu$ is a cell from $E_1,\ldots, E_n$ to $E$ over the identities, i.e., $\mu$ is an isomorphism class of objects in $\map_{A_0 \times A_n}(E_1 \times \cdots \times E_n, E)$.  
\end{itemize}
\end{ntn}

\begin{defn}[composition of modules]
A composable sequence of modules
\begin{equation}\label{eq:multi-domain} E_1 \colon A_0 \prof A_1, E_2 \colon A_1 \prof A_2, \ldots, \cdots E_n \colon A_{n-1}\prof A_n,\end{equation} \emph{admits a composite} if there exists a module $E \colon A_0 \prof A_n$ and a cell 
\begin{equation}\label{eq:generic-composite}\xymatrix{ A_0 \ar@{=}[d] \ar[r]|{\mid}^{E_1} & A_1 \ar[r]|{\mid}^{E_2}  \ar@{}[dr]|{\Downarrow\mu} & \cdots \ar[r]|{\mid}^{E_n} & A_n \ar@{=}[d]\\ A_0 \ar[rrr]|{\mid}_{E} & & & A_n}\end{equation} that is an cocartesian cell in the virtual double category of modules: any cell of the form
\[  \vcenter{\xymatrix@C=15pt{ X \ar[d]_{f} \ar[r]|{\mid}^{F_1} & \cdots \ar[r]|-{\mid}^-{F_k} &  A_0 \ar[r]|{\mid}^{E_1} & \cdots \ar@{}[d]|{\Downarrow} \ar[r]|\mid^{E_n} & A_n  \ar[r]|-{\mid}^-{G_1} & \cdots \ar[r]|{\mid}^{G_m} & Y \ar[d]^g \\ B \ar[rrrrrr]|{\mid}_{H} & &  & & &  & C}} \]
factors uniquely along the cell $\mu$ together with the identity cells for the modules $F_i$ and $G_j$
\[ \vcenter{ \xymatrix@C=15pt{ X \ar@{=}[d] \ar[r]|{\mid}^{F_1}  & \cdots \ar@{=}@<-1.5ex>[d]^{\cdots} \ar@{=}@<1.5ex>[d] \ar[r]|-{\mid}^-{F_k} &  A_0\ar@{=}[d] \ar[r]|{\mid}^{E_1} & \cdots \ar[r]|{\mid}^{E_n}  \ar@{}[d]|{\Downarrow\mu}   &  \ar@{=}[d]A_n  \ar[r]|-{\mid}^-{G_1} & \cdots \ar@{=}@<-1.5ex>[d]^{\cdots} \ar@{=}@<1.5ex>[d] \ar[r]|{\mid}^{G_m} & Y \ar@{=}[d]  \\ X \ar[d]_{f} \ar[r]|{\mid}^{F_1} & \cdots \ar[r]|-{\mid}^-{F_k} &  A \ar[rr]|{\mid}^{E}  & \ar@{}[d]|{\exists !\Downarrow}   & B  \ar[r]|-{\mid}^-{G_1} & \cdots \ar[r]|{\mid}^{G_m} & Y \ar[d]^g \\ B \ar[rrrrrr]|{\mid}_{H} & &&& &   & C}}\]

Thus, a composite $\mu \colon E_1 \times \cdots \times E_n \To E$ can be used to reduce the domain of a cell  by replacing any occurrence of a sequence $E_1 \times \cdots \times E_n$ from $A_0$ to $A_n$ with the single module $E$.  Particularly in the case of binary composites, we write $E_1 \otimes E_2$ to denote the composite of $E_1$ and $E_2$, a module equipped with a binary cocartesian cell $E_1 \times E_2 \To E_1 \otimes E_2$.
\end{defn}

\begin{obs}[nullary and unary composites]\label{obs:unary-comp}
Proposition~\ref{prop:arrows-are-units}  asserts that arrow $\infty$-categories act as nullary composites in $\MMod{K}$. It's easy to see that a unary cell $\mu \colon E \To F$ between modules is a composite if and only if it is an isomorphism in the vertical 2-category of $\MMod{K}$, i.e., if and only if the modules $E$ and $F$ are equivalent as spans. 
\end{obs}

\begin{obs}[associativity of composition]\label{obs:comp-assoc}
Suppose the cells $\mu_i \colon E_{i1} \times \cdots \times E_{in_i} \To E_i$, for $i=1,\ldots, n$, exhibit each $E_i$ as a composite of the corresponding $E_{ij}$, and suppose further that the $E_i$ define a composable sequence of modules \eqref{eq:multi-domain}. If $\mu \colon E_1 \times\cdots \times E_n \To E$ exhibits $E$ as a composite of the $E_i$, then 
\[ \xymatrix@C=40pt{ E_{11} \times \cdots \times E_{nn_n} \ar@{=>}[r]^-{\mu_1 \times \cdots \times \mu_n} & E_1 \times \cdots \times E_n \ar@{=>}[r]^-{\mu} & E}\] exhibits $E$ as a composite of $E_{11} \times \cdots \times E_{nn_n}$. The required bijection factors as a composite of $n+1$-bijections induced by the maps $\mu_1,\ldots, \mu_n,\mu$.
\end{obs}

\begin{obs}[left cancelation of composites]\label{obs:comp-cancel}
Suppose the cells $\mu_i \colon E_{i1} \times \cdots \times E_{in_i} \To E_i$, for $i=1,\ldots, n$, exhibit each $E_i$ as a composite of the corresponding $E_{ij}$, and suppose further that the $E_i$ define a composable sequence of modules \eqref{eq:multi-domain}. If $\mu \colon E_1 \times\cdots \times E_n \To E$ is any cell so that
\[ \xymatrix@C=40pt{ E_{11} \times \cdots \times E_{nn_n} \ar@{=>}[r]^-{\mu_1 \times \cdots \times \mu_n} & E_1 \times \cdots \times E_n \ar@{=>}[r]^-{\mu} & E}\] exhibits $E$ as a composite of $E_{11} \times \cdots \times E_{nn_n}$, then $\mu \colon E_1 \times \cdots \times E_n \To E$ exhibits $E$ as a composite of $E_1 \times \cdots \times E_n$. The required bijection composes with the bijections supplied by the maps  $\mu_1,\ldots, \mu_n$ to a bijection, and is thus itself a bijection by the 2-of-3 property for isomorphisms.
\end{obs} 

\begin{obs}\label{obs:comp-proof-strategy}
On account of the universal property described by Proposition~\ref{prop:cartesian-pullback-cells} of the cells encoding pullback modules, to prove that a cell  \eqref{eq:generic-composite} is a composite, it suffices to consider cells whose vertical 1-morphisms are all identities.

To prove that a cell \eqref{eq:generic-composite} is a composite in $\MMod{K}$, we frequently exhibit a stronger  universal property. Writing $F \tfib B \times A_0$ and $ G \tfib A_n \times C$ for the pullbacks of finite composable sequences $F_1,\ldots, F_k$ and $G_1, \ldots, G_m$ of modules, it (more than) suffices to show that restriction along $\mu$ induces an equivalence of quasi-categories
\[ \map_{B \times C}(F \pbtimes{A_0} E \pbtimes{A_n} G,H) \xrightarrow{\map_{B \times C}(F \pbtimes{A_0} \mu \pbtimes{A_n} G,H)} \map_{B \times C}(F \pbtimes{A_0} E_1 \times \cdots \times E_n \pbtimes{A_n} G, H)\] for every module $H \colon B \prof C$. This equivalence of hom quasi-categories induces a bijection between sets of cells whose vertical boundaries are comprised of identities. This strategy was employed in the proof of Proposition~\ref{prop:arrows-are-units}.
\end{obs}

\begin{lem}[composites with units]\label{lem:comp-with-unit}
Given any module $E \colon A \prof B$, the unique cell $\circ \colon A^\cattwo \times E \times B^\cattwo \To E$ defined using the universal properties of the cocartesian cells associated to the unit modules
\begin{equation}\label{eq:comp-with-unit} \vcenter{\xymatrix{ A \ar@{=}[d] \ar@{=}[r] \ar@{}[dr]|{\Downarrow\iota} & \ar@{=}[d] A \ar[r]|{\mid}^E & B \ar@{=}[d] \ar@{=}[r] \ar@{}[dr]|{\Downarrow\iota} & B \ar@{=}[d]   \\A \ar@{=}[d] \ar[r]|{\mid}^{A^\cattwo} & A \ar[r]|{\mid}^E\ar@{}[dr]|{\Downarrow\circ} & B \ar[r]|{\mid}^{B^\cattwo} & B \ar@{=}[d] \\ A \ar[rrr]|{\mid}_E & & &  B}} \qquad \defeq \qquad \vcenter{ \xymatrix{ A \ar@{=}[d] \ar[r]|{\mid}^E & B \ar@{=}[d] \\ A \ar[r]|{\mid}_E & B}} \end{equation} displays $E$ as a composite of $E$ with the units $A^\cattwo$ and $B^\cattwo$ at its domain and codomain objects. 
\end{lem}
\begin{proof}
The result is immediate from Proposition~\ref{prop:arrows-are-units} and Observation~\ref{obs:comp-cancel}.
\end{proof}

\begin{obs}\label{obs:comp-with-unit-over-comma}
In the case of a comma module $h \comma k \colon  A \prof B$ associated to a cospan $A \xrightarrow{k} C \xleftarrow{h} B$, the cell $\circ \colon A^\cattwo \times h \comma k \times B^\cattwo \To h \comma k$ in $\MMod{K}$ defined by Lemma~\ref{lem:comp-with-unit} corresponds, under the identification of Observation~\ref{obs:cells-above-a-comma}, to the pasting diagram
\[ \xymatrix@=15pt{ & & & {}\save[]+<0cm,.5pc>*{A^\cattwo \pbtimes{A} h \comma k \pbtimes{B} B^\cattwo}\ar[dll] \ar[drr] \ar[d] \restore \\ 
 & A^\cattwo \ar@{->>}[dl]_{q_1} \ar@{->>}[dr]^{q_0} \ar@{}[d]|(.4){\psi}|(.6){\Leftarrow} & & h \comma k  \ar@{->>}[dl]_{p_1} \ar@{->>}[dr]^{p_0} \ar@{}[dd]|{\Leftarrow} & & B^\cattwo  \ar@{->>}[dl]_{q_1} \ar@{->>}[dr]^{q_0} \ar@{}[d]|(.4){\psi}|(.6){\Leftarrow}  \\ A \ar@{=}[rr] & & A \ar[dr]_k & & B \ar[dl]^h \ar@{=}[rr]& & B \\ & & & C}\]
in $\lcat{K}_2$.
\end{obs}

\begin{defn}[unit cells]\label{defn:unit-cells}
Using the unit modules in $\MMod{K}$, we can define unit cells 
\[ \xymatrix{ A \ar[d]_f \ar[r]|\mid^{A^\cattwo} \ar@{}[dr]|{\Downarrow f^\cattwo} & A \ar[d]^f \\ B \ar[r]|\mid_{B^\cattwo} & B}\] 
associated to a (vertical) functor $f \colon A \to B$ between $\infty$-categories. By the universal property of the cocartesian cell associated to the unit $A^\cattwo$, it suffices to define the left-hand composite
\[ \vcenter{\xymatrix{ A \ar@{=}[d] \ar@{=}[r] \ar@{}[dr]|{\Downarrow\iota} & A \ar@{=}[d] & & A \ar[d]_f \ar@{=}[r] & A \ar[d]^f \\ A \ar[d]_f \ar[r]|\mid^{A^\cattwo} \ar@{}[dr]|{\Downarrow f^\cattwo} & A \ar[d]^f  & \defeq & B \ar@{=}[d] \ar@{=}[r] \ar@{}[dr]|{\Downarrow \iota} & B \ar@{=}[d] \\ B \ar[r]|\mid_{B^\cattwo} & B & & B \ar[r]|\mid_{B^\cattwo} & B}} \qquad \leftrightsquigarrow \qquad   \vcenter{
  \xymatrix@R=2.5em@C=0.8em{
    {A}\ar@/^2ex/[]!D(0.4);[d]!U(0.5)^{f}
    \ar@/_2ex/!D(0.4);[d]!U(0.5)_{f}
    \ar@{}[d]|(.4){\id_f}|(.6){\Leftarrow} \\ {B}
  }} \] and we take this to be the composite of the cocartesian cell associated to the unit $B^\cattwo$ with a nullary morphism. Applying Observation~\ref{obs:cells-above-a-comma} both composites correspond to the identity 2-cell $\id_f \colon f \To f \colon A \to B$ in the homotopy 2-category $\lcat{K}_2$.
\end{defn}

\begin{lem}[composite with unit cells]\label{lem:comp-with-unit-cells} For any cell $\alpha$ whose boundary is of the form displayed below-left, the composite cell
\[  \vcenter{\xymatrix{ A \ar@{=}[d] \ar@{}[dr]|{\Downarrow\iota} \ar@{=}[r]  & A \ar@{=}[d]  \ar[r]|{\mid}^{E_1} & A_1\ar@{=}[d]  \ar[r]|{\mid}^{E_2}  & \cdots \ar@{=}@<-1.5ex>[d]^{\cdots} \ar@{=}@<1.5ex>[d] \ar[r]|{\mid}^{E_n}  & C \ar@{=}[d]  \ar@{}[dr]|{\Downarrow\iota} \ar@{=}[r] & C \ar@{=}[d]   \\ A \ar[d]_f \ar[r]|{\mid}^{A^\cattwo} \ar@{}[dr]|{\Downarrow f^\cattwo} &  A \ar[d]^{f} \ar[r]|{\mid}^{E_1} & A_1 \ar[r]|{\mid}^{E_2}  \ar@{}[dr]|{\Downarrow\alpha} & \cdots \ar[r]|{\mid}^{E_n} & C \ar[d]_g \ar[r]|\mid^{C^\cattwo} \ar@{}[dr]|{\Downarrow g^\cattwo} & C \ar[d]^g  \\ B \ar@{=}[d] \ar[r]|\mid_{B^\cattwo} & B \ar[rrr]|{\mid}_{E} & \ar@{}[dr]|{\Downarrow\circ} & & D \ar[r]|\mid_{D^\cattwo} & D \ar@{=}[d] \\ B \ar[rrrrr]|{\mid}_{E} &  & & & & D}}
 =   \vcenter{\xymatrix{ A \ar[d]_{f} \ar[r]|{\mid}^{E_1} & A_1 \ar[r]|{\mid}^{E_2}  \ar@{}[dr]|{\Downarrow\alpha} & \cdots \ar[r]|{\mid}^{E_n} & C \ar[d]^g \\ B \ar[rrr]|{\mid}_{E} & & & D}}\]
equals $\alpha$.
\end{lem}
\begin{proof}
 By Definition~\ref{defn:unit-cells} and the identity laws in a virtual double category, the left-hand side  is the composite cell displayed on the left
 \[  \vcenter{\xymatrix{ A \ar[d]_f \ar@{=}[r]  &  A \ar[d]^{f} \ar[r]|{\mid}^{E_1} & A_1 \ar[r]|{\mid}^{E_2}  \ar@{}[dr]|{\Downarrow\alpha} & \cdots \ar[r]|{\mid}^{E_n} & C \ar[d]_g \ar@{=}[r] & C \ar[d]^g  \\B  \ar@{=}[d] \ar@{=}[r] \ar@{}[dr]|{\Downarrow \iota} &  B \ar@{=}[d] \ar[rrr]|{\mid}_{E} &   \ar@{}[dr]|{\Downarrow\id_E} &  & D \ar@{=}[d] \ar@{=}[r] \ar@{}[dr]|{\Downarrow \iota} & D \ar@{=}[d]  \\ B \ar@{=}[d] \ar[r]|\mid_{B^\cattwo} & B \ar[rrr]|{\mid}_{E} & \ar@{}[dr]|{\Downarrow\circ} & & D \ar[r]|\mid_{D^\cattwo} & D \ar@{=}[d] \\ B \ar[rrrrr]|{\mid}_{E} &  & & & & D}} = \vcenter{\xymatrix{   A \ar[d]_{f} \ar[r]|{\mid}^{E_1} & A_1 \ar[r]|{\mid}^{E_2}  \ar@{}[dr]|{\Downarrow\alpha} & \cdots \ar[r]|{\mid}^{E_n} & C \ar[d]^g   \\  B \ar@{=}[d] \ar[rrr]|{\mid}_{E} &   \ar@{}[dr]|{\Downarrow\id_E} &  & D \ar@{=}[d]  \\ B \ar[rrr]|{\mid}_{E} & & & D}} \] which equals the composite cell displayed on the right by the definition \eqref{eq:comp-with-unit} of $\circ \colon B^\cattwo \times E \times D^\cattwo \To E$. Applying the virtual double category identity laws, the right-hand side equals $\alpha$.
\end{proof}

\subsection{Representable modules}

The restriction and unit cells present in any virtual equipment imply that any vertical morphism has a pair of associated horizontal morphisms together with cells that have universal properties similar to companions and conjoints in an ordinary double category. In $\MMod{K}$, the horizontal morphisms associated to a functor $f \colon A \to B$ are the covariant $B \comma f \colon A \prof B$ and contravariant $f \comma B \colon B \prof A$ represented modules. This section is devoted to exploring their properties. 

\def\hquad{\hskip0.5em\relax}

\begin{defn}\label{defn:companion-conjoint}
The covariant and contravariant representable modules associated to a functor $f \colon A \to B$ are defined by pulling back the module $B^\cattwo \colon B \prof B$. Thus Proposition~\ref{prop:cartesian-pullback-cells} implies that the cells, defined using the identification of Observation~\ref{obs:cells-above-a-comma} by the pasting diagrams
\[ \vcenter{\xymatrix{ A \ar[d]_f \ar[r]|{\mid}^{B \comma f} & B \ar@{=}[d] \ar@{}[dl]|{\Downarrow\rho} \\ B \ar[r]|{\mid}_{B^\cattwo} & B}} \hquad \leftrightsquigarrow \hquad
\vcenter{ \xymatrix@R=15pt@C=10pt{ & B \comma f \ar[dl] \ar[dr] \ar@{}[d]|(.6){\Leftarrow} \\ A \ar[d]_f \ar[rr]_f & & B \ar@{=}[d] \\ B \ar@{=}[rr] & &B}} \qquad\quad \vcenter{\xymatrix{ B \ar@{=}[d] \ar[r]|{\mid}^{f \comma B} & A \ar[d]^f \ar@{}[dl]|{\Downarrow\rho} \\ B \ar[r]|{\mid}_{B^\cattwo} & B}} \quad \leftrightsquigarrow \quad
\vcenter{ \xymatrix@R=15pt@C=10pt{ & f \comma B \ar[dl] \ar[dr] \ar@{}[d]|(.6){\Leftarrow} \\ B \ar@{=}[d]  & & A \ar[ll]^f \ar[d]^f \\ B \ar@{=}[rr] & &B}}\] are cartesian cells in the virtual equipment of modules. 

We also have cells
\[ \vcenter{ \xymatrix{ A \ar[r]|{\mid}^{A^\cattwo} \ar@{=}[d] & A \ar[d]^f \ar@{}[dl]|{\Downarrow\kappa }  \\ A \ar[r]|\mid_{B \comma f} & B}} \hquad\leftrightsquigarrow\hquad \vcenter{ \xymatrix@R=15pt@C=10pt{ & A^\cattwo \ar@{->>}[dl] \ar@{->>}[dr] \ar@{}[d]|(.6){\Leftarrow} \\ A \ar@{=}[rr] \ar@{=}[d] & & A \ar[d]^f \\ A \ar[rr]_f & & B}}\qquad\quad\vcenter{ \xymatrix{ A \ar[r]|{\mid}^{A^\cattwo} \ar[d]_f & A \ar@{=}[d] \ar@{}[dl]|{\Downarrow\kappa }  \\ B \ar[r]|\mid_{f \comma B} & A}} \hquad\leftrightsquigarrow\hquad \vcenter{ \xymatrix@R=15pt@C=10pt{ & A^\cattwo \ar@{->>}[dl] \ar@{->>}[dr] \ar@{}[d]|(.6){\Leftarrow} \\ A \ar@{=}[rr] \ar[d]_f & & A \ar@{=}[d] \\ B & & A \ar[ll]^f }}\] which compose vertically to the unit cell $f^\cattwo$ associated to the functor $f$, introduced in Definition~\ref{defn:unit-cells}
\[ \vcenter{ \xymatrix{ A \ar[r]|{\mid}^{A^\cattwo} \ar@{=}[d] & A \ar[d]^f \ar@{}[dl]|{\Downarrow\kappa }  \\ A \ar[d]_f \ar[r]|{\mid}^{B \comma f} & B \ar@{=}[d] \ar@{}[dl]|{\Downarrow\rho} \\ B \ar[r]|{\mid}_{B^\cattwo} & B}} \hquad=\hquad \vcenter{ \xymatrix{ A \ar[r]|{\mid}^{A^\cattwo} \ar[d]_f & A \ar@{=}[d] \ar@{}[dl]|{\Downarrow\kappa }  \\B \ar@{=}[d] \ar[r]|{\mid}^{f \comma B} & A \ar[d]^f \ar@{}[dl]|{\Downarrow\rho} \\ B \ar[r]|{\mid}_{B^\cattwo} & B}} \hquad=\hquad \vcenter{ \xymatrix{A \ar[r]|{\mid}^{A^\cattwo} \ar[d]_f & A  \ar[d]^f \ar@{}[dl]|{\Downarrow f^{\cattwo}} \\ B \ar[r]|{\mid}_{B^\cattwo} & B}} \hquad \leftrightsquigarrow \hquad  \vcenter{
    \xymatrix@R=2.5em@C=0.8em{
         {A^\cattwo}\ar@{->>}@/^2ex/[]!D(0.4);[d]!U(0.5)^{q_0}
      \ar@{->>}@/_2ex/!D(0.4);[d]!U(0.5)_{q_1}
      \ar@{}[d]|{\Leftarrow\psi} \\ {A}  \ar[d]^-{f} \\ {B}
    }}  =   \vcenter{
    \xymatrix@R=2.5em@C=0.8em{
      {A^\cattwo}\ar[d]^-{f^\cattwo} \\
      {B^\cattwo}\ar@{->>}@/^2ex/[]!D(0.4);[d]!U(0.5)^{q_0}
      \ar@{->>}@/_2ex/!D(0.4);[d]!U(0.5)_{q_1}
      \ar@{}[d]|{\Leftarrow\psi} \\ {B}
    }} 
\] 

Moreover, by Observation~\ref{obs:comp-with-unit-over-comma}, we have identities
\[ \vcenter{\xymatrix{ A \ar@{=}[d] \ar[r]|{\mid}^{A^\cattwo} \ar@{}[dr]|{\Downarrow\kappa} & A \ar[r]|{\mid}^{B \comma f} \ar[d]^f & B \ar@{=}[d] \ar@{}[dl]|{\Downarrow\rho} \\ A \ar[r]|{\mid}^{B \comma f} \ar@{=}[d] & B \ar[r]|{\mid}^{B^\cattwo} \ar@{}[d]|{\Downarrow\circ} & B \ar@{=}[d] \\ A \ar[rr]|{\mid}_{B \comma f} & &B}} \hquad \leftrightsquigarrow
 \hquad \vcenter{ \xymatrix@R=15pt@C=10pt{ & A^\cattwo \ar@{->>}[dr]\ar@{->>}[dl] \ar@{}[d]|(.6){\Leftarrow} & & B \comma f \ar@{->>}[dl] \ar@{->>}[dr] \ar@{}[d]|(.6){\Leftarrow} \\ A \ar@{=}[rr] \ar@{=}[d] & & A \ar[rr]_f & & B \ar@{=}[d] \\ A \ar[rrrr]_f & &  & & B}} \hquad \leftrightsquigarrow 
\hquad \vcenter{\xymatrix{ A \ar@{=}[d] \ar[r]|{\mid}^{A^\cattwo} & A \ar[r]|{\mid}^{B \comma f} \ar@{}[d]|{\Downarrow\circ} & B \ar@{=}[d] \\ A \ar[rr]|{\mid}_{B \comma f} & &B}} \] 
and dually
\[ \vcenter{\xymatrix{ B \ar@{=}[d] \ar[r]|{\mid}^{f \comma B} \ar@{}[dr]|{\Downarrow\rho} & A \ar[r]|{\mid}^{A^\cattwo} \ar[d]^f & A \ar@{=}[d] \ar@{}[dl]|{\Downarrow\kappa} \\ B \ar[r]|{\mid}^{B^\cattwo} \ar@{=}[d] & B \ar[r]|{\mid}^{f \comma B} \ar@{}[d]|{\Downarrow\circ} & A \ar@{=}[d] \\ B \ar[rr]|{\mid}_{f \comma B} & &A}} \quad = \quad \vcenter{\xymatrix{ B \ar[r]|{\mid}^{f \comma B}  \ar@{=}[d] & A \ar[r]|{\mid}^{A^\cattwo} \ar@{}[d]|{\Downarrow\circ} & A\ar@{=}[d] \\ B \ar[rr]|{\mid}_{f \comma B} & &A}} \] 
relating these canonical cells to the composition cells introduced in Lemma~\ref{lem:comp-with-unit}. To summarize this situation, we say that these cells display $f \colon A \to B$ and $B \comma f \colon A \prof B$ as \emph{companions} and display $f \colon A \to B$ and $f \comma B \colon B \prof A$ as \emph{conjoints} in a sense appropriate for a virtual equipment. 
\end{defn}

\begin{thm}\label{thm:companion-conjoint} In the virtual equipment of modules, there are bijections between cells 
\[ \xymatrix{ & & & B \ar@{=}[d] \ar[r]|\mid^{f \comma B} & A \ar[r]|\mid^E \ar@{}[d]|{\Downarrow} & C \ar[d]^g  \\ A \ar[d]_f \ar[r]|{\mid}^{E} \ar@{}[dr]|{\Downarrow\alpha} & C \ar[d]^g &  \ar@{}[d]|{\displaystyle\leftrightsquigarrow} & B \ar[rr]|{\mid}_F & & D & \ar@{}[d]|{\displaystyle\leftrightsquigarrow}  & B \ar@{=}[d] \ar[r]|{\mid}^{f \comma B} & A \ar[r]|{\mid}^E\ar@{}[dr]|{\Downarrow\beta} & C \ar[r]|{\mid}^{D \comma g} & D \ar@{=}[d] \\ B \ar[r]|{\mid}_F & D &  & A \ar[d]_f \ar[r]|\mid^{E} & C \ar[r]|\mid^{D \comma g} \ar@{}[d]|{\Downarrow} & D \ar@{=}[d] & & B \ar[rrr]|{\mid}_F & & &  D\\ & & & B \ar[rr]|{\mid}_F & & D }\] implemented by composing with the canonical cells $\kappa$ and $\rho$ and with the composition and nullary cells associated with the units. 
\end{thm}
\begin{proof}
The composite bijection carries the cells $\alpha$ and $\beta$ to the cells displayed on the left and right, respectively:
\[ \hat\alpha \defeq \vcenter{ \xymatrix{ B  \ar[r]|{\mid}^{f \comma B} \ar@{=}[d] & A \ar@{}[dl]|{\Downarrow\rho}  \ar[r]|{\mid}^E \ar[d]_f  \ar@{}[dr]|{\Downarrow\alpha} & C \ar[d]^g  \ar[r]|{\mid}^{D \comma g} & D \ar@{=}[d] \ar@{}[dl]|{\Downarrow\rho} \\ B \ar@{=}[d] \ar[r]|{\mid}^{B^\cattwo} & B \ar[r]|{\mid}^F \ar@{}[dr]|{\Downarrow\circ} & D \ar[r]|{\mid}^{D^\cattwo} & D \ar@{=}[d] \\ B \ar[rrr]|{\mid}_F & & & D}} \qquad\qquad \bar\beta \defeq \vcenter{\xymatrix{   A \ar@{=}[d] \ar@{=}[r] & A \ar@{=}[d] \ar@{}[dl]|{\Downarrow \iota} \ar[r]|{\mid}^E & C \ar@{=}[d] \ar@{=}[r] & C \ar@{=}[d] \ar@{}[dl]|{\Downarrow \iota}\\ A \ar[d]_f \ar[r]|{\mid}^{A^\cattwo} & A \ar@{=}[d] \ar[r]|{\mid}^E  \ar@{}[dl]|{\Downarrow\kappa}& C \ar@{=}[d] \ar[r]|{\mid}^{C^\cattwo} & C \ar[d]^g \ar@{}[dl]|{\Downarrow\kappa} \\ B \ar@{=}[d] \ar[r]|{\mid}^{f \comma B} & A \ar[r]|{\mid}^E\ar@{}[dr]|{\Downarrow\beta} & C \ar[r]|{\mid}^{D \comma g} & D \ar@{=}[d] \\  B \ar[rrr]|{\mid}_F & & &  D}} \]
We have
\[ \bar{\hat{\alpha}} \defeq \vcenter{ \xymatrix@=2em{ A \ar@{=}[d] \ar@{=}[r] & A \ar@{=}[d] \ar@{}[dl]|{\Downarrow \iota} \ar[r]|{\mid}^E & C \ar@{=}[d] \ar@{=}[r] & C \ar@{=}[d] \ar@{}[dl]|{\Downarrow \iota}\\ A \ar[d]_f \ar[r]|{\mid}^{A^\cattwo} & A \ar@{=}[d] \ar[r]|{\mid}^E  \ar@{}[dl]|{\Downarrow\kappa}& C \ar@{=}[d] \ar[r]|{\mid}^{C^\cattwo} & C \ar[d]^g \ar@{}[dl]|{\Downarrow\kappa} \\ B  \ar[r]|{\mid}^{f \comma B} \ar@{=}[d] & A \ar@{}[dl]|{\Downarrow\rho}  \ar[r]|{\mid}^E \ar[d]_f  \ar@{}[dr]|{\Downarrow\alpha} & C \ar[d]^g  \ar[r]|{\mid}^{D \comma g} & D \ar@{=}[d] \ar@{}[dl]|{\Downarrow\rho} \\ B \ar@{=}[d] \ar[r]|{\mid}^{B^\cattwo} & B \ar[r]|{\mid}^F \ar@{}[dr]|{\Downarrow\circ} & D \ar[r]|{\mid}^{D^\cattwo} & D \ar@{=}[d] \\ B \ar[rrr]|{\mid}_F & & & D}} = \vcenter{ \xymatrix@=2em{ A \ar@{=}[d] \ar@{=}[r] & A \ar@{=}[d] \ar@{}[dl]|{\Downarrow \iota} \ar[r]|{\mid}^E & C \ar@{=}[d] \ar@{=}[r] & C \ar@{=}[d] \ar@{}[dl]|{\Downarrow \iota}\\ A \ar[d]_f \ar[r]|{\mid}^{A^\cattwo} & A \ar[d]^f \ar[r]|{\mid}^E \ar@{}[dr]|{\Downarrow\alpha} \ar@{}[dl]|{\Downarrow f^\cattwo}& C \ar[d]_g \ar[r]|{\mid}^{C^\cattwo} & C \ar[d]^g \ar@{}[dl]|{\Downarrow g^\cattwo} \\ B \ar@{=}[d] \ar[r]|{\mid}^{B^\cattwo} & B \ar[r]|{\mid}^F \ar@{}[dr]|{\Downarrow\circ} & D \ar[r]|{\mid}^{D^\cattwo} & D \ar@{=}[d] \\ B \ar[rrr]|{\mid}_F & & & D}}  =  \vcenter{\xymatrix@=2em{ A \ar[d]_f \ar[r]|{\mid}^{E} \ar@{}[dr]|{\Downarrow\alpha} & C \ar[d]^g  \\ B \ar[r]|{\mid}_F & D }}\] 
by applying the companion and conjoint identities and Lemma~\ref{lem:comp-with-unit-cells}.

The other composite is displayed below-left:
\[  \vcenter{\xymatrix@=1.8em{ B \ar@{=}[d] \ar[r]|{\mid}^{f \comma B} &  A \ar@{=}[d] \ar@{=}[r] & A \ar@{=}[d] \ar@{}[dl]|{\Downarrow \iota} \ar[r]|{\mid}^E & C \ar@{=}[d] \ar@{=}[r] & C \ar@{=}[d] \ar@{}[dl]|{\Downarrow \iota} \ar[r]|{\mid}^{D \comma g} & D \ar@{=}[d] \\  B \ar@{=}[d] \ar[r]|{\mid}^{f \comma B} \ar@{}[dr]|{\Downarrow\rho} &  A \ar[d]_f \ar[r]|{\mid}^{A^\cattwo} & A \ar@{=}[d] \ar[r]|{\mid}^E  \ar@{}[dl]|{\Downarrow\kappa}& C \ar@{=}[d] \ar[r]|{\mid}^{C^\cattwo} & C \ar[d]^g \ar@{}[dl]|{\Downarrow\kappa} \ar[r]|{\mid}^{D \comma g} & D \ar@{=}[d] \ar@{}[dl]|{\Downarrow\rho}  \\ B \ar@{=}[d]  \ar[r]|{\mid}^{B^\cattwo} & B \ar@{=}[d] \ar[r]|{\mid}^{f \comma B} & A \ar[r]|{\mid}^E\ar@{}[dr]|{\Downarrow\beta} & C \ar[r]|{\mid}^{D \comma g} & D \ar[r]|{\mid}^{D^\cattwo} \ar@{=}[d] & D \ar@{=}[d] \\  B \ar@{=}[d]  \ar[r]|{\mid}^{B^\cattwo} & B \ar[rrr]|{\mid}_F & \ar@{}[dr]|{\Downarrow\circ} & & D   \ar[r]|{\mid}^{D^\cattwo} & D \ar@{=}[d] \\ B \ar[rrrrr]|{\mid}_F & & &  & & D }} = \vcenter{\xymatrix@=1.8em{ B \ar@{=}[d] \ar[r]|{\mid}^{f \comma B} &  A \ar@{=}[d] \ar@{=}[r] & A \ar@{=}[d] \ar@{}[dl]|{\Downarrow \iota} \ar[r]|{\mid}^E & C \ar@{=}[d] \ar@{=}[r] & C \ar@{=}[d] \ar@{}[dl]|{\Downarrow \iota} \ar[r]|{\mid}^{D \comma g} & D \ar@{=}[d] \\  B \ar@{=}[d] \ar[r]|{\mid}^{f \comma B} \ar@{}[dr]|{\Downarrow\rho} &  A \ar[d]_f \ar[r]|{\mid}^{A^\cattwo} & A \ar@{=}[d] \ar[r]|{\mid}^E  \ar@{}[dl]|{\Downarrow\kappa}& C \ar@{=}[d] \ar[r]|{\mid}^{C^\cattwo} & C \ar[d]^g \ar@{}[dl]|{\Downarrow\kappa} \ar[r]|{\mid}^{D \comma g} & D \ar@{=}[d] \ar@{}[dl]|{\Downarrow\rho}  \\ B \ar@{=}[d]  \ar[r]|{\mid}^{B^\cattwo} & B \ar@{}[d]|{\Downarrow\circ}  \ar[r]|{\mid}^{f \comma B} & A \ar@{=}[d]  \ar[r]|{\mid}^E & C \ar@{=}[d] \ar[r]|{\mid}^{D \comma g} & D \ar[r]|{\mid}^{D^\cattwo} \ar@{}[d]|{\Downarrow\circ} & D \ar@{=}[d] \\  B \ar@{=}[d]  \ar[rr]|{\mid}^{f \comma B} & & A \ar[r]|{\mid}^E  \ar@{}[dr]|{\beta} & C    \ar[rr]|{\mid}^{D\comma g} &  & D \ar@{=}[d] \\ B \ar[rrrrr]|{\mid}_F & & &  & & D }} \]
The composite of the cells in the bottom two rows in the figure on the left equals the composite of the cells in the bottom two rows in the figure on the right because both compose with the unit cells $\iota$ for $B^\cattwo$ and $D^\cattwo$ to $\beta$. Applying the conjoint identities to the right-hand figure and the definition \eqref{eq:comp-with-unit} of the cells $\circ$ in Lemma~\ref{lem:comp-with-unit}, we recover $\beta$. 

Vertically bisecting these constructions, one obtains the one-sided versions of these bijections with the cells displayed in the middle column of the statement.
\end{proof}

We frequently apply Theorem~\ref{thm:companion-conjoint} in an alternate form enabled by Proposition~\ref{prop:cartesian-pullback-cells}:

\begin{cor}\label{cor:companion-conjoint}
For any modules $E \colon A \prof C$ and $F \colon B \prof D$ and functors $f \colon A \to B$ and $g \colon C \to D$ there are bijections between cells
\[ \xymatrix{ A \ar@{=}[d] \ar[r]|{\mid}^{E} \ar@{}[dr]|{\Downarrow\alpha} & C \ar@{=}[d] & \ar@{}[d]|{\displaystyle\leftrightsquigarrow} & B \ar@{=}[d] \ar[r]|{\mid}^{f \comma B} & A \ar[r]|{\mid}^E\ar@{}[dr]|{\Downarrow\beta} & C \ar[r]|{\mid}^{D \comma g} & D \ar@{=}[d] \\ A \ar[r]|{\mid}_{F(g,f)} & C & & B \ar[rrr]|{\mid}_F & & &  D}\] 
\end{cor}

Our aim now is to prove that certain composites involving represented modules exist. Several of these proofs will take advantage of the following lemma.

\begin{lem}\label{lem:fibred-adjoint-reflection}
Consider a cell $\mu \colon E_1 \times \cdots \times E_n \To E$, where $E \colon A \prof B$ is a module from $A$ to $B$, and choose a representing map of spans
\[ \xymatrix{ {}\save[]*{E_1 \times \cdots \times E_n }  \ar[rr]^-m \ar@{->>}[dr] \restore & & E \ar@{->>}[dl]  \\ & A \times B}\]  If $m$ admits an adjoint over $A \times B$, then $\mu$ exhibits $E$ as a composite of the sequence $E_1,\ldots, E_n$.
\end{lem}
\begin{proof}
We will employ the proof strategy outlined in Observation~\ref{obs:comp-proof-strategy}. Given isofibrations $(q,p) \colon F \tfib \bar{A} \times A$ and $(s,r) \colon G \tfib B \times \bar{B}$ defined as pullbacks of finite composable sequences of modules, we use Remark~\ref{rmk:adjunctions-pullback} to pull back the adjunction over $A \times B$ along $p \times s \colon F \times G \to A \times B$. Then composing with $q \times r \colon F \times G \tfib \bar{A} \times \bar{B}$, we obtain an adjunction over $\bar{A} \times \bar{B}$.

For any module $H \colon \bar{A} \prof \bar{B}$, the 2-functor $\map_{\bar{A} \times \bar{B}}(-,H) \colon \lcat{K}_2/\bar{A} \times \bar{B} \to \qCat_2$ transforms this adjunction into an adjoint equivalence: the isofibration $H \tfib \bar{A} \times \bar{B}$ is a groupoidal object in $\lcat{K}_2/\bar{A} \times \bar{B}$ and thus the unit and counit 2-cells map to isomorphisms. Passing to isomorphism classes of objects, the equivalence
\[ \xymatrix@C=40pt{ \map_{\bar{A} \times \bar{B}}( F \times E \times G,H)  \ar[r]^-{(F \times m \times G)^*}_-{\simeq} & \map_{\bar{A} \times \bar{B}}(F \times E_1\times \cdots \times E_n \times G,H)}\] induces the required bijection between cells in $\MMod{K}$.
\end{proof}

\begin{lem}\label{lem:one-sided-rep-comp} For any module $E \colon A \prof B$ and functor $g \colon C \to A$, the composite $A \comma g \otimes_A E$ exists and is given by $E(1,g) \colon C \prof B$, the pullback of $(q,p) \colon E \tfib A \times B$ along $g \times B$. 
\end{lem}
\begin{proof}
By Lemma~\ref{lem:sliced-adjunction-on-the-right}, the functor $i \colon E \to q \comma A$ admits a left adjoint $\ell$ over $A \times B$. By Remark~\ref{rmk:adjunctions-pullback}, $\ell \dashv i$ pulls back along $g \times B$ to define an adjunction 
\[ \adjdisplay \ell' -| i' : E(1,g) -> q \comma g .\] over $C \times B$. Here we use familiar composition and cancelation results for simplicial pullbacks to form a diagram of pullback squares and rectangles
\[ \xymatrix{ q \comma g \ar@{->>}[d] \pbexcursion \ar[r]  & q \comma A \ar@{->>}[d]_{(p_1,p_0)} \pbexcursion \ar[r] & A^\cattwo \ar@{->>}[d]^{(p_1,p_0)}  \\ C \times E \pbexcursion \ar@{->>}[d]_{C \times p} \ar[r]^{g \times E} & A \times E \ar@{->>}[d]^{A \times p} \ar[r]_{A \times q} & A \times A \\ C \times B \ar[r]_{g \times B} & A \times B} \] allowing us to recognize the pullback of $q \comma A$ along $g \times B$ as the module $q \comma g$.

The simplicial pullback diagram of Lemma~\ref{lem:commas-pullback}
\[ \xymatrix@!=5pt{ & & q \comma g \pbdiamond \ar@{->>}[dl]_-{\pi_1} \ar@{->>}[dr]^-{\pi_0} \\  & A \comma g \ar@{->>}[dl]_{p_1} \ar@{->>}[dr]^{p_0} &  & E \ar@{->>}[dl]_q \ar@{->>}[dr]^p \\ C  & & A & & B}\] reveals that $q \comma g$ is the horizontal composite of the isofibrations $(p_1,p_0) \colon A \comma g \tfib C \times A$ and $(q,p)\colon E \tfib A \times B$. Applying Lemma~\ref{lem:fibred-adjoint-reflection}, the left adjoint $\ell' \colon A \comma g \pbtimes{A} E \to E(1,g)$ over $C \times B$ represents a binary cell $A \comma g \pbtimes{A} E \To E(1,g)$ that exhibits $E(1,g)$ as the composite $A \comma g \otimes E$, as claimed.
\end{proof}

\begin{obs}
Unpacking the proof of Lemma~\ref{lem:one-sided-rep-comp}, the composition cell $\mu \colon A \comma g \times E \To E(1,g)$ represented by the map $\ell'$  is defined in the following pasting diagram via  the universal property of the cartesian cell defining the pullback $E(1,g)$:
\[ \vcenter{\xymatrix{ C \ar[d]_g \ar[r]|\mid^{A \comma g} \ar@{}[dr]|{\Downarrow\rho} & A \ar@{=}[d] \ar[r]|\mid^E & B \ar@{=}[d] \\ A \ar[r]|\mid_{A^\cattwo} \ar@{=}[d] & A \ar[r]|\mid_E \ar@{}[d]|{\Downarrow\circ} & B \ar@{=}[d]  \\ A \ar[rr]|\mid_{E} & & B}} = \vcenter{\xymatrix{ C \ar@{=}[d] \ar[r]|\mid^{A \comma g}  & A \ar@{}[d]|{\Downarrow\mu} \ar[r]|\mid^E & B \ar@{=}[d] \\ C \ar[rr]|\mid_{E(1,g)} \ar[d]_g & \ar@{}[d]|{\Downarrow\rho} & B \ar@{=}[d]  \\ A \ar[rr]|\mid_{E} & & B}}
\] 
\end{obs}

Dually, for any functor $f \colon D \to B$, the composite $E \otimes_B f\comma B$ exists in $\Mod{A}{D}$ and equals $E(f,1)$, the pullback of $(q,p) \colon E \tfib A \times B$ along $A \times f$. Via Observation~\ref{obs:comp-assoc}, these results combine to prove:

\begin{cor}\label{cor:two-sided-rep-comp} For any module $E \colon A \prof B$ and pair of functors $g \colon C \to A$ and $f \colon D \to B$, the composite $A\comma g \otimes_A E \otimes_B f \comma B$ exists and is given by  $E(f,g) \colon C \prof D$, the pullback of $(q,p) \colon E \tfib A \times B$ along $g \times f \colon C \times D \to A \times B$. \qed
\end{cor}

\begin{ex}\label{ex:comp-of-rep} In particular, for any functors $A \xrightarrow{f} B \xrightarrow{g} C$, the cell $B \comma f \times_B C \comma g \To C \comma gf$ encoded by the pasting diagram
\[ \vcenter{\xymatrix@!=5pt{   && B \comma f \times_B C \comma g \pbdiamond \ar[dl] \ar[dr] \\ & B \comma f \ar[dl] \ar[dr] \ar@{}[d]|(.6){\Leftarrow} & &  C \comma g \ar[dl] \ar[dr] \ar@{}[d]|(.6){\Leftarrow}   \\  A \ar[rr]_f & & B \ar[rr]_g  & & C}} \]
 displays $C \comma gf$ as the composite $B \comma f \otimes_B C \comma g$.
\end{ex}

\begin{ex}\label{ex:two-sided-comma-as-composite} For any cospan $C \xrightarrow{g} A \xleftarrow{f} B$,  by  Corollary~\ref{cor:two-sided-rep-comp} the composite $A \comma g \otimes_A f \comma A$ is given by the module $f \comma g\colon C \prof B$. Under the interpretation of Observation~\ref{obs:cells-above-a-comma}, the cell $m \colon A \comma g \times_A f \comma A \To f \comma g$ witnessing the composite is encoded by the map of spans defined by the following pasting equality:
\[ \vcenter{\xymatrix@!=5pt{   && A \comma g \times_A f \comma A \pbdiamond \ar[dl] \ar[dr] \\ & A \comma g \ar[dl] \ar[dr] \ar@{}[d]|(.6){\Leftarrow} & &  f \comma A \ar[dl] \ar[dr] \ar@{}[d]|(.6){\Leftarrow}   \\  C \ar[rr]_g & & A  & & B \ar[ll]^f }} \qquad = \qquad \vcenter{\xymatrix@!0@=30pt{ & A \comma g \pbtimes{A} f \comma A \ar@/^2ex/[ddr] \ar@/_2ex/[ddl] \ar[d]^m \\ & f \comma g \ar[dl] \ar[dr] \ar@{}[dd]|{\Leftarrow} \\ C \ar[dr]_g & & B \ar[dl]^f \\ & A}}\]

In the context of Observation~\ref{obs:cells-above-a-comma}, if the above left pasting diagram appears as part of a 2-cell representing a multimap whose domain includes the product $A \comma g \times_A f \comma A$, then the corresponding multimap whose domain substitutes $f \comma g$ replaces this 2-cell by the canonical 2-cell displayed above right, with the map $m$ omitted.
\end{ex}

\begin{lem}\label{lem:module-as-rep-composite} Any module  $E \colon A \prof B$, encoded by an isofibration $(q,p) \colon E \tfib A \times B$,  can be regarded as a composite $E \cong q \comma A \otimes_E B \comma p$ of representable modules. More generally for any span $A \xleftarrow{g} X \xrightarrow{f} B$, not necessarily even comprised of isofibrations, there is a bijection between  cells whose horizontal domain is comprised of a list of spans, one component being $X$, and whose horizontal codomain is a module whose horizontal domain contains  one additional variable, with $g \comma A \times_X B \comma f$ in place of $X$.
\end{lem}

\begin{proof}
In the case where $(q,p) \colon E \tfib A \times B$ defines a module $E \colon A \prof B$ there are bijections
\[ \xymatrix{ A \ar@{=}[d] \ar[r]|\mid^{q \comma A} & E \ar@{}[d]|{\Downarrow} \ar[r]|\mid^{B \comma p} & B \ar@{=}[d] & \ar@{}[d]|{\displaystyle\leftrightsquigarrow} & A \ar@{=}[d] \ar[r]|\mid^{A^\cattwo} & A \ar[r]|\mid^E \ar@{}[dr]|{\Downarrow} & B \ar[r]|\mid^{B^\cattwo} & B \ar@{=}[d] \\ A \ar[rr]|\mid_{E} & & B & & A \ar[rrr]|\mid_{E} & & & B}\] 
because the simplicial pullbacks $q \comma A \pbtimes{E} B \comma p$ and $A^\cattwo \pbtimes{A} E \pbtimes{B} B^\cattwo$ are equivalent over $A \times B$. In particular, the canonical cell $\circ \colon q \comma A \times B \comma p \To E$ defined in Lemma~\ref{lem:comp-with-unit} displays $E$ as the composite of the representables at its legs. 

The point is that the proof of Proposition~\ref{prop:arrows-are-units}, which supplied the universal property used in  Lemma~\ref{lem:comp-with-unit} applies more generally. Given isofibrations $F \tfib A' \times A$ and $G \tfib B \times B'$ that are cartesian on the left and right and a module $H \colon A' \prof B'$, the proof of Proposition~\ref{prop:arrows-are-units} defines an equivalence 
\[ \map_{A' \times B'}( F \pbtimes{A} A^\cattwo \pbtimes{A} X \pbtimes{B} B^\cattwo \pbtimes{B} G, H) \xrightarrow{\simeq} \map_{A' \times B'} (F \pbtimes{A} X \pbtimes{B} G, H).\]
The domain of the left-hand hom quasi-category is equivalent to $F \times_A g \comma A \times_X B \comma f \times_B G$, completing the proof.
\end{proof}

\begin{lem}\label{lem:yoneda-embedding} For any pair of parallel functors there are natural bijections between 2-cells
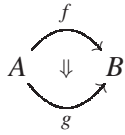
\[ \xymatrix{ A \ar@/^3ex/[r]^f \ar@/_3ex/[r]_g \ar@{}[r]|{\Downarrow} & B}\] in the homotopy 2-category and cells
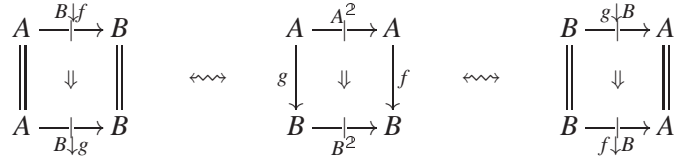
\[ \xymatrix{  A \ar@{=}[d] \ar[r]|\mid^{B \comma f} \ar@{}[dr]|{\Downarrow} & B \ar@{=}[d] & \ar@{}[d]|{\displaystyle\leftrightsquigarrow} & A \ar[r]|\mid^{A^{\cattwo}} \ar[d]_g  \ar@{}[dr]|{\Downarrow} & A \ar[d]^f &  \ar@{}[d]|{\displaystyle\leftrightsquigarrow} & B \ar@{}[dr]|{\Downarrow} \ar@{=}[d] \ar[r]|\mid^{g \comma B} & A \ar@{=}[d] \\ A \ar[r]|\mid_{B \comma g} & B & & B \ar[r]|\mid_{B^\cattwo} & B & & B \ar[r]|\mid_{f \comma B} & A}\]  
in the virtual equipment of modules.
\end{lem}
\begin{proof}
Observation~\ref{obs:cells-above-a-comma} and Proposition~\ref{prop:arrows-are-units} imply that cells in the middle square correspond to cells
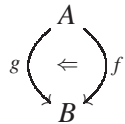
\[ \xymatrix{ A \ar@/^3ex/[d]^f \ar@/_3ex/[d]_g \ar@{}[d]|{\Leftarrow} \\ B}\]  in the homotopy 2-category. Theorem~\ref{thm:companion-conjoint} and Corollary~\ref{cor:companion-conjoint} supply the bijections to the cells displayed on the left and on the right.
\end{proof}

\begin{rmk}\label{rmk:yoneda-embedding} Lemma~\ref{lem:yoneda-embedding} and Example~\ref{ex:comp-of-rep} imply that there are two locally-fully-faithful homomorphisms $\lcat{K}_2 \hookrightarrow \MMod{K}$ and $\lcat{K}_2\coop \hookrightarrow\MMod{K}$ embedding the homotopy 2-category into the sub bicategory of $\MMod{K}$ comprised only of unary cells whose vertical boundaries are identities. The modules in the image of the first homomorphism are the covariant representables and the modules in the image of the second homomorphism are the contravariant representables. We refer to these as the \emph{covariant} and \emph{contravariant embeddings}, respectively.
\end{rmk}



\section{Pointwise Kan extensions}\label{sec:pointwise-kan}

Right and left Kan extensions can be defined internally to any 2-category --- a right Kan extension is comprised of a 1-cell and a 2-cell that define a terminal object in an appropriate category. However, in many 2-categories, as is the case for instance in the homotopy 2-category of an $\infty$-cosmos, the notion of right Kan extension defined in this way fails to be sufficiently robust.  The more useful universal property is associated to the stronger notion is of a \emph{pointwise Kan extension}. Our aim in this section is to define and study pointwise Kan extensions for functors between $\infty$-categories.

In fact, we give multiple definitions of pointwise Kan extension. One is fundamentally 2-categorical: a pointwise Kan extension is an ordinary 2-categorical Kan extension in the homotopy 2-category that is stable under pasting with comma squares. Another definition is that a 2-cell defines a pointwise right Kan extension if and only if its image under the covariant embedding into the virtual equipment of modules defines a right Kan extension there. Proposition~\ref{prop:pointwise-kan} proves that these two notions coincide.

Before turning our attention to pointwise Kan extensions, we first introduce exact squares in \S\ref{ssec:exact-squares}, a class of squares in the homotopy 2-category that include comma squares and which will be used to define initial and final functors. Pointwise Kan extensions are introduced in a variety of equivalent ways in \S\ref{ssec:pointwise-kan}. In \S\ref{ssec:pointwise-kan-closed}, we conclude with a discussion of pointwise Kan extensions in a cartesian closed $\infty$-cosmos, in which context these relate to the absolute lifting diagrams and limits and colimits studied in \S\refI{sec:limits}.

\subsection{Exact squares}\label{ssec:exact-squares}

\setcounter{thm}{0}
\begin{defn}[exact squares]\label{defn:exact-square} By Lemma~\ref{lem:yoneda-embedding} there are bijections between 2-cells in a square in the homotopy 2-category and cells in the virtual double category of modules:
\[ \xymatrix{ D \ar[r]^h \ar[d]_k  \ar@{}[dr]|{\Leftarrow\lambda} & B \ar[d]^f & \ar@{}[d]|{\displaystyle\leftrightsquigarrow} & D \ar@{=}[d] \ar[r]|\mid^{A \comma fh} \ar@{}[dr]|{\Downarrow\lambda} & A \ar@{=}[d]  \\ C \ar[r]_g & A & &  D \ar[r]|\mid_{A \comma gk} & A}\] These cells correspond bijectively to cells
\[ \xymatrix{D \ar[d]_k \ar[r]|\mid^{A \comma fh} \ar@{}[dr]|{\Downarrow} & A \ar@{=}[d] & \ar@{}[d]|{\displaystyle\leftrightsquigarrow} & D \ar[d]_k \ar[r]|\mid^{B \comma h} & B \ar@{}[d]|{\Downarrow} \ar[r]|\mid^{A \comma f} & A \ar@{=}[d] &  \ar@{}[d]|{\displaystyle\leftrightsquigarrow}  &  C \ar@{=}[d] \ar[r]|\mid^{k \comma C} & D \ar@{}[d]|{\Downarrow} \ar[r]|\mid^{B \comma h} & B \ar[d]^f    \\ C \ar[r]|\mid_{A \comma g} & A & & C \ar[rr]|\mid_{A \comma g} & & A & & C \ar[rr]|\mid_{A \comma g} & & A  }\] by Proposition~\ref{prop:cartesian-pullback-cells}, Lemma~\ref{lem:one-sided-rep-comp}, and Theorem~\ref{thm:companion-conjoint}, respectively. Applying Proposition~\ref{prop:cartesian-pullback-cells} again, these cells are in bijection with cells as displayed on the left:
\[ \vcenter{\xymatrix{  C \ar@{=}[d] \ar[r]|\mid^{k \comma C} & D \ar@{}[d]|{\Downarrow\hat{\lambda}} \ar[r]|\mid^{B \comma h} & B \ar@{=}[d] \\
C \ar[rr]|\mid_{f \comma g} & & B}} \qquad \leftrightsquigarrow \qquad
\vcenter{ \xymatrix@!0@C=40pt{ & k \comma C \pbtimes{D} B \comma h \pbdiamond \ar[dr] \ar[dl] \\  k \comma C \ar[dr] \ar[dd] \ar@{}[dddr]|(.3){\Leftarrow} & & B \comma h \ar[dl] \ar[dd] \ar@{}[dddl]|(.3){\Leftarrow} \\ & D \ar[dr]^h \ar[dl]_k \ar@{}[dd]|(.4){\lambda}|(.6){\Leftarrow} \\ C \ar[dr]_g & & B \ar[dl]^f \\ & A}}
\] 
Under the isomorphism described by Observation~\ref{obs:cells-above-a-comma}, we can represent the 2-cell $\hat{\lambda}$ as the pasting diagram displayed above right in the homotopy 2-category $\lcat{K}_2$. 
If $\hat{\lambda}$ displays $f \comma g$ as the composite $k \comma C \otimes_D B \comma h$ in $\MMod{K}$, then we say that the square $\lambda \colon fh \To gk$ in $\lcat{K}_2$ is \emph{exact}.
\end{defn}

\begin{lem}[composites of exact squares]\label{lem:exact-squares-comp} Exact squares can be composed both ``horizontally'' and ``vertically'': given a diagram in the homotopy 2-category
\[ \xymatrix{ & H \ar[d]_t \ar[r]^s \ar@{}[dr]|{\Leftarrow\tau}& G \ar[d]^r \\ F \ar[d]_q \ar[r]^\ell \ar@{}[dr]|{\Leftarrow\mu}& D \ar[d]_k \ar[r]^h \ar@{}[dr]|{\Leftarrow\lambda}  & B \ar[d]^f \\ E \ar[r]_p & C \ar[r]_g & A}\] if $\lambda \colon fh \To gk$, $\mu \colon k\ell \To pq$, and $\tau \colon rs \To ht$ are exact, then so are their composites $fh\ell \xTo{\lambda \ell} gk\ell \xTo{g\mu} gpq$ and $frs \xTo{f\tau} fht \xTo{\lambda t} gkt$.
\end{lem}
\begin{proof}
We prove the result for horizontal composition; a similar argument shows that exact squares can also be composed vertically. The cell induced by the composite $\lambda\mu \colon fh\ell \To gpq$ factors as 
\[ \vcenter{\xymatrix{ E \ar@{=}[d] \ar[r]|\mid^{q\comma E} & F \ar[r]|\mid^{D \comma \ell} \ar@{}[dr]|{\Downarrow{\widehat{\lambda\mu}}} & D \ar[r]|\mid^{B \comma h} & B \ar@{=}[d] \\ E \ar[rrr]|\mid_{f \comma gp} & & & B}} = \vcenter{\xymatrix{ E \ar@{=}[d] \ar[r]|\mid^{q\comma E} & \ar@{}[d]|{\Downarrow\hat\mu}  F \ar[r]|\mid^{D \comma \ell} & D \ar@{=}[d]  \ar[r]|\mid^{B \comma h} & B \ar@{=}[d] \\  E \ar@{=}[d] \ar[rr]|\mid^{k \comma p} & & D \ar@{}[dl]|{\Downarrow{\tilde{\lambda}}} \ar[r]|\mid^{B \comma h} & B \ar@{=}[d] \\ E \ar[rrr]|\mid_{f \comma gp} & & & B}}\]
where $\tilde{\lambda}$ is the cell defined by by the pasting equality
\[\vcenter{\xymatrix{ E \ar@{=}[d] \ar[r]|\mid^{C \comma p} & \ar@{}[d]|{\Downarrow\circ}  C \ar[r]|\mid^{k \comma C} & D \ar@{=}[d]  \ar[r]|\mid^{B \comma h} & B \ar@{=}[d] \\  E \ar@{=}[d] \ar[rr]|\mid^{k \comma p} & & D \ar@{}[dl]|{\Downarrow{\tilde{\lambda}}} \ar[r]|\mid^{B \comma h} & B \ar@{=}[d] \\ E \ar[rrr]|\mid_{f \comma gp} & & & B}} \defeq \vcenter{\xymatrix{ E \ar@{=}[d] \ar[r]|\mid^{C \comma p} & \ar@{=}[d]  C \ar[r]|\mid^{k \comma C} & D \ar@{}[d]|{\Downarrow\hat\lambda}  \ar[r]|\mid^{B \comma h} & B \ar@{=}[d] \\  E \ar@{=}[d] \ar[r]|\mid^{C \comma p} & C \ar@{}[dr]|{\Downarrow{\circ}} \ar[rr]|\mid^{f \comma g} &  & B \ar@{=}[d] \\ E \ar[rrr]|\mid_{f \comma gp} & & & B}} \] 
via the universal property of the composite $\circ \colon C \comma p \times k \comma C \To k \comma p$ of Lemma~\ref{lem:one-sided-rep-comp}. By exactness of $\lambda$ and Lemma~\ref{lem:one-sided-rep-comp}, the cell $\hat\lambda$ and both cells labelled $\circ$ are composites; thus Observation~\ref{obs:comp-cancel} implies that $\tilde\lambda\colon k \comma p \times B \comma h \To f \comma gp$ is also a composite. Now Observation~\ref{obs:comp-assoc} and exactness of $\mu$ implies that $\widehat{\lambda\mu} \colon q \comma E \times D \comma \ell \times B \comma h \To f \comma gp$ is also a composite, proving that the composite 2-cell is exact.
\end{proof}

\begin{lem}\label{lem:comma-exact} Any comma square is exact.
\end{lem}
\begin{proof}
Consider
\[ \xymatrix{ f \comma g \ar[d]_q \ar[r]^-p \ar@{}[dr]|{\Leftarrow\lambda} & B \ar[d]^f \\ C \ar[r]_g & A}\] 
Applying Lemma~\ref{lem:module-as-rep-composite} to $f \comma g \colon C \prof B$, the canonical cell $q \comma C \times B \comma p \To f \comma g$ is a composite. Example~\ref{ex:two-sided-comma-as-composite} explains that this is $\hat{\lambda}$.
\end{proof}

\begin{lem}\label{lem:pullback-exact} Consider a pullback square
\[ \xymatrix{ P \pbexcursion \ar[d]_{\pi_1} \ar[r]^{\pi_0} & B \ar[d]^f \\ C \ar[r]_g & A}\] If $g$ is a cartesian fibration or if $f$ is a cocartesian fibration, then the pullback square is exact.
\end{lem}
\begin{proof}
The two cases are dual. Suppose that $f$ is a cocartesian fibration and consider the induced map
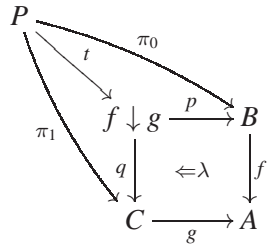
\[ \xymatrix{ P \ar[dr]^t\ar@/_1ex/[ddr]_{\pi_1} \ar@/^1ex/[drr]^{\pi_0}  \\ & f \comma g \ar[d]_q \ar[r]^p \ar@{}[dr]|{\Leftarrow\lambda} & B \ar[d]^f \\ & C \ar[r]_g & A}\]   Observe that $t \colon P \to f \comma g$ is the pullback of the map $i \colon B \to f \comma A$ along $g \colon C \to A$. 
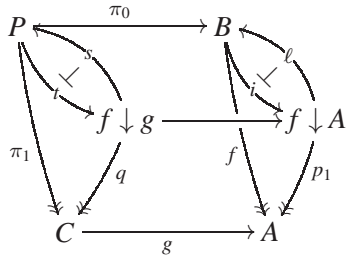
  \begin{equation*}
    \begin{xy}
      0;<5pc,0pc>:
      *{\xybox{
        \POS(0,0)*+{C}="A",
        \POS(1.3,0)*+{A}="B",
        \POS(-0.3,1.3)*+{P}="F",
        \POS(0.4,0.7)*+{f \comma g}="F'",
        \POS(1,1.3)*+{B}="E",
        \POS(1.6,0.7)*+{f \comma A}="E'",	
        \ar"A";"B"_g
        \ar@{->>}@/_0.2pc/"F";"A"_(0.6){\pi_1}
        \ar@{->>}@/^0.2pc/"F'";"A"^(0.4){q}
        \ar@{->>}@/_0.2pc/"E";"B"_(0.6){f}|!{"F'";"E'"}\hole
        \ar@{->>}@/^0.2pc/"E'";"B"^(0.4){p_1}
        \ar"F";"E"^{\pi_0}
        \ar"F'";"E'"
        \ar@{<-}@/^0.8pc/"E";"E'"|{\ell}_{}="h'"
        \ar@/_0.8pc/"E";"E'"|{i}^{}="h"
        \ar@{<-}@/^0.8pc/"F";"F'"|{s}_{}="k'"
        \ar@/_0.8pc/"F";"F'"|{t}^{}="k"
        \ar@{}"h";"h'"|{\rotatebox{-45}{$\perp$}}
        \ar@{}"k";"k'"|{\rotatebox{-45}{$\perp$}}
      }}
    \end{xy}
  \end{equation*}
By Theorem~\refIV{thm:cart.fib.chars}, $i$ has a left adjoint over $A$. By Remark~\ref{rmk:adjunctions-pullback}, this pulls back to define a left adjoint $s \dashv t$ over $C$. 

We wish to show that the cell $\hat{\id} \colon \pi_1 \comma C \times B \comma \pi_0 \To f \comma g$ is a composite. By 
Lemma~\ref{lem:one-sided-rep-comp}, the canonical cell induces a bijection between cells with $\pi_1 \comma C \times B \comma \pi_0$ among their horizontal domain and cells with $\pi_1 \comma C \times (f \comma g) \comma t \times B \comma p$ among their domains. By Proposition~\refI{prop:adjointequiv}, the adjoint $s \dashv t$ implies that the modules $(f \comma g) \comma t$ and $s \comma P$ are equivalendt, so these cells are in bijection with cells that have 
$\pi_1 \comma C \times s \comma P \times B \comma p$ among their horizontal domains. Applying Lemma~\ref{lem:one-sided-rep-comp} again, the canonical cell induces a bijection between these cells and those with $q \comma C \times B \comma p$ among their domains.

The equation $\id = \lambda t \colon f \pi_0 = f p t \To gq t = g \pi_1$ can be interpreted as saying that this 2-cell is the transpose along $s \dashv t$ of the 2-cell $\lambda \colon fp \To g  q = g \pi_1 s$. This relation tells us that the cells
\[ \hat{\id} \colon  \pi_1 \comma C \times B \comma pt \To f \comma g \qquad \mathrm{and} \qquad \hat\lambda \colon \pi_1 s \comma C \times B \comma p \To f \comma g \] correspond under the bijection just described. By Lemma~\ref{lem:comma-exact} $\hat\lambda$ is a composite; thus $\hat\id$ is as well.
\end{proof}

We conclude this section with a pair of technical lemmas that will be used to prove Proposition~\ref{prop:ran-abs-lifting}.

\begin{lem}\label{lem:product-exact-square} For any pair of functors $k \colon A \to B$ and $h \colon C \to D$, the square
\[ \xymatrix{ A \times C \ar[r]^{k \times C} \ar[d]_{A \times h} & B \times C \ar[d]^{B \times h} \\ A \times D \ar[r]_{k \times D} & B \times D}\] is exact.
\end{lem}
\begin{proof}
Following the prescription of Definition~\ref{defn:exact-square}, the identity 2-cell $\id_{k \times h}$ transposes to define a cell $\widehat{\id_{k \times h}}$ in $\MMod{K}$ whose horizontal domain is the span computed by the simplicial pullback
\[ \xymatrix@=12pt{ & & \save[] { (A^\cattwo \pbtimes{A} B \comma k) \times (h \comma D \pbtimes{C} C^\cattwo)} \ar@{->>}[dl] \ar@{->>}[dr]\restore \pbdiamond \\ & \save[]{ A^\cattwo \times h \comma D}  \ar@{->>}[dl]_{p_1 \times p_1} \ar@{->>}[dr]^{p_0 \times p_0 } \restore & & \save[]{B \comma k \times C^\cattwo} \ar@{->>}[dl]_{ p_1 \times p_1} \ar@{->>}[dr]^{p_0 \times p_0}\restore  \\ A \times D & & A \times C & & B \times C}\] 
The horizontal codomain is isomorphic to the span
\[ \xymatrix{ & B \comma k \times h \comma D \ar@{->>}[dl]_{p_1 \times p_1} \ar@{->>}[dr]^{p_0 \times p_0} \\ A \times D & & B \times C}\]
as this is $(B \times h) \comma (k \times D)$. 

By inspection, the cell $\widehat{\id_{k \times h}} \colon (A^\cattwo \pbtimes{A} B \comma k) \times (h \comma D \pbtimes{C} C^\cattwo) \To B \comma k \times h \comma D$ that we seek to show defines a composite in $\MMod{K}$ is represented in the slice 2-category over $A \times D \times B \times C$ by the product of the functors considered in Lemma~\ref{lem:one-sided-rep-comp} and its dual:
\[ \xymatrix@=1.5em{ A^\cattwo \times_A B \comma k \ar@{->>}[dr] \ar[rr]^-{\ell'} & & B \comma k \ar@{->>}[dl] & & h \comma D \times_C C^\cattwo \ar[rr]^-{r'} \ar@{->>}[dr] & & h \comma D \ar@{->>}[dl] \\ & A \times B & & & & D \times C}\] The former admits a fibered right adjoint while the latter admits a fibered left adjoint. The product of these adjoints defines a fibered functor $B \comma k \times h \comma D \to (A^\cattwo \pbtimes{A} B \comma k) \times (h \comma D \pbtimes{C} C^\cattwo)$ whose composites with $\ell' \times r' \colon (A^\cattwo \pbtimes{A} B \comma k) \times (h \comma D \pbtimes{C} C^\cattwo) \to B \comma k \times h \comma D$ are connected to the identity functors via a zig-zag of fibered 2-cells. As in the proof of Lemma~\ref{lem:fibred-adjoint-reflection}, these fibred cells are inverted upon mapping into a groupoidal object, exhibiting $\widehat{\id_{k \times h}}$ as a composite, as required. 
\end{proof}

\begin{lem}\label{lem:product-comma-exact} If the left-hand square is a comma square in $\lcat{K}_2$ and $K$ is any object, then the right-hand square is exact. 
\[ \xymatrix{ D \ar[r]^h \ar[d]_k  \ar@{}[dr]|{\Leftarrow\lambda} & B \ar[d]^f &  & D \times K \ar[r]^{h \times K} \ar[d]_{k \times K} \ar@{}[dr]|{\Leftarrow \lambda \times K} & B \times K \ar[d]^{f \times K} \\ C \ar[r]_g & A & &  C \times K \ar[r]_{g \times K} & A \times K}\]  
\end{lem}
\begin{proof}
The proof that comma squares are exact is derived from Lemma~\ref{lem:fibred-adjoint-reflection}: the cell $\hat\lambda \colon k \comma C \times B \comma h \To f \comma g$ is represented by a functor $\ell \colon k \comma C \times B \comma h \to f \comma g$ over $C \times B$ that admits a fibred adjoint. Similarly, the cell $\widehat{\lambda \times K} \colon ((k \times K) \comma (C \times K)) \times_{D \times K} ((B \times K) \comma (h \times K)) \To (f \times K) \comma (g \times K)$ is represented by a fibred functor 
\[ \xymatrix{ k \comma C \pbtimes{D}  B \comma h \times (K^\cattwo \pbtimes{K} K^\cattwo) \ar[rr]^-{\ell \times m} \ar@{->>}[dr] & & f \comma g \times K^\cattwo \ar@{->>}[dl] \\ & C \times B \times K \times K}\] admitting a fibred adjoint: left and right fibred adjoints to the ``composition functor'' $m \colon K^\cattwo \times_K K^\cattwo \to K^\cattwo$ are constructed in \refI{ex:comp.ident.adj}. Applying Lemma~\ref{lem:fibred-adjoint-reflection}, we conclude that $\widehat{\lambda \times K}$ is a composite, so $\lambda \times K$ is exact.
\end{proof}

\subsection{Pointwise Kan Extensions}\label{ssec:pointwise-kan}

In this section, we give two definitions of pointwise right Kan extension in the homotopy 2-category of an $\infty$-cosmos and prove that they are equivalent.

\begin{defn}[right extension of modules]\label{defn:right-extension-mod}
In the virtual equipment $\MMod{K}$ of modules, a \emph{right extension} of a module $F \colon A \prof C$ along a module $K \colon A \prof B$ is given by a module $R \colon B \prof C$ together with a cell $\mu \colon K \times R \To F$  so that for any composable sequence of modules $E_1,\ldots, E_n$ from $B$ to $C$, composition with $\mu$ defines a bijection
 \[\vcenter{ \xymatrix@=1em{ A \ar[rr]|\mid^K \ar[dd]|{\rotatebox{90}{$\labelstyle\mid$}}_F &  & B \ar[d]|(.4){\rotatebox{90}{$\labelstyle\mid$}}^{E_1} \\ & &   A_1\ar@/^/@{..>}[dl]    \\C & A_{n-1} \ar[l]|(.6)\mid^{E_n} \ar@{}[uu]|{\Leftarrow\chi}}} \qquad = \qquad \vcenter{ \xymatrix@=1em{ A \ar[rr]|\mid^K  \ar@{}[ddrr]|(.3){\Leftarrow\nu}|(.7){\Leftarrow\exists!} \ar[dd]|{\rotatebox{90}{$\labelstyle\mid$}}_F &  & B \ar[ddll]|{\rotatebox{45}{$\labelstyle\mid$}}^R  \ar[d]|(.4){\rotatebox{90}{$\labelstyle\mid$}}^{E_1} \\ & &   A_1\ar@/^/@{..>}[dl]    \\C & A_{n-1} \ar[l]|(.6)\mid^{E_n} &  }} \]
\end{defn}

In the case where the modules $K \colon A \prof B$, $F \colon A \prof C$, and $R \colon B \prof C$ are all covariant representables, the Yoneda lemma, in the form of Lemma~\ref{lem:yoneda-embedding}, implies that the binary cell arises from a 2-cell in the homotopy 2-category. The following lemma shows that Definition~\ref{defn:right-extension-mod} implies that this 2-cell is a right extension in $\lcat{K}_2$, in the usual sense.

\begin{lem}\label{lem:mod-extensions-in-2-cat}
If $\mu \colon B \comma k \times C \comma r \To C \comma f$ displays $C \comma r \colon B \prof C$ as a right extension of $C \comma f \colon A \prof C$ along $B \comma k \colon A \prof B$ in $\MMod{K}$, then $\mu \colon rk \To f$ displays $r$ as the right extension of $f$ along $k$ in $\lcat{K}_2$.
\end{lem}
\begin{proof}
By Example~\ref{ex:comp-of-rep}, the binary cell $\mu$ is represented by a unary cell $C \comma rk \To C \comma f$ in $\MMod{K}$. The covariant embedding $\lcat{K}_2 \hookrightarrow \MMod{K}$ described in Remark~\ref{rmk:yoneda-embedding} is locally fully faithful, so this cell comes from a unique 2-cell $\mu \colon rk \To f$ in $\lcat{K}_2$. Local fully-faithfulness implies immediately that for any $e \colon B \to C$ pasting with $\mu$ defines a bijection
\[ \hom(B,C)(e,r) \xrightarrow{ \mu\cdot(-\circ k)} \hom(A,C)(ek,f),\] derived from the similar bijection between cells between the corresponding covariant represented modules. 
\end{proof}

\begin{defn}[stability of extensions under pasting]
In any 2-category, a right extension diagram 
\[ \xymatrix{ A \ar[rr]^k \ar[dr]_f & \ar@{}[d]|(.4){\Leftarrow\mu} &  B \ar[dl]^r \\ & C}\]
is said to be \emph{stable under pasting with a square}
\[ \xymatrix{ D \ar[d]_g \ar[r]^h \ar@{}[dr]|{\Leftarrow\lambda} & E \ar[d]^b \\ A \ar[r]_k & B}\]
if the pasted diagram
\[ \xymatrix{ D \ar[rr]^h \ar[d]_g \ar@{}[drr]|{\Leftarrow\lambda} & & E \ar[d]^b\\
A \ar[rr]^k \ar[dr]_f & \ar@{}[d]|(.4){\Leftarrow\mu} &  B \ar[dl]^r \\ & C}\]  displays $br$ as a right extension of $fg$ along $h$.
\end{defn}

\begin{prop}\label{prop:pointwise-kan} For a diagram
\begin{equation}\label{eq:pointwise-right-extension} \xymatrix{ A \ar[rr]^k \ar[dr]_f & \ar@{}[d]|(.4){\Leftarrow\mu} &  B \ar[dl]^r \\ & C}\end{equation}
in the homotopy 2-category  of an $\infty$-cosmos $\lcat{K}$ the following are equivalent. 
\begin{enumerate}[label=(\roman*)]
\item\label{itm:pointwise-exact-stable} $\mu \colon rk \To f$ defines a right extension in $\lcat{K}_2$ that is stable under pasting with exact squares.
\item\label{itm:pointwise-comma-stable} $\mu \colon rk \To f$ defines a right extension in $\lcat{K}_2$ that is stable under pasting with comma squares.
\item\label{itm:pointwise-in-mod} The image $\mu \colon B \comma k \times C \comma r \To C \comma f$ of $\mu$ under the covariant embedding $\lcat{K}_2\hookrightarrow\MMod{K}$ defines a right extension in $\MMod{K}$.
\item\label{itm:pointwise-exact-stable-in-mod} The image of the pasted composite of $\mu$ with any exact square under the covariant embedding $\lcat{K}_2\hookrightarrow\MMod{K}$ defines a right extension in $\MMod{K}$.
\end{enumerate}
\end{prop}

If any of these equivalent conditions hold, we say that \eqref{eq:pointwise-right-extension} defines a \emph{pointwise right Kan extension} in the homotopy 2-category $\lcat{K}_2$. 

\begin{proof}
Lemma~\ref{lem:comma-exact} proves that \ref{itm:pointwise-exact-stable}$\To$\ref{itm:pointwise-comma-stable}.

To show \ref{itm:pointwise-comma-stable}$\Rightarrow$\ref{itm:pointwise-in-mod}, suppose \eqref{eq:pointwise-right-extension} defines a right extension in $\lcat{K}_2$ that is stable under pasting with comma squares and consider a cone over the cell $\mu \colon B \comma k \times C \comma r \To C \comma f$ with summit given by an isofibration $(q,p)\colon E \tfib B \times C$. By our hypothesis \ref{itm:pointwise-comma-stable}, the pasted composite
\[ \xymatrix{ q \comma k \ar[rr]^-t \ar[d]_s \ar@{}[drr]|{\Leftarrow\lambda} & & E \ar[d]^q\\
A \ar[rr]^k \ar[dr]_f & \ar@{}[d]|(.4){\Leftarrow\mu} &  B \ar[dl]^r \\ & C}\] defines a right extension in $\lcat{K}_2$.
A 2-cell $B\comma k \pbtimes{B} E \To C \comma f$ is, by Lemma~\ref{lem:module-as-rep-composite}, the same as a 2-cell
\[ B \comma k \pbtimes{B} q \comma B \pbtimes{E} C \comma p \To C \comma f.\] Using Corollary~\ref{cor:companion-conjoint} and Lemma~\ref{lem:one-sided-rep-comp} this is the same as $q \comma k \To p \comma f$, which, by Observation~\ref{obs:cells-above-a-comma} is the same as a 2-cell
\[ \xymatrix{ q \comma k \ar[d]_s \ar[r]^-t \ar@{}[dr]|{\Leftarrow} & E \ar[d]^p \\ A \ar[r]_f & C}\] Using the hypothesis that $rq$ is the right extension of $fs$ along  $t$ in the homotopy 2-category this is the same as a 2-cell $p \To rq$, or by Lemma~\ref{lem:yoneda-embedding}, as a cell $C \comma p \To C \comma rq$. By Corollary~\ref{cor:companion-conjoint}, this is the same as a cell \[q \comma B \pbtimes{E} C \comma p \To C \comma r,\] which by Lemma~\ref{lem:module-as-rep-composite} produces the desired factorization $E \To C \comma r$.

To show \ref{itm:pointwise-in-mod}$\Rightarrow$\ref{itm:pointwise-exact-stable-in-mod} consider a diagram 
\[ \xymatrix{ D \ar[rr]^h \ar[d]_g \ar@{}[drr]|{\Leftarrow\lambda} & & E \ar[d]^b\\
A \ar[rr]^k \ar[dr]_f & \ar@{}[d]|(.4){\Leftarrow\mu} &  B \ar[dl]^r \\ & C}\] in which $\lambda$ is exact and $\mu$ displays $C \comma r$ as the right extension of $C \comma f$ along $B \comma k$. We will show that the pasted composite again defines a right extension diagram at the level of modules. 

To that end, observe that a cell \[ E \comma h \times E_1 \times \cdots \times E_n \To C\comma fg\] corresponds to a cell \[ g \comma A \times E \comma h \times E_1 \times \cdots \times E_n \To C \comma f\] by Corollary~\ref{cor:companion-conjoint}. By exactness of $\lambda$, this corresponds to a cell \[ b \comma k \times E_1 \times \cdots \times E_n \To C \comma f,\] or equivalently, upon restricting along the composition map $B \comma k \times b \comma B \To b \comma k$ of Lemma~\ref{lem:one-sided-rep-comp} to a cell \[ B \comma k \times b \comma B  \times E_1 \times \cdots \times E_n \To C \comma f.\]
 As $C \comma r$ is the right extension of $C \comma f$ along $B \comma k$, this corresponds to a cell \[  b \comma B  \times E_1 \times \cdots \times E_n  \To C\comma r,\] which transposes, via Corollary~\ref{cor:companion-conjoint}, to the the desired factorization \[ E_1 \times \cdots \times E_n \To  C \comma rb.\]

To see that this bijection is implemented by composing with $\mu\lambda \colon E \comma h \times C \comma rb \To C \comma fg$, it suffices, by the Yoneda lemma, to start with the identity cell $C \comma rb \To C \comma rb$ and trace backwards through each step in this bijection to see that the result is $\mu\lambda \colon E \comma h \times C \comma rb \To C \comma fg$. Employing Observation~\ref{obs:cells-above-a-comma} to represent each cell in the virtual double category as a pasting diagram in the homotopy 2-category, this is straightforward.

Finally, Lemma~\ref{lem:mod-extensions-in-2-cat}, together with the trivial observation that the identity 2-cell defines an exact square
\[\xymatrix{ A \ar[r]^k \ar@{=}[d] & B \ar@{=}[d] \\ A \ar[r]_k & B}\] proves that \ref{itm:pointwise-exact-stable-in-mod}$\To$\ref{itm:pointwise-exact-stable}.
\end{proof}

\begin{obs} Lemma~\ref{lem:exact-squares-comp} implies that the pasted composite of a pointwise Kan extension with an exact square again defines a pointwise Kan extension.
\end{obs}

\begin{defn}[fully faithful]\label{defn:fully-faithful} A functor $k \colon A \to B$ is \emph{fully faithful} if and only if the square
\[ \xymatrix{ A \ar@{=}[r] \ar@{=}[d] & A \ar[d]^k \\ A \ar[r]_k & B}\] is exact, i.e., if and only if $A^\cattwo \To k \comma k$ is a composite. Observation~\ref{obs:unary-comp} reminds us that a cell between parallel modules is a composite if and only if it is an isomorphism in the 1-category of modules between a pair of fixed objects, so this is the case if and only if the canonical cell $A^\cattwo \To k \comma k$ defines an equivalence of modules from $A$ to $A$. 
\end{defn}

\begin{lem} If 
\[ \xymatrix{ A \ar[rr]^k \ar[dr]_f & \ar@{}[d]|(.4){\Leftarrow\mu} &  B \ar[dl]^r \\ & C}\] is a pointwise right extension and $k$ is fully faithful, then $\mu$ is an isomorphism.
\end{lem}
\begin{proof}
Pasting $\mu$ with the exact square $\id_k$ yields, by Proposition~\ref{prop:pointwise-kan}.\ref{itm:pointwise-exact-stable}, a pointwise right extension diagram
\[ \xymatrix{ A \ar[rr]^{\id_A} \ar[dr]_f & \ar@{}[d]|(.4){\Leftarrow\mu} &  A \ar[dl]^{rk} \\ & C}\] 
Proposition~\ref{prop:arrows-are-units} asserts that any functor $f \colon A \to C$ defines a pointwise extension of itself along $\id_A \colon A \to A$ in the sense of \ref{prop:pointwise-kan}\ref{itm:pointwise-in-mod}. The unique factorization in $\lcat{K}_2$ of the pointwise right extension $\id_f$ through $rk$ defines an inverse isomorphism to $\mu$.
\end{proof}

\begin{lem} A right adjoint $u \colon A \to B$ is fully faithful if and only if the counit $\epsilon \colon fu \To 1_A$ of the adjunction is an isomorphism.
\end{lem}
\begin{proof} If $f \dashv u$ with counit $\epsilon \colon fu \To 1_A$, then Proposition~\refI{prop:adjointequiv} demonstrates that composing with $\epsilon$ defines an isomorphism of modules $B \comma u \To f \comma A$, as recalled in Example~\ref{ex:adjoint-equivalence-mods}. By Observation~\ref{obs:unary-comp}, this says that the bottom square is exact.
\[ \xymatrix{ A \ar@{=}[d] \ar@{=}[r] & A \ar[d]^u \\ A \ar[r]^u \ar@{}[dr]|{\Leftarrow\epsilon} \ar@{=}[d] & B \ar[d]^f \\ A \ar@{=}[r] & A}\]
If $u$ is fully faithful, then by Lemma~\ref{lem:exact-squares-comp}, then so is the composite rectangle. This says that the contravariant embedding of $\epsilon$ into $\MMod{K}$ defines an isomorphism $A^\cattwo \To fu \comma A$ of modules from $A$ to $A$, which by fully faithfulness of the Yoneda embedding implies that $\epsilon$ is an isomorphism.

Conversely, Example~\ref{ex:two-sided-comma-as-composite} tells us that $B \comma u \times_B u \comma B \To u \comma u$ is a composite. Substituting the equivalent module $f \comma A$, Example~\ref{ex:comp-of-rep} provides another composite $f \comma A \times_B u \comma B \To fu \comma A$. Factoring one composite through the other, we obtain an equivalence $fu \comma A \To u \comma u$. If $\epsilon$ is an isomorphism, we have a composite equivalence $A^\cattwo \To fu \comma A \To u \comma u$, which proves that $u$ is fully faithful.
\end{proof}

\subsection{Pointwise Kan extensions in a cartesian closed \texorpdfstring{$\infty$}{infinity}-cosmos}\label{ssec:pointwise-kan-closed}

In this section we work in the homotopy 2-category of a cartesian closed $\infty$-cosmos $\lcat{K}$.

\begin{prop}\label{prop:ran-abs-lifting} Suppose \[ \xymatrix{  A \times K \ar[rr]^{ k \times K} \ar[dr]_f & \ar@{}[d]|(.3){\nu}|{\Leftarrow} &   B \times K \ar[dl]^r \\ & E}\] is a pointwise right Kan extension in a cartesian closed $\infty$-cosmos $\lcat{K}$. Then the transpose 
\begin{equation}\label{eq:pointwise-kan-abs-lifting} \xymatrix{ \ar@{}[dr]|(.7){\Downarrow\nu} & E^B \ar[d]^{E^k} \\ K \ar[r]_-{f} \ar[ur]^r & E^A}\end{equation} defines an absolute right lifting diagram in $\lcat{K}_2$ and moreover this absolute lifting diagram is stable under pasting with $E^\lambda$ for any comma square $\lambda$. 
\end{prop}
\begin{proof}
Given a cone as displayed on the left, we construct the required factorization as displayed on the right
\[ \xymatrix{ X \ar[r]^q \ar[d]_{p} \ar@{}[dr]|{\Downarrow\chi} & E^B \ar[d]^{E^k} \ar@{}[dr]|{\displaystyle\ =} & X \ar[d]_{p} \ar[r]^q \ar@{}[dr]|(.3){\Downarrow\phi}|(.7){\Downarrow\nu} & E^B \ar[d]^{E^k} \\ K \ar[r]_f & E^A & K \ar[ur]|r \ar[r]_f & E^A}\] by solving this problem in transposed form:
\[ \xymatrix@C=10pt{  A \times X \ar[d]_{A \times p}  \ar[rr]^{k \times X} \ar@{}[ddrr]|(.4){\Leftarrow\chi} & &  B \times X \ar@/^/[ddl]^q & &  A \times X \ar[d]_{A \times p}  \ar[rr]^{k \times X} & & B \times X \ar[d]_{B \times p} \ar@/^12ex/[ddl]^q\\  A  \times K \ar[dr]_f &  & & = &  A \times K  \ar[dr]_f \ar[rr]^{ k \times K} & \ar@{}[d]|(.3){\nu}|\Leftarrow &  B \times K \ar[dl]^r \ar@{}[d]^{\exists !\Leftarrow\phi} \\ & E  & & & & E  & }\] 
Lemma~\ref{lem:product-exact-square} tells us that the top right square is exact. Thus, $r ( B \times p)$ is a pointwise right Kan extension of $f (A \times p)$ along $k \times X$, inducing the desired 2-cell $\phi$.

Now the pasted composite of $\nu$ with an exponentiated comma square, as displayed below-left, transposes to the diagram displayed below right.
\[ \vcenter{ \xymatrix{ \ar@{}[dr]|(.7){\Downarrow\nu} & E^B \ar[d]^{E^k} \ar[r]^{E^h} \ar@{}[dr]|{~\Downarrow E^\lambda} & E^D \ar[d]^{E^p} \\ K \ar[r]_-{f} \ar[ur]^r & E^A \ar[r]_{E^q} & E^C}} \qquad \leftrightsquigarrow \qquad
\vcenter{ \xymatrix@C=10pt{  C \times K \ar[rr]^{p \times K} \ar[d]_{q \times K} \ar@{}[drr]|{\Leftarrow \lambda \times K} & & D \times K \ar[d]^{h \times K} \\ A \times K  \ar[dr]_f \ar[rr]^{ k \times K} & \ar@{}[d]|(.3){\nu}|\Leftarrow &  B \times K \ar[dl]^r \\ & E}} \] By Lemma~\ref{lem:product-comma-exact}, $\lambda \times K$ is exact, so the right-hand pasting diagram defines a pointwise Kan extension. The universal property of this right Kan extension diagram in $\lcat{K}_2$ transposes across $-\times K \dashv (-)^K$ to demonstrate that the left-hand side defines an absolute right lifting diagram.
\end{proof}

Recall Definition~\refI{defn:limit}: in a cartesian closed $\infty$-cosmos, the \emph{limit} of a diagram $f \colon A \to E$ is a point $\ell \colon 1 \to E$ equipped with an  absolute right lifting diagram
\begin{equation}\label{eq:limit-abs-lifting}  \xymatrix{ \ar@{}[dr]|(.7){\Downarrow\nu} & E \ar[d]^{E^{!}} \\ 1 \ar[r]_-{f} \ar[ur]^\ell & E^A}\end{equation} Here the 2-cell $\nu$ encodes the data of the limit cone. 

\begin{prop}\label{prop:limits-as-pointwise-kan} In a cartesian closed $\infty$-cosmos $\lcat{K}$, any limit \eqref{eq:limit-abs-lifting} defines a pointwise right Kan extension 
\[ \xymatrix{ A \ar[dr]_f \ar[rr]^{!} & \ar@{}[d]|(.3){\nu}|{\Leftarrow} & 1 \ar[dl]^\ell \\ & E}\] Conversely, any pointwise right Kan extension of this form transposes to define a limit \eqref{eq:limit-abs-lifting} in $E$.
\end{prop}
\begin{proof}
 Comma squares over the terminal object have the form 
\[ \xymatrix{ A \times X \ar[d]_{\pi_1} \ar[r]^-{\pi_0} & X \ar[d]^{!} \\ A \ar[r]_{!} & 1}\] for some $X$. We can show that the pasted composite of $\nu \colon \ell ! \To f$ with this comma square defines a right extension diagram in $\lcat{K}_2$ by proving that the transposed diagram
\[  \xymatrix{ \ar@{}[dr]|(.7){\Downarrow\nu} & E \ar[d]^{E^{!}} \ar[r]^{E^{!}} & E^X \ar[d]^{E^{\pi_0}} \\ 1 \ar[r]_-{f} \ar[ur]^\ell & E^A \ar[r]_-{E^{\pi_1}} & E^{A \times X}}\] 
defines a right lifting diagram. In fact this diagram is an absolute right lifting diagram. This follows easily from the universal property of the absolute lifting diagram \eqref{eq:limit-abs-lifting} by transposing across the 2-adjunction $X \times - \dashv (-)^X$.

The converse is a special case of Proposition~\ref{prop:ran-abs-lifting}.
\end{proof}

\begin{defn}[initial/final functor]\label{defn:initial-final} A functor $k \colon A \to B$ is \emph{final} if and only if the left-hand square is exact
\[ \xymatrix{ A \ar[d]_{!} \ar[r]^k & B \ar[d]^{!} & & A \ar[d]_k \ar[r]^{!} & 1 \ar@{=}[d]  \\ 1 \ar@{=}[r] & 1 & & B \ar[r]_{!} & 1}\]  
and \emph{initial} if and only if the right-hand square is exact.
\end{defn}

\begin{prop} In a cartesian closed $\infty$-cosmos, if $k \colon A \to B$ is initial and $f \colon B \to C$ is any diagram, then a limit of $f$ also defines a limit of $fk \colon A \to C$. Conversely, if the limit of $fk \colon A \to C$ exists then so does the limit of $f$ and it is given by the same point $\ell \colon 1 \to C$. 
\end{prop}
\begin{proof}
By Proposition~\ref{prop:limits-as-pointwise-kan}, a limit of $f$ is a pointwise right Kan extension.
\[ \xymatrix{ B \ar[dr]_f \ar[rr]^{!} & \ar@{}[d]|(.3){\nu}|{\Leftarrow} & 1 \ar[dl]^\ell \\ & C}\] If $k$ is final, then by  \ref{prop:pointwise-kan}\ref{itm:pointwise-exact-stable},
\[ \xymatrix{A \ar[d]_k \ar[rr]^{!} &  & 1  \ar@{=}[d] \\ B \ar[dr]_f \ar[rr]^{!} & \ar@{}[d]|(.3){\nu}|{\Leftarrow} & 1 \ar[dl]^\ell \\ & C}\] 
is again a pointwise right Kan extension, which defines a limit of $fk$ by  Proposition~\ref{prop:limits-as-pointwise-kan}.

For the converse, suppose we are given a pointwise right Kan extension diagram \[ \xymatrix{ A \ar[dr]_{fk} \ar[rr]^{!} & \ar@{}[d]|(.3){\nu}|{\Leftarrow} & 1 \ar[dl]^\ell \\ & C}\] in $\lcat{K}_2$, which Proposition~\ref{prop:pointwise-kan} tells us defines a right extension between covariant represented modules in $\MMod{K}$. This universal property tells us that for any composable sequence of modules $E_1, \ldots, E_n$ from $1$ to $C$, composing with $\nu \colon 1\comma ! \times C \comma \ell \To C \comma fk$ defines a bijection between cells $E_1 \times \cdots \times E_n \To C \comma \ell$ and cells
\[ 
\xymatrix{ A \ar@{=}[d] \ar[r]|\mid^{1 \comma !} & 1 \ar[r]|\mid^{E_1} \ar@{}[dr]|{\Downarrow} & \cdots \ar[r]|\mid^{E_n} & C \ar@{=}[d] \\ A \ar[rrr]|\mid_{C \comma fk}  && & C}
\] By Corollary~\ref{cor:companion-conjoint}, composing with $\nu \colon 1\comma ! \times C \comma \ell \To C \comma fk$ also defines a bijection between cells  $E_1 \times \cdots \times E_n \To C \comma \ell$ and cells 
\[ 
\xymatrix{ B \ar@{=}[d] \ar[r]|\mid^{k \comma A} & A  \ar[r]|\mid^{1 \comma !} & 1 \ar[r]|\mid^{E_1} \ar@{}[d]|{\Downarrow} & \cdots \ar[r]|\mid^{E_n} & C \ar@{=}[d] \\ B \ar[rrrr]|\mid_{C \comma f} &   && & C}
\]  As $k \colon A \to B$ is initial, the induced cell $k \comma A \times 1 \comma ! \To 1 \comma !$ of modules from $B$ to $1$ is a composite. Thus, composing with $\nu \colon 1\comma ! \times C \comma \ell \To C \comma fk$ also defines a bijection between cells  $E_1 \times \cdots \times E_n \To C \comma \ell$ and cells
\[ 
\xymatrix{ B \ar@{=}[d]  \ar[r]|\mid^{1 \comma !} & 1 \ar[r]|\mid^{E_1} \ar@{}[dr]|{\Downarrow} & \cdots \ar[r]|\mid^{E_n} & C \ar@{=}[d] \\ B \ar[rrr]|\mid_{C \comma f}    && & C}
\] But this says exactly that the cell $1 \comma ! \times C \comma \ell \To C \comma f$ that corresponds to $\nu$ under this series of bijections displays $C \comma \ell \colon 1 \prof C$ as a right extension of $C \comma f \colon B \prof C$ along $1 \comma ! \colon B \prof 1$. By Proposition~\ref{prop:pointwise-kan} and Proposition~\ref{prop:limits-as-pointwise-kan} we conclude that $\ell$ also defines the limit of $f \colon B \to C$, as claimed.
\end{proof}

\begin{lem} If $f \colon B \to A$ admits a right adjoint $u \colon A \to B$, then $k$ is initial.
\end{lem}
\begin{proof}
The functor $f$ is initial if and only if the map $p_1 \colon f \comma A \To A$ of modules from $A$ to $1$ exhibits $A$ as the reflection into modules of the isofibration $(p_1,!) \colon f \comma A \tfib A \times 1$. If $f \dashv u$, we have $f \comma A \simeq B \comma u$ over $A$. Lemma~\refI{lem:technicalsliceadjunction} constructs a right adjoint right inverse to $p_1$ and the conclusion follows from  Lemma~\ref{lem:fibred-adjoint-reflection}.
\end{proof}

\begin{defn} In a cartesian closed $\infty$-cosmos, an $\infty$-category $E$ \emph{admits functorial pointwise right Kan extension} along a functor $k \colon A \to B$ if there is a pointwise right Kan extension 
\[  \xymatrix{A \times E^A \ar[dr]_{\ev} \ar[rr]^{k \times E^A} & \ar@{}[d]|(.3){\nu }|{\Leftarrow} & B \times E^A \ar[dl]^{\ran_k(-)} \\ & E}\]
of the evaluation functor along $k \times {E^A}$.
\end{defn}

\begin{prop}[Beck-Chevalley condition]\label{prop:beck-chevalley} For any comma square 
\[ \xymatrix{ D \ar[r]^h \ar[d]_k  \ar@{}[dr]|{\Leftarrow\lambda} & B \ar[d]^f   \\ C \ar[r]_g & A }\] in a cartesian closed $\infty$-cosmos and any object $E$, the Beck-Chevalley condition is satisfied for the induced 2-cell
\[ \xymatrix{ E^A \ar[r]^{f^*} \ar[d]_{g^*} \ar@{}[dr]|{\Downarrow\lambda^*} & E^B \ar[d]^{h^*} \\ E^C \ar[r]_{k^*} & E^D}\] whenever functorial pointwise left or right Kan extensions of these functors exist: that is, the mates of $\lambda^*$ are isomorphisms.
\[ \xymatrix{ E^A \ar[r]^{f^*} \ar@{}[dr]|{\Downarrow \lambda^!} & E^B & & E^A \ar[d]_{g^*} \ar@{}[dr]|{\Uparrow \lambda_!} & E^B \ar[d]^{h^*} \ar[l]_{\lan_f}   \\ E^C \ar[r]_{k^*} \ar[u]^{\ran_g} & E^D \ar[u]_{\ran_h} & & E^C  & E^D \ar[l]^{\lan_k}}\]
\end{prop}
\begin{proof}
By Proposition~\ref{prop:ran-abs-lifting}, the pointwise right Kan extensions define absolute right lifting diagrams
\[ \xymatrix{  \ar@{}[dr]|(.7){\Downarrow\epsilon} & E^A \ar[r]^{f^*} \ar[d]^{g^*} \ar@{}[dr]|{\Downarrow\lambda^*} & E^B \ar[d]^{h^*} & &\ar@{}[dr]|(.7){\Downarrow\epsilon}  & E^B \ar[d]^{h^*} \\  E^C \ar@{=}[r] \ar[ur]^{\ran_g}  & E^C \ar[r]_{k^*} & E^D & E^C \ar[r]_{k^*} & E^D \ar@{=}[r] \ar[ur]^{\ran_h} & E^D}\] and moreover the mate $\lambda^!$ of $\lambda_*$ defines a factorization of the left-hand diagram through the right-hand diagram
\[ \xymatrix{  \ar@{}[dr]|(.7){\Downarrow\epsilon} & E^A \ar[r]^{f^*} \ar[d]^{g^*} \ar@{}[dr]|{\Downarrow\lambda^*} & E^B \ar[d]^{h^*}  \ar@{}[dr]|{\displaystyle =} & & E^A \ar[r]^{f^*} \ar@{}[d]|{\lambda^!\Downarrow} \ar@{}[dr]|(.7){\Downarrow\epsilon}  & E^B \ar[d]^{h^*} \\  E^C \ar@{=}[r] \ar[ur]^{\ran_g}  & E^C \ar[r]_{k^*} & E^D & E^C \ar[r]_{k^*} \ar[ur]^{\ran_g}  & E^D \ar@{=}[r] \ar[ur]|(.6){\ran_h} & E^D}\] Immediately from the universal property of the absolute liftings of $k^*$ along $h^*$ we have that $\lambda^!$ is an isomorphism.
\end{proof}

\begin{rmk}[on derivators for (co)complete quasi-categories]\label{rmk:qcat-derivator}
Derivators were introduced independently by Heller \cite{Heller:1988ly} and by Grothendieck in \emph{Pursuing Stacks}. A \emph{derivator} is a 2-functor $D \colon \Cat_2\op \to \category{CAT}_2$ from the 2-category of small 2-categories, thought of as indexing shapes for diagrams, to the 2-category of large categories, satisfying the following axioms:
\begin{enumerate}[label = (Der\arabic*)]
\item $D$ carries coproducts to products.
\item For each $\scat{A} \in \Cat_2$, the functor $D(\scat{A}) \to \prod_{a \in \scat{A}} D(\catone)$ induced by the family of functors $a \colon \catone \to \scat{A}$ is conservative.
\item For every functor $k \colon \scat{A} \to \scat{B} \in \Cat_2$, its image $k^* \colon D(\scat{B}) \to D(\scat{A})$ admits a left adjoint $\lan_k \colon D(\scat{A}) \to D(\scat{B})$ and a right adjoint $\ran_k \colon D(\scat{A}) \to D(\scat{B})$.
\item For every comma square in $\Cat_2$, the Beck-Chevalley condition is satisfied: that is the mates of the induced 2-cell in the image of $D$ are isomorphisms.
\item For each $\scat{A} \in \Cat_2$, the induced functor  $D(\scat{A} \times \cattwo) \to D(\scat{A})^\cattwo$ is essentially surjective and full.
\end{enumerate}

Under the embedding $\Cat_2\inc\qCat_2$ categorical indexing shapes can be regarded as special cases of quasi-categorical indexing shapes.  Thus, for any large quasi-category $E$, we have a 2-functor 
\begin{equation}\label{eq:qcat-derivator}
\xymatrix{\Cat_2\op \ar[r]^-{E^{-}} & \category{qCAT}_2 \ar[r]^-{\ho} & \category{CAT}_2,}\end{equation} which sends a category $\scat{A}$ to the homotopy category of the large quasi-category of $\scat{A}$-indexed diagrams, valued in $E$.

Suppose $E$ admits functorial pointwise Kan extensions for all functors $k \colon \scat{A} \to \scat{B}$. By Proposition~\ref{prop:ran-abs-lifting}, these define adjoints $\lan_k \dashv k^* \dashv \ran_k$ to the induced functor $k^* \colon E^{\scat{B}} \to E^{\scat{A}}$,  which define adjunctions between homotopy categories. This proves (Der3). The embedding $\Cat_2\inc\qCat_2$ carries comma squares to comma squares. By Proposition~\ref{prop:beck-chevalley}, the Beck-Chevalley 2-cells are isomorphisms in $\category{qCAT}_2$ and hence also in $\category{CAT}_2$, proving (Der 4). Axiom (Der 1) follows from the fact that exponentiation converts coproducts in the domain to products,  $E^{\coprod_i \scat{A}_i} \cong \prod_i E^{\scat{A}_i}$, and the homotopy category functor $\ho \colon \category{qCAT}_2 \to \category{CAT}_2$ preserves small products. (Der 2) and (Der 5) were proven as Corollary~\refI{cor:pointwise-equiv} and Proposition~\refI{prop:weak-cotensors}. Indeed, in the arguments just given, $\category{qCAT}_2$ can be replaced by the homotopy 2-category of any cartesian closed $\infty$-cosmos admitting a comma-preserving 2-functor $\Cat_2 \to \lcat{K}_2$. In the general case, the 2-functor
\[ h\defeq \hom(1,-) \colon \lcat{K}_2 \to \category{CAT}_2\] maps a (large) $\infty$-category $E$ to its homotopy category $\hom(1,E)$.

In the special case of quasi-categories, we can argue further that any complete and cocomplete quasi-category $E$ admits functorial pointwise Kan extensions along all functors $k \colon \scat{A} \to \scat{B}$, thus defining a derivator \eqref{eq:qcat-derivator}. A complete and cocomplete quasi-category is a large quasi-category admitting limits and colimits of all diagrams indexed by small simplicial sets. We outline the argument here, defering full details to a future paper that will focus on the quasi-categorical case.

The first step is to show that in $\MMod{\category{qCAT}}$ right extensions always exist. Consider a module $K \colon A \prof B$ represented by an isofibration $(q,p) \colon K \tfib A \times B$. The operation of horizontal composition with this isofibration can be represented as a composite simplicial functor
\[ \xymatrix{ \category{qCAT}\slice{B \times C} \ar[r]^{p^*} & \category{qCAT}\slice{K \times C} \ar[r]^{q \circ -} & \category{qCAT}\slice{A \times C}}\] formed by first pulling back along $p \times \id_C$ and then composing with $q \times \id_C$. The latter functor has a right adjoint, pullback along $q \times \id_C$, which is a functor of $\infty$-cosmoi. Because $p$ is a cartesian fibration, it is \emph{homotopy exponentiable}, i.e., $p^* \colon  \category{qCAT}\slice{B \times C} \to \category{qCAT}\slice{K \times C}$ also admits a right adjoint $\Pi_p$, defining a functor of $\infty$-cosmoi; see \cite[\S B.3]{Lurie:2012uq}, where homotopy exponentiable maps are called \emph{flat fibrations}, for a discussion. Now, given a module $F \colon A \prof C$, the component at $F$ of the counit of the composite adjunction defines a right extension diagram
\[ \xymatrix{ A \ar[dr]|\mid_F \ar[rr]|\mid^K & \ar@{}[d]|(.3){\epsilon }|{\Leftarrow} & B \ar[dl]|\mid^{\Pi_p (q^*F)} \\ & C}\]
in  $\MMod{\category{qCAT}}$.

In particular, given the quasi-categories and functors displayed on the left, where $A$ and $B$ are small, there is some module $G$ that defines the right extension displayed on the right.
\begin{equation}\label{eq:mod-right-ext}  \xymatrix{A \times E^A \ar[dr]_{\ev} \ar[rr]^{k \times E^A} &  & B \times E^A & & A \times E^A \ar[dr]|\mid_{E \comma \ev} \ar[rr]|\mid^{(B \times E^A)\comma (k \times E^A)} & \ar@{}[d]|(.3){\nu }|{\Leftarrow} & B \times E^A \ar[dl]|\mid^{G} \\ & E & & &&  E}\end{equation} Now the quasi-category $E$ will admit functorial pointwise right Kan extensions along $k \colon A \to B$, just when the module $G$ in \eqref{eq:mod-right-ext} is covariantly represented. By Corollary~\ref{cor:fibrewise-rep}, the module $G \colon B \times E^A \prof E$ is covariantly represented if and only if its pullbacks along each vertex $(b,f) \in B \times E^A$ are covariantly represented.

Now the proof of Proposition~\ref{prop:pointwise-kan} \ref{itm:pointwise-in-mod}$\Rightarrow$\ref{itm:pointwise-exact-stable-in-mod}  tells us that the right extension diagram \eqref{eq:mod-right-ext} is stable under pasting with the images of exact squares in $\category{qCAT}_2$ under the covariant embedding $\category{qCAT}_2\inc\MMod{\category{qCAT}}$. Thus, by Lemma~\ref{lem:comma-exact}, Lemma~\ref{lem:product-exact-square}, and Lemma~\ref{lem:exact-squares-comp} we have a right extension diagram
\[ \xymatrix{ b \comma k \ar@{}[drr]|{\Leftarrow\phi} \ar[d]_{A \comma p_1} \ar[rr]^{\Del^0 \comma p_0} & & \Del^0 \ar[d]^{B \comma b} \\ 
A \ar[rr]^{B \comma k} \ar[d]_{(A \times E^A) \comma (A \times f)} & & B \ar[d]^{(B \times E^A) \comma (B \times f)}
\\ A \times E^A \ar[dr]|\mid_{E \comma \ev} \ar[rr]|\mid^{(B \times E^A)\comma (k \times E^A)} & \ar@{}[d]|(.3){\nu }|{\Leftarrow} & B \times E^A \ar[dl]|\mid^{G} \\ &  E}\]

Now if $A$ and $B$ are small then so is $b \comma k$, and hence if $E$ is complete, Proposition~\ref{prop:limits-as-pointwise-kan} tells us that there is a pointwise right Kan extension diagram
\[ \xymatrix{ b \comma k \ar[d]_{p_1} \ar[rr]^{!} \ar@{}[ddr]^{\Leftarrow\nu} & & \Del^0 \ar[ddl]^{\ell} \\ A \ar[dr]_f & & \\ & E & }\] given by forming the limit $\ell \in E$ of $p_1 f \colon b \comma k \to E$. Thus the fiber of the module $G$ over the point $(b,f) \colon \Del^0 \to B \times E^A$ is equivalent to $E \comma \ell$. As argued above, this implies that $G$ is represented, which implies that functorial pointwise right Kan extensions exist for any complete quasi-category $E$.
\end{rmk}

The result described in Remark \ref{rmk:qcat-derivator}, which will be proven in full in a forthcoming paper on quasi-categories~\cite{RiehlVerity:2016ys}, also provides convenient motivation for a second paper in progress.  Specific details of the quasi-categorical model of $(\infty,1)$-categories were used in two places in the argument just given: \begin{enumerate}[label=(\roman*)]
\item Particular features of the Joyal model structure on simplicial sets are used to prove that cartesian fibrations are homotopy exponentiable.
\item An inductive argument over dimensions of simplices is used to prove Lemma \ref{lem:qcat-pointwise-rep}, which is applied in  the proof of Corollary \ref{cor:fibrewise-rep} to characterize represented modules between quasi-categories.
\end{enumerate}
But there is no reason why the conclusion, that a complete and cocomplete quasi-category defines a derivator, should be restricted to this model of $(\infty,1)$-categories, and indeed a forthcoming paper on model independence of $\infty$-category theory~\cite{RiehlVerity:2016mi} will prove this.

The main idea is quite simple to describe. Certain functors between $\infty$-cosmoi define what we call \emph{weak equivalences of $\infty$-cosmoi}: functors that are surjective on objects up to equivalence and define equivalences of mapping quasi-categories. Each of the functors listed in Example \ref{ex:functors} between the $\infty$-cosmoi of quasi-categories, complete Segal spaces, Segal categories, and naturally marked simplicial sets is an weak equivalence of $\infty$-cosmoi. Each weak equivalence of $\infty$-cosmoi induces what we call a \emph{biequivalence of virtual equipments}. Informally, a biequivalence of virtual equipments preserves, reflects, and creates all equivalence-invariant features of the virtual equipment, e.g., whether a module is represented by a functor.

In particular, since right extensions always exist in $\MMod{\category{qCAT}}$, this is also true in any biequivalent virtual equipment. Furthermore, since a module between quasi-categories is covariantly represented if and only if its pullbacks to a module whose domain is the terminal object is covariantly represented, the same result holds for modules between $\infty$-categories in any biequivalent virtual equipment. In this way, we will conclude that any complete and cocomplete complete Segal space or Segal category also defines a derivator.


  \bibliographystyle{gtart}
  \bibliography{../../common/index}

\end{document}